%% file: main.tex
\documentclass[a4paper,english,11pt]{article}
\input{preamble}

\title{Routing in Line Networks with Handling Times}
\author[1]{Gabriel Deza\thanks{\href{mailto:gabrieldeza@mail.tau.ac.il}{gabrieldeza@mail.tau.ac.il}}}
\author[1]{Michal Tzur}
\author[1]{Tal Raviv}
\affil[1]{School of Industrial and Intelligent Systems Engineering\protect\newline Tel Aviv University, Tel Aviv, Israel}
\date{October 1, 2026}
\begin{document}
\maketitle
\begin{abstract}
\input{sections/abstract}
\end{abstract}

\input{sections/01-introduction}
\input{sections/02-literature-review}
\input{sections/03-problem-description}
\input{sections/04-primal-dual-certificates}
\input{sections/05-convex-frontier-bound}
\input{sections/06-computational-results}
\input{sections/07-conclusion}

\section*{Acknowledgments}
This research was supported by an Israel Science Foundation (ISF) grant no. 1460/20.

\bibliographystyle{plainnat}
\bibliography{cit}

\clearpage
\appendix
\numberwithin{equation}{section}
\numberwithin{figure}{section}
\numberwithin{table}{section}
\input{appendices}
\end{document}

%% file: preamble.tex
\usepackage[T1]{fontenc}
\usepackage[utf8]{inputenc}
\usepackage[english]{babel}
\usepackage{lmodern}
\usepackage{amsmath,amssymb,amsfonts,amsthm,mathtools,bm}
\usepackage[a4paper,left=0.85in,right=0.85in,top=0.75in,bottom=0.85in]{geometry}
\usepackage{setspace}
\usepackage{authblk}
\usepackage{microtype}
\usepackage{algorithm,algpseudocode}
\usepackage{graphicx,float,adjustbox}
\usepackage{booktabs,array,multirow,enumitem}
\usepackage{caption,subcaption}
\usepackage{xcolor}
\usepackage{tikz,pgfplots,csvsimple,siunitx,csquotes}
\usetikzlibrary{arrows.meta,positioning,calc,automata,shapes.geometric}
\pgfplotsset{compat=1.18}
\usepackage{natbib}
\bibpunct[, ]{(}{)}{,}{a}{}{,}

\newtheorem{theorem}{Theorem}
\newtheorem{proposition}{Proposition}
\theoremstyle{remark}
\newtheorem{remark}{Remark}
\usepackage[hyphens]{url}
\usepackage[colorlinks=true,linkcolor=blue!45!black,citecolor=green!35!black,urlcolor=blue!55!black]{hyperref}
\hypersetup{pdftitle={Routing in Line Networks with Handling Times},pdfauthor={Gabriel Deza, Michal Tzur, Tal Raviv}}
\newcommand{\ptnimg}[1]{%
  \adjustbox{max width=\linewidth, max height=6cm, center}{%
    \includegraphics{#1}%
  }%
}

\newcommand{\Xset}{\mathcal{X}}

\definecolor{ArcOrig}{RGB}{228,26,28}      %
\definecolor{ArcNoDelta}{RGB}{55,126,184}    %
\definecolor{socp}{RGB}{77,175,74}    %
\definecolor{under}{RGB}{152,78,163} %
\colorlet{colorArcOrig}{ArcOrig}
\colorlet{colorArcNoDelta}{ArcNoDelta}
\colorlet{colorsocp}{socp}
\colorlet{colorunder}{under}

\tikzset{
  lbOrig/.style={dash pattern=on 7pt off 3pt},                    %
  lbNoDelta/.style={dash pattern=on 4pt off 2pt},                    %
  lbBase/.style={dash pattern=on 2pt off 2pt},                    %
  lbSOCP/.style={dash pattern=on 6pt off 2pt on 1pt off 2pt},     %
  lbUnder/.style={dash pattern=on 6pt off 2pt on 1pt off 2pt on 1pt off 2pt}, %
}

\newcommand{\CSVGraphPanel}[3]{%
  \def\S{1.5} \def\pad{0.35} \def\head{0.4}%
  \begin{scope}
    \draw[black] ({-\pad},{-\pad}) rectangle ({\S+\pad},{\S+\pad});
    \fill[gray!35] ({-\pad},{\S+\pad}) rectangle ({\S+\pad},{\S+\pad+\head});
    \draw[black] ({-\pad},{\S+\pad}) rectangle ({\S+\pad},{\S+\pad+\head});
    \node[font=\bfseries] at ({\S/2},{\S+\pad+0.18}) {#3};
    \csvreader[head to column names]{#1}{}{%
      \pgfmathsetmacro{\X}{(\x/60.0)*\S}%
      \pgfmathsetmacro{\Y}{(\y/60.0)*\S}%
      \coordinate (n\id) at (\X,\Y);}
    \csvreader[head to column names]{#2}{}{%
      \draw[thick] (n\source) -- (n\target);}
    \csvreader[head to column names]{#1}{}{%
      \fill (n\id) circle[radius=2pt];
      \node[font=\scriptsize, xshift=5pt, yshift=5pt] at (n\id) {\id};
    }
  \end{scope}%
}
\newif\ifcsvplotflipy

\pgfkeys{
  /csvplot/.is family, /csvplot,
  width/.initial=6,        %
  xmax/.initial=60,        %
  ymax/.initial=60,        %
  xticks/.initial=5,       %
  yticks/.initial=5,       %
  flipy/.is if=csvplotflipy, flipy/.default=true
}
\newcommand{\PlotCSVGraphWithAxes}[4][]{%
  \pgfkeys{/csvplot, #1}%
  \pgfmathsetmacro{\W}{\pgfkeysvalueof{/csvplot/width}}
  \pgfmathsetmacro{\XMAX}{\pgfkeysvalueof{/csvplot/xmax}}
  \pgfmathsetmacro{\YMAX}{\pgfkeysvalueof{/csvplot/ymax}}
  \pgfmathsetmacro{\NX}{\pgfkeysvalueof{/csvplot/xticks}}
  \pgfmathsetmacro{\NY}{\pgfkeysvalueof{/csvplot/yticks}}
  \pgfmathsetmacro{\H}{\W * (\YMAX/\XMAX)}
  \pgfmathsetmacro{\sx}{\W/\XMAX}
  \pgfmathsetmacro{\sy}{\H/\YMAX}

  \begin{scope}
    \draw[black] (0,0) rectangle (\W,\H);

    \draw[black,-] (0,0) -- (\W,0) node[below] {};
    \draw[black,-] (0,0) -- (0,\H) node[left]  {};

    \pgfmathsetmacro{\dx}{\XMAX/\NX}
    \foreach \i in {0,...,\NX}{
      \pgfmathsetmacro{\xdata}{\i*\dx}
      \pgfmathsetmacro{\xplot}{\xdata*\sx}
      \draw[black] (\xplot,0) -- ++(0,-0.08);
      \node[below] at (\xplot,-0.08) {\pgfmathprintnumber{\xdata}};
    }

    \pgfmathsetmacro{\dy}{\YMAX/\NY}
    \foreach \j in {0,...,\NY}{
      \pgfmathsetmacro{\ydata}{\j*\dy}
      \pgfmathsetmacro{\yplotraw}{\ydata*\sy}
      \pgfmathsetmacro{\yplot}{\ifcsvplotflipy \H-\yplotraw \else \yplotraw \fi}
      \draw[black] (0,\yplot) -- ++(-0.08,0);
      \node[left] at (-0.08,\yplot) {\pgfmathprintnumber{\ydata}};
    }

    \ifx&#4&\else
      \def\head{0.4}
      \fill[gray!35] (0,\H) rectangle (\W,\H+\head);
      \draw[black] (0,\H) rectangle (\W,\H+\head);
      \node[font=\bfseries] at (\W/2,\H+\head/2) {#4};
    \fi

    \csvreader[head to column names]{#2}{}{%
      \pgfmathsetmacro{\X}{\x * \sx}%
      \pgfmathsetmacro{\Yraw}{\y * \sy}%
      \pgfmathsetmacro{\Y}{\ifcsvplotflipy \H - \Yraw \else \Yraw \fi}%
      \coordinate (n\id) at (\X,\Y);
    }%

    \csvreader[head to column names]{#3}{}{%
      \draw[thick] (n\source) -- (n\target);
    }%

    \csvreader[head to column names]{#2}{}{%
      \fill (n\id) circle[radius=2.5pt];
      \node[xshift=7pt, yshift=7pt] at (n\id) {\id};
    }%
  \end{scope}
}

\pgfdeclareplotmark{mystar}{
  \node[
    star, star point ratio=2.55,
    minimum size=7pt, inner sep=0pt,
    draw=black, line width=0.6pt,
    fill=black
  ] {};
}

\pgfdeclareplotmark{mystarblue}{
  \node[
    star, star point ratio=2.55,
    minimum size=7pt, inner sep=0pt,
    draw=black, line width=0.6pt,
    fill=black
  ] {};
}

%% file: sections/abstract.tex
Motivated by hyperconnected urban parcel networks for same-day and time-sensitive delivery, we study the fixed-line routing problem that arises when evaluating a candidate set of recurring inter-hub vehicle services. Vehicles follow predefined cyclic sequences of service points, or lines, and parcels may transfer between lines. Such transfers expand path choice but require loading and unloading operations that make vehicles dwell at service points. Because handled volumes are routing-dependent, dwell times become endogenous, linking routing decisions to line cycle times, effective service frequencies, and transfer waiting times.

We introduce the Routing in Line Networks with Handling Times (RLN-HT) problem and formulate it as a bilinear generalized capacity-flow assignment model. We characterize optimal dwell times and the resulting frequencies as functions of handling workloads, enabling a reformulation solely in terms of commodity flows. We exploit its convex substructure to derive two second-order conic relaxations that preserve the original feasible region and therefore yield feasible incumbent solutions along with valid lower bounds. We also develop a residual-frontier framework that characterizes the tradeoff between the convex relaxation objective and the remaining nonconvex residual, yielding a general mechanism for constructing lower-bound certificates. Computational experiments show that ignoring handling times can yield poor or infeasible routing, whereas accounting for endogenous dwell times can improve average parcel travel times by up to 26.5\%. In instances with up to 257 nodes and 25 lines, the relaxations solve quickly, yield near-optimal routing, and obtain small optimality gaps, whereas off-the-shelf spatial branch-and-bound struggles to produce useful bounds.

%% file: sections/01-introduction.tex
\section{Introduction}
\label{sec:Introduction}
Integrated logistics service providers have traditionally organized parcel and freight movements through multi-echelon hierarchical hub networks. Parcels are typically routed upward through successive echelons of consolidation hubs and then downward through successive echelons of distribution hubs before reaching the recipient. Transportation between adjacent echelons is typically performed by multi-trip vehicle routes with backhauls, such as trucks moving parcels between a regional gateway hub and local hubs, or couriers serving customers from a local hub for first- and last-mile pickup and delivery. Fig.~\ref{fig:intro_hierarchy} illustrates this structure for a two-echelon system, where the filled and hollow squares denote a parcel's origin and destination, respectively, and the highlighted paths shows the route taken.

These hierarchical networks achieve consolidation at scale by limiting path choice and routing parcels upward and downward through the hierarchy, so many flows share long intermediate legs. However, this vertical routing can be too restrictive for time-sensitive intracity flows, where sender and recipient are in the same city or service region and routing through a higher-level hub adds unnecessary diversion. A natural alternative is to allow more lateral movement among hubs within the same echelon. For instance, SF Express concluded that transshipment through regional gateway hubs would not achieve the company’s six-hour delivery goal and instead operate a same-day network among local hubs \citep{huang2026sfexpress}. In this paper, we propose a framework for high-volume urban parcel delivery services that leverage line networks to address challenges companies like SF face, as well as the practical considerations of handling times. More broadly, this shift to lateral movements among hubs is aligned with Physical Internet and hyperconnected city-logistics proposals, in which parcels move through meshed inter-hub networks rather than a purely hierarchical hub-and-spoke structure \citep{montreuil2018urbanparcel}.

\input{graphics/intro_network_concepts.tex}

In networks with lateral movements, each hub may simultaneously generate, receive, and relay parcels, creating a many-to-many operating environment in which the appropriate vehicle-routing structure is less obvious. An extreme is to operate a fully flexible pickup-and-delivery system, where vehicles are dynamically assigned and routed as requests arise, as illustrated in Fig~\ref{fig:intro_flexible}. While flexible, this approach shifts complexity to real-time dispatch and can produce volatile utilization, unstable schedules, and unreliable handoffs because operations follow demand too closely. A more structured alternative is to operate repeatable primary routes, forming what is often called a \emph{line network} \citep{ido_tal}. In such a network, vehicles repeatedly operate \emph{lines}, each defined as a sequence of hubs where parcels may be loaded, unloaded, or transferred throughout the day. Temporary storage at hubs allows parcels to move through the network over multiple legs. This structure improves operational stability by fixing vehicle movements in advance, building route familiarity for drivers, and reducing the need for direct routes between every pair of hubs. Fig~\ref{fig:intro_lines} illustrates this line network structure.

Designing such a network requires selecting these vehicle lines, and often their service frequencies, making the problem closely related to line-planning in public transit and liner shipping. As discussed in Section~\ref{sec:lit_review}, jointly optimizing vehicle lines and demand routes is difficult. Most approaches, therefore, evaluate many candidate line networks, where each fixed line network requires solving a \emph{parcel routing} or \emph{assignment subproblem}, the focus of this work.

For parcel delivery, a key challenge compared with many standard assignment and routing problems in public transit and liner shipping is that \emph{handling times} at hubs cannot be treated as negligible. In public transit, boarding and alighting times are often treated as fixed or absorbed into schedules. In liner shipping, port times are largely governed by terminal and berth operations and are often handled outside the cargo-assignment model. By contrast, the lower-echelon urban hubs that motivate our setting are smaller and less optimized for loading, unloading, and sorting than higher-echelon consolidation hubs, especially in dense urban environments. As a result, vehicle dwell times are not simply positive constant slacks for handling, but functions of the total handled flow per hub visit.

Importantly, vehicle dwell times depend on the routing solution through the number of parcels loaded and unloaded at each hub, while the routing solution itself depends on dwell times through parcel travel times. Longer dwell times provide more handling capacity, but they also slow the line, delay parcels already on board, and increase waiting times for parcels that need to board or transfer. While one could fix dwell times a priori, such a heuristic will likely yield suboptimal, potentially infeasible parcel routing, since appropriate dwell values depend on the routing-induced handling volumes. Here, infeasibility means that the routing exceeds the available handling or effective vehicle capacity over the planning period. A routing subproblem should therefore jointly optimize dwell times and parcel routes, as illustrated in the following example.

To demonstrate these relationships, consider the network in Figure~\ref{fig:toy_network}. It contains two \(A \rightarrow B \rightarrow A\) lines, shown as a solid red line and a dashed blue line, with base cycle times of \(2\) and \(3\) time units, respectively. There are \(\lambda=140\) parcels traveling from \(A\) to \(B\) over a period of length \(T=60\), which can be split between the two lines. Each line has one vehicle, capacity is nonbinding, and loading and unloading times are equal, with \(\alpha=\beta=0.1\), respectively. Let \(s_{\text{red}}\) denote the fraction routed on the red line, so that \(1-s_{\text{red}}\) is routed on the blue line. Let \(\Delta_i^l\) denote the dwell time of line \(l\) at hub \(i\) per visit. The effective cycle times are therefore \(2+\Delta_A^{\mathrm{red}}+\Delta_B^{\mathrm{red}}\) for the red line and \(3+\Delta_A^{\mathrm{blue}}+\Delta_B^{\mathrm{blue}}\) for the blue line. In this illustrative example, dwell times are computed from the temporal-capacity relation used later in the model. Over the period, the total handling workload assigned to a line must equal its dwell time per visit multiplied by the number of visits.

A parcel routed on the red line from $A$ to $B$ has three travel-time components. It first waits to board at hub \(A\), with expected waiting time \(\frac{2+\Delta_A^{\mathrm{red}}+\Delta_B^{\mathrm{red}}}{2}\). It then spends one time unit traveling to hub \(B\). Finally, it incurs unloading time at \(B\), represented conservatively by \(\Delta_B^{\mathrm{red}}\). Thus, the cost of routing a parcel on the red line depends on \(\Delta_A^{\mathrm{red}}\) and \(\Delta_B^{\mathrm{red}}\), but these dwell times are themselves determined by the routing decision \(s_{\text{red}}\). In this case, the red and blue lines carry \(140s_{\text{red}}\) and \(140(1-s_{\text{red}})\) parcels, respectively.

\begin{figure}[h]
    \centering
    \begin{subfigure}[b]{0.4\textwidth}
        \centering
        \raisebox{16mm}[0pt][0pt]{%
              \begin{tikzpicture}[>=Stealth, node distance=2.5cm, thick]
                \node[circle, draw=black, thick, minimum size=0.5cm] (A) {\large $A$};
                \node[circle, draw=black, thick, minimum size=0.5cm, right of=A] (B) {\large $B$};
                \draw[->, red, line width=1.4pt] (A) to[bend left=0] node[pos=0.5, above=0pt] {$1$} (B);
                \draw[->, red, line width=1.4pt] (B) to[bend left=0] node[pos=0.5, below=0pt] {} (A);
                \draw[->, dashed,blue, line width=1.4pt] (A) to[bend left=40] node[pos=0.5, above, blue] {$1.5$} (B);
                \draw[->, dashed,blue, line width=1.4pt] (B) to[bend left=-40] node[pos=0.5, below, blue] {} (A);
              \end{tikzpicture}%
         }
        \caption{Two-line network.}
        \label{fig:toy_network}
    \end{subfigure}
    \hfill
    \begin{subfigure}[b]{0.55\textwidth}
        \centering
        \begin{tikzpicture}
          \begin{axis}[
            width=\linewidth,
            height=0.5\linewidth,
            xlabel={$s_{\text{red}}$},
            ylabel={aTT},
            grid=both,
            grid style={dashed, opacity=0.35},
            legend cell align=left,
            legend style={font=\scriptsize, at={(1,1)}, anchor=north east, fill=white, draw=black!20},
            unbounded coords=discard,
            ymin=0,
            ymax=6,
            ytick={0,1,2,3,4,5,6},
            minor y tick num=1,
            every axis plot/.append style={mark=none, line width=1.2pt},
          ]
            \addplot[black, thick]
              table[col sep=comma, x=s, y=actual_travel_time]{graphics/toy_example_CSVs/main_series.csv};
            \addlegendentry{$\alpha=\beta=0.1$}
            \addplot[black, thick, dashed]
              table[col sep=comma, x=s, y=expected_travel_time]{graphics/toy_example_CSVs/main_series.csv};
            \addlegendentry{$\alpha=\beta=0$}
            \addplot[only marks,mark=mystar] coordinates {(1,2)};
            \node[above, text=black] at (axis cs:1,2) {$s^{*}_{\text{red}} = 1$};
            \addplot[only marks, mark=mystarblue] coordinates {(0.6659623,3.1196918)};
            \node[below, text=black] at (axis cs:0.6659623,2.1196918) {$s^{*}_{\text{red}} = 0.67$};
          \end{axis}
        \end{tikzpicture}
        \caption{aTT as a function of $s_{\text{red}}$.}
        \label{fig:opt_policy}
    \end{subfigure}

    \caption{Toy two-line network and average travel time under different routing splits.}
    \label{fig:toy_combined}
\end{figure}

Figure~\ref{fig:opt_policy} shows the average travel time (aTT) as a function of \(s_{\text{red}}\) when handling is absent \((\alpha=\beta=0)\) and when handling is present \((\alpha=\beta=0.1)\), shown by the dashed and solid lines, respectively. If the routing policy ignores handling, it solves the problem with \(\alpha=\beta=0\) and sends all demand on the red line, since the red line has the shorter base cycle time. When this same routing is evaluated with handling present, the red line becomes overloaded. Additional loading and unloading increase \(\Delta_A^{\mathrm{red}}\) and \(\Delta_B^{\mathrm{red}}\), lengthen the effective cycle time, and reduce frequency. For the parameter values in the figure, the handling-aware objective is minimized by splitting demand, with \(s_{\text{red}}^{*}=0.67\), avoiding excessive dwell times on the seemingly shortest path.

As a result, our work focuses on the parcel-routing subproblem that arises after a candidate line network has been fixed. We develop an \emph{oracle} for this subproblem using a fluid approximation that evaluates the mean delivery time under system-optimal routing while accounting for handling-induced dwell times. The oracle returns high-quality solutions quickly on large instances and certifies their quality with small optimality gaps. To this end, our work makes the following main contributions:
\begin{itemize}
    \item We motivate and formalize line networks in which vehicles repeatedly traverse predefined cyclic routes as a middle-echelon operating concept for high-volume urban parcel delivery, sitting between hierarchical hub routing and fully flexible pickup-and-delivery while allowing lateral parcel movement among local hubs.
    \item We introduce the \emph{Routing in Line Networks with Handling Times} (RLN-HT) problem as a fixed-line parcel-routing oracle that captures handling-induced dwell times, effective service frequencies, and vehicle capacity.
    \item We formulate RLN-HT as a bilinear generalized capacity-and-flow assignment problem and use optimality conditions to obtain a pure multicommodity flow reformulation that is nonconvex and nonseparable, but exposes exploitable convex structure.
    \item We derive two second-order conic relaxations that preserve the routing feasible region, return feasible parcel routings, and provide fast lower-bound certificates.
    \item We develop a residual-frontier framework for convex-plus-residual routing problems and use it to obtain substantially stronger lower-bound certificates for RLN-HT.
    \item We show computationally that ignoring handling times can produce poor or infeasible routing solutions, while the proposed oracle scales to large instances, returns high-quality routings quickly, and certifies small optimality gaps.
\end{itemize}

The remainder of this paper is organized as follows. Section~\ref{sec:lit_review} reviews the related literature. Section~\ref{sec:methods} introduces RLN-HT and its pure-flow reformulation. Sections~\ref{sec:primal_dual_certificates} and~\ref{sec:convex_frontier_bound} develop primal and dual certificates. Section~\ref{sec:num_res} presents computational results, and Section~\ref{sec:conclusion} concludes.

%% file: graphics/intro_network_concepts.tex
\begin{figure}[t]
\centering
\begingroup
\definecolor{introGHFill}{RGB}{219,91,91}
\definecolor{introGHDraw}{RGB}{143,29,29}
\definecolor{introLHFill}{RGB}{142,207,136}
\definecolor{introLHDraw}{RGB}{35,106,58}
\definecolor{introGray}{RGB}{186,194,188}
\definecolor{introBlue}{RGB}{68,116,185}
\definecolor{introOrange}{RGB}{211,130,43}
\definecolor{introPurple}{RGB}{126,87,194}
\definecolor{introTeal}{RGB}{45,139,87}
\definecolor{introParcelFill}{RGB}{232,190,112}
\definecolor{introParcelDraw}{RGB}{145,96,39}
\definecolor{introParcelPath}{RGB}{122,36,28}

\tikzset{
  introPanel/.style={x=0.34cm,y=0.34cm,>=Stealth},
  introLH/.style={
    circle,
    draw=introLHDraw,
    fill=introLHFill,
    line width=0.45pt,
    minimum size=4.2pt,
    inner sep=0pt
  },
  introGH/.style={
    rectangle,
    draw=introGHDraw,
    fill=introGHFill,
    line width=0.55pt,
    minimum size=6.5pt,
    inner sep=0.8pt,
    font=\scriptsize\bfseries
  },
  introRoute/.style={line width=0.9pt, rounded corners=1pt},
  introRouteEnd/.style={
    -{Stealth[length=3.0pt,width=3.4pt]},
    line width=0.9pt,
    rounded corners=1pt,
    dashed,
    dash pattern=on 2.8pt off 1.6pt,
    shorten >=2.2pt
  },
  introFlex/.style={-{Stealth[length=4.2pt,width=4.8pt]}, densely dashed, line width=0.95pt},
  introDispatch/.style={
    -{Stealth[length=4.2pt,width=4.8pt]},
    densely dashed,
    line width=0.95pt,
    shorten >=2.2pt,
    shorten <=7.0pt
  },
  introDelivery/.style={
    -{Stealth[length=3.6pt,width=4.0pt]},
    line width=0.9pt,
    rounded corners=1.1pt,
    shorten >=2.2pt,
    shorten <=2.2pt
  },
  introMetro/.style={
    line width=1.55pt,
    line cap=round,
    line join=round,
    rounded corners=3.5pt
  },
  introLine/.style={
    line width=1.65pt,
    line cap=round,
    line join=round
  },
  introLineStation/.style={
    circle,
    draw=introLHDraw,
    fill=introLHFill,
    line width=0.45pt,
    minimum size=5.8pt,
    inner sep=0pt
  },
  introTruck/.pic={
    \draw[pic actions, rounded corners=0.5pt] (-3.1,-1.05) rectangle (0.8,1.05);
    \draw[pic actions, rounded corners=0.5pt] (0.8,-0.85) -- (2.55,-0.85) -- (3.05,-0.15) -- (2.15,1.05) -- (0.8,1.05) -- cycle;
    \draw[pic actions] (1.55,0.15) -- (2.35,0.15) -- (1.95,0.72) -- (1.25,0.72) -- cycle;
    \draw[pic actions, fill=white, line width=0.55pt] (-2.05,-1.15) circle (0.32);
    \draw[pic actions, fill=white, line width=0.55pt] (1.75,-1.15) circle (0.32);
  },
  introTruckPurpleCargo/.pic={
    \draw[pic actions, rounded corners=0.5pt] (-3.1,-1.05) rectangle (0.8,1.05);
    \draw[pic actions, rounded corners=0.5pt] (0.8,-0.85) -- (2.55,-0.85) -- (3.05,-0.15) -- (2.15,1.05) -- (0.8,1.05) -- cycle;
    \draw[pic actions] (1.55,0.15) -- (2.35,0.15) -- (1.95,0.72) -- (1.25,0.72) -- cycle;
    \filldraw[draw=introPurple!85!black, fill=introPurple!28, line width=0.35pt] (-2.35,-0.42) rectangle (-1.63,0.30);
    \filldraw[draw=introPurple!85!black, fill=introPurple!38, line width=0.35pt] (-1.63,-0.42) rectangle (-0.91,0.30);
    \filldraw[draw=introPurple!85!black, fill=introPurple!48, line width=0.35pt] (-1.99,0.30) rectangle (-1.27,1.02);
    \draw[pic actions, fill=white, line width=0.55pt] (-2.05,-1.15) circle (0.32);
    \draw[pic actions, fill=white, line width=0.55pt] (1.75,-1.15) circle (0.32);
  }
}

\newcommand{\IntroNetworkCoords}{%
  \coordinate (GH) at (5.000000,7.050000);
  \coordinate (LH1) at (1.040413,2.980979);
  \coordinate (LH2) at (5.519259,1.799496);
  \coordinate (LH3) at (1.737150,0.601582);
  \coordinate (LH4) at (0.806377,1.667238);
  \coordinate (LH5) at (7.535336,3.526217);
  \coordinate (LH6) at (9.248581,4.719448);
  \coordinate (LH7) at (0.849977,5.165978);
  \coordinate (LH8) at (8.905736,2.076669);
  \coordinate (LH9) at (7.294776,5.260988);
  \coordinate (LH10) at (6.306590,3.507102);
  \coordinate (LH11) at (2.986671,1.813660);
  \coordinate (LH12) at (4.306270,3.287407);
}

\newcommand{\IntroLocalHubs}{%
  \foreach \n in {1,...,12}{\node[introLH] at (LH\n) {};}
}

\newcommand{\IntroGateway}{%
  \node[introGH] at (GH) {};
}

\newcommand{\IntroBackgroundEdges}{%
  \foreach \a/\b in {LH1/LH4,LH1/LH7,LH1/LH11,LH1/LH12,LH2/LH3,LH2/LH5,LH2/LH8,LH2/LH10,LH2/LH11,LH2/LH12,LH3/LH4,LH3/LH11,LH4/LH7,LH4/LH11,LH5/LH6,LH5/LH8,LH5/LH9,LH5/LH10,LH6/LH8,LH6/LH9,LH7/LH12,LH9/LH10,LH9/LH12,LH10/LH12,LH11/LH12}{
    \draw[introGray!55, line width=0.35pt] (\a) -- (\b);
  }
}

\newcommand{\DrawDirectedSegment}[3]{%
  \draw[introRoute,#3] (#1) -- (#2);
  \draw[-{Stealth[length=4.2pt,width=4.8pt]}, line width=0.9pt,#3]
    ($(#1)!0.46!(#2)$) -- ($(#1)!0.58!(#2)$);
}

\newcommand{\DrawDirectedEndSegment}[3]{%
  \draw[introRouteEnd,#3] (#1) -- (#2);
}

\newcommand{\DrawRoute}[2]{%
  \foreach \a/\b in {#2}{\DrawDirectedSegment{\a}{\b}{#1}}
}

\newcommand{\DrawRouteEnd}[2]{%
  \foreach \a/\b in {#2}{\DrawDirectedEndSegment{\a}{\b}{#1}}
}

\newcommand{\IntroLineSegment}[4]{%
  \draw[introLine,#1] ($(#2)!#4!90:(#3)$) -- ($(#3)!#4!-90:(#2)$);
}

\newcommand{\IntroTruckOnSegment}[6]{%
  \path ($(#2)!#4!90:(#3)$) -- ($(#3)!#4!-90:(#2)$)
    pic[pos=#5,sloped,yshift=#6,scale=0.16,draw=#1,fill=#1!12] {introTruck};
}

\newcommand{\IntroLineStations}{%
  \foreach \s in {A,B,C,D,E,F,G,H,I,J,K,L}{\node[introLineStation] at (\s) {};}
}

\newcommand{\IntroParcelStackShift}[3]{%
  \begin{scope}[shift={($(#1)+(#2,#3)$)}]
    \filldraw[draw=introParcelDraw, fill=introParcelFill, line width=0.25pt] (0,0) rectangle (0.34,0.34);
    \filldraw[draw=introParcelDraw, fill=introParcelFill!80, line width=0.25pt] (0.34,0) rectangle (0.68,0.34);
    \filldraw[draw=introParcelDraw, fill=introParcelFill!65, line width=0.25pt] (0.17,0.34) rectangle (0.51,0.68);
  \end{scope}
}

\newcommand{\IntroParcelStack}[1]{%
  \IntroParcelStackShift{#1}{-0.78}{0.22}
}

\newcommand{\IntroParcelBoxShift}[4]{%
  \begin{scope}[shift={($(#1)+(#2,#3)$)}]
    \filldraw[draw=#4!80!black, fill=#4!42, line width=0.25pt] (0,0) rectangle (0.30,0.30);
  \end{scope}
}

\newcommand{\IntroEmptyParcelBoxShift}[4]{%
  \begin{scope}[shift={($(#1)+(#2,#3)$)}]
    \draw[draw=#4!85!black, fill=white, dashed, dash pattern=on 0.8pt off 0.5pt, line width=0.28pt] (0,0) rectangle (0.30,0.30);
  \end{scope}
}

\newcommand{\IntroLargeParcelBoxShift}[4]{%
  \begin{scope}[shift={($(#1)+(#2,#3)$)}]
    \filldraw[draw=#4!80!black, fill=#4!42, line width=0.32pt] (0,0) rectangle (0.42,0.42);
  \end{scope}
}

\newcommand{\IntroLargeEmptyParcelBoxShift}[4]{%
  \begin{scope}[shift={($(#1)+(#2,#3)$)}]
    \draw[draw=#4!85!black, fill=white, dashed, dash pattern=on 1.0pt off 0.6pt, line width=0.34pt] (0,0) rectangle (0.42,0.42);
  \end{scope}
}

\newcommand{\IntroEmptyParcelStackShift}[3]{%
  \begin{scope}[shift={($(#1)+(#2,#3)$)}]
    \draw[draw=introParcelDraw, fill=white, dashed, dash pattern=on 0.8pt off 0.5pt, line width=0.28pt] (0,0) rectangle (0.34,0.34);
    \draw[draw=introParcelDraw, fill=white, dashed, dash pattern=on 0.8pt off 0.5pt, line width=0.28pt] (0.34,0) rectangle (0.68,0.34);
    \draw[draw=introParcelDraw, fill=white, dashed, dash pattern=on 0.8pt off 0.5pt, line width=0.28pt] (0.17,0.34) rectangle (0.51,0.68);
  \end{scope}
}

\newcommand{\IntroRequestMark}[1]{%
  \node[font=\fontsize{7.7pt}{7.7pt}\selectfont\bfseries, text=introGHDraw, inner sep=0pt] at ($(#1)+(0,0.78)$) {!};
}

\begin{subfigure}[t]{0.315\linewidth}
\centering
\begin{tikzpicture}[introPanel]
  \IntroNetworkCoords
  \draw[draw=introParcelPath, opacity=0.18, line width=3.2pt, line cap=round] (LH5) -- (GH);
  \draw[draw=introParcelPath, opacity=0.18, line width=3.2pt, line cap=round] (GH) -- (LH10);
  \draw[draw=introBlue, opacity=0.18, line width=3.2pt, line cap=round] (LH12) -- (GH);
  \draw[draw=introBlue, opacity=0.18, line width=3.2pt, line cap=round] (GH) -- (LH7);
  \draw[draw=introBlue, opacity=0.18, line width=3.2pt, line cap=round] (LH7) -- (LH1);
  \draw[draw=introBlue, opacity=0.18, line width=3.2pt, line cap=round] (LH1) -- (LH4);
  \DrawRouteEnd{introBlue}{GH/LH7,LH7/LH1,LH1/LH4,LH4/LH3}
  \DrawRouteEnd{introOrange}{GH/LH9,LH9/LH6,LH6/LH8,LH8/LH5}
  \DrawRouteEnd{introPurple}{GH/LH10,LH10/LH2,LH2/LH11,LH11/LH12}
  \draw[introRouteEnd,introBlue,shorten >=4.8pt] (LH3) -- (GH);
  \draw[introRouteEnd,introOrange,shorten >=4.8pt] (LH5) -- (GH);
  \draw[introRouteEnd,introPurple,shorten >=4.8pt] (LH12) -- (GH);
  \draw[introGray!70, densely dotted, line width=0.5pt] (LH5) -- (LH10);
  \path (GH) -- (LH7) pic[pos=0.52,sloped,yshift=4.5pt,scale=0.27,xscale=-1,draw=introBlue,fill=introBlue!15] {introTruck};
  \path (LH9) -- (LH6) pic[pos=0.58,sloped,yshift=4.5pt,scale=0.27,draw=introOrange,fill=introOrange!15] {introTruck};
  \path (LH2) -- (LH11) pic[pos=0.47,sloped,yshift=4.5pt,scale=0.27,xscale=-1,draw=introPurple,fill=introPurple!15] {introTruck};
  \IntroGateway
  \IntroLocalHubs
  \IntroLargeParcelBoxShift{LH5}{-0.18}{-0.82}{introParcelPath}
  \IntroLargeEmptyParcelBoxShift{LH10}{-0.82}{-0.20}{introParcelPath}
  \IntroLargeParcelBoxShift{LH12}{-0.78}{-0.18}{introBlue}
  \IntroLargeEmptyParcelBoxShift{LH4}{-0.78}{-0.18}{introBlue}
  \path[use as bounding box] (0.0,0.1) rectangle (10.0,7.7);
\end{tikzpicture}
\caption{Hierarchical tours}
\label{fig:intro_hierarchy}
\end{subfigure}
\hfill
\begin{subfigure}[t]{0.315\linewidth}
\centering
\begin{tikzpicture}[introPanel]
  \IntroNetworkCoords
  \coordinate (truckBlueFlex) at (2.25,5.20);
  \coordinate (truckOrangeFlex) at (9.15,4.90);
  \coordinate (truckPurpleFlex) at (4.30,0.60);
  \draw[introDispatch,introBlue] (truckBlueFlex) -- (LH12);
  \draw[introDispatch,introOrange] (truckOrangeFlex) -- (LH5);
  \draw[introDispatch,introPurple] (truckPurpleFlex) -- (LH11);
  \draw[introDelivery,introBlue] (LH12) -- (LH7);
  \draw[introDelivery,introBlue] (LH7) -- (LH1);
  \draw[introDelivery,introBlue] (LH1) -- (LH4);
  \draw[introDelivery,introOrange] (LH5) -- (LH8);
  \draw[introDelivery,introOrange] (LH8) -- (LH2);
  \draw[introDelivery,introPurple] (LH11) -- (LH3);
  \path (truckBlueFlex) -- (LH12) pic[pos=0,sloped,yshift=4.5pt,scale=0.27,draw=introBlue,fill=introBlue!15] {introTruck};
  \path (truckOrangeFlex) -- (LH5) pic[pos=0,sloped,yshift=4.5pt,scale=0.27,xscale=-1,draw=introOrange,fill=introOrange!15] {introTruck};
  \path (truckPurpleFlex) -- (LH11) pic[pos=0,sloped,yshift=4.5pt,scale=0.31,xscale=-1,draw=introPurple,fill=introPurple!15] {introTruckPurpleCargo};
  \IntroLocalHubs
  \IntroParcelBoxShift{LH12}{0.28}{0.20}{introGHFill}
  \IntroParcelBoxShift{LH12}{0.58}{0.20}{introGHFill}
  \IntroParcelBoxShift{LH12}{0.43}{0.50}{introBlue}
  \IntroEmptyParcelBoxShift{LH7}{0.26}{0.22}{introGHFill}
  \IntroEmptyParcelBoxShift{LH1}{0.25}{0.20}{introBlue}
  \IntroEmptyParcelBoxShift{LH4}{0.24}{0.18}{introGHFill}
  \IntroParcelBoxShift{LH5}{-0.88}{0.38}{introPurple}
  \IntroParcelBoxShift{LH5}{-0.58}{0.38}{introOrange}
  \IntroEmptyParcelBoxShift{LH8}{0.23}{0.22}{introPurple}
  \IntroEmptyParcelBoxShift{LH2}{0.18}{0.18}{introOrange}
  \IntroParcelBoxShift{LH11}{-0.74}{0.26}{introPurple}
  \IntroEmptyParcelBoxShift{LH11}{-1.08}{0.26}{introPurple}
  \IntroEmptyParcelBoxShift{LH3}{-0.74}{0.24}{introPurple}
  \IntroRequestMark{LH12}
  \IntroRequestMark{LH5}
  \IntroRequestMark{LH11}
  \path[use as bounding box] (0.0,0.1) rectangle (10.0,7.7);
\end{tikzpicture}
\caption{Flexible PDP}
\label{fig:intro_flexible}
\end{subfigure}
\hfill
\begin{subfigure}[t]{0.315\linewidth}
\centering
\begin{tikzpicture}[x=0.42cm,y=0.42cm]
  \coordinate (A) at (0.9,5.1);
  \coordinate (B) at (1.2,3.4);
  \coordinate (C) at (2.4,1.7);
  \coordinate (D) at (3.9,2.2);
  \coordinate (E) at (4.8,3.4);
  \coordinate (F) at (5.8,4.5);
  \coordinate (G) at (7.1,4.4);
  \coordinate (H) at (8.0,3.1);
  \coordinate (I) at (8.8,5.0);
  \coordinate (J) at (9.8,3.8);
  \coordinate (K) at (9.4,1.8);
  \coordinate (L) at (6.6,1.7);

  \IntroLineSegment{introBlue}{A}{B}{0pt}
  \IntroLineSegment{introBlue}{B}{C}{0pt}
  \IntroLineSegment{introBlue}{C}{D}{-0.9pt}
  \IntroLineSegment{introBlue}{D}{E}{-1.7pt}
  \IntroLineSegment{introBlue}{E}{F}{-1.7pt}
  \IntroLineSegment{introBlue}{F}{A}{0pt}

  \IntroLineSegment{introPurple}{D}{E}{0pt}
  \IntroLineSegment{introPurple}{E}{F}{0pt}
  \IntroLineSegment{introPurple}{F}{G}{-0.9pt}
  \IntroLineSegment{introPurple}{G}{H}{-1.7pt}
  \IntroLineSegment{introPurple}{H}{L}{-0.9pt}
  \IntroLineSegment{introPurple}{L}{D}{-0.9pt}

  \IntroLineSegment{introTeal}{C}{D}{0.9pt}
  \IntroLineSegment{introTeal}{D}{E}{1.7pt}
  \IntroLineSegment{introTeal}{E}{F}{1.7pt}
  \IntroLineSegment{introTeal}{F}{G}{0.9pt}
  \IntroLineSegment{introTeal}{G}{H}{0pt}
  \IntroLineSegment{introTeal}{H}{L}{0.9pt}
  \IntroLineSegment{introTeal}{L}{C}{0pt}

  \IntroLineSegment{introOrange}{G}{I}{0pt}
  \IntroLineSegment{introOrange}{I}{J}{0pt}
  \IntroLineSegment{introOrange}{J}{K}{0pt}
  \IntroLineSegment{introOrange}{K}{H}{0pt}
  \IntroLineSegment{introOrange}{G}{H}{1.7pt}

  \IntroTruckOnSegment{introBlue}{A}{F}{0pt}{0.42}{4.0pt}
  \IntroTruckOnSegment{introPurple}{H}{L}{-0.9pt}{0.54}{4.0pt}
  \IntroTruckOnSegment{introTeal}{C}{L}{0pt}{0.60}{-4.0pt}
  \IntroTruckOnSegment{introOrange}{I}{J}{0pt}{0.48}{4.0pt}

  \IntroLineStations
  \path[use as bounding box] (0.1,0.7) rectangle (10.4,5.7);
\end{tikzpicture}
\caption{Repeatable lines}
\label{fig:intro_lines}
\end{subfigure}
\caption{Three middle-echelon operating structures for urban parcel networks.}
\endgroup
\end{figure}

%% file: sections/02-literature-review.tex
\section{Literature Review}
\label{sec:lit_review}
While our motivation lies in parcel delivery, the core problem is routing some underlying \emph{demand} through a fixed line network, a setting studied primarily in public transit and liner shipping. These literatures provide the closest modeling foundation, but their assumptions differ from our setting. We therefore review line-network routing and assignment models, focusing on where their treatment of demand routing, service frequency, capacity, and dwell or handling times differs from ours.

In a line network, vehicles repeatedly traverse predefined sequences of hubs, called lines. The line-planning problem decides which lines to operate and at what frequencies as to minimize travel times. While possible to jointly optimize lines and the underlying routed demand, the computational difficulty led to compromises on how the demand routing was modelled. For instance, \citep{LP_MT_Schobel} were the first to introduce an IP formulation for self-optimizing passengers, although without transfer times, and the resulting problem was NP-Hard. Even under centralized planning, scalable column-generation approaches simplify routing by limiting transfers \citep{borndorfer2012direct} or allowing at most one transfer \citep{bertsimas2021data}.

As a result, much of the literature repeatedly evaluates candidate line networks through a passenger-routing subproblem, often viewed as a simplified passenger assignment problem (PAP) within the broader transit assignment problem (TAP) literature. These include GRASP and local-search frameworks \citep{mauttone2009multi}, exact branch-and-bound or column-generation approaches \citep{vermeir2021exact}, and evolutionary, genetic, or tabu-search heuristics \citep{fan2010metaheuristic,nayeem2014transit,yao2014transit}. We now turn to how these fixed-line TAPs are modeled, with particular attention to endogenous handling-induced dwell times.

\paragraph{Travel-time components}
PAPs and TAPs differ in how they represent travel time. The simplest formulations consider only in-vehicle travel times and omit transfer costs \citep{borndorfer2007column}. When transfers are included, and no timetable is available at the line-planning stage, they are often represented by fixed penalties per transfer \citep{mumford2013new}. Frequency-based models instead approximate the expected waiting time for line $l$ over a planning period of duration $T$ by $\frac{T}{2f_l}$, where $f_l$ is the service frequency (see, e.g., \citep{bull2015optimization,sadrani2022optimization}). For fixed frequencies, these waiting and transfer costs are exogenous.

\paragraph{Routing behavior and control}
The routing model also depends on who controls the demand. In public transit, passengers are often assumed to choose individually optimal paths. With fixed travel times and no capacity effects, this reduces to shortest-path problems on the induced line network (see, e.g., \cite{mumford2013new, vermeir2021exact}). When costs depend on aggregate flows, the assignment is typically modeled as a user-equilibrium or fixed-point problem (see, e.g., \cite{Shmidt_Book_Integrated_Routing, goerigk2017line}). By contrast, under centralized control, demand routing is system optimal and is commonly modeled using multi-commodity flow formulations (see, e.g., \cite{laporte2007integrated, nachtigall2008simultaneous, friedrich2017integrating}).

\paragraph{Dwell Times}
In public transit and liner shipping, standard assignment models often have reasonable grounds for treating dwell or handling times as fixed or exogenous. In bus and rail networks, passenger boarding and alighting are relatively short, partially parallelized through multiple doors or cars, and small compared with line travel times, which makes it natural to omit dwell times or absorb them into fixed schedule buffers. In liner shipping, cargo handling can be sequential and time-consuming, but it is often not the binding determinant of port-call duration. Berth times also reflect terminal and berth scheduling, equipment availability, crew changes, refueling, schedule buffers, and synchronization with cargo arrivals, while voyages last days or weeks. As a result, port times are often treated outside the cargo-assignment model rather than as a direct consequence of routed cargo flows. Nonetheless, both literatures contain work on flow-induced dwell or handling times in line networks. We group this work into three streams below.

\subparagraph{Transit assignment with flow-dependent dwell times.}
One stream consists of transit assignment models with loading- and unloading-dependent dwell times \citep{babazadeh2005algorithm,larrain2008public,sun2013transit}. These papers study fixed transit line networks or corridors and determine equilibrium passenger assignments by extending strategy- or hyperpath-based assignment models with flow-dependent in-vehicle travel times. However, these models differ from our setting in several important ways. They seek user equilibria rather than centralized routing solutions; vehicle capacity constraints are generally omitted or treated indirectly; and, most importantly, service frequencies are treated as fixed even when dwell times depend on passenger flows. Moreover, once dwell times depend on boarding and alighting flows, the assignment cost mapping is generally nonseparable and asymmetric. Passengers boarding at a stop delay passengers already on board, while passengers already on board do not create the same dwell-time externality for those boarding. Hence, the mapping is not generally integrable, and the problem does not admit the standard Beckmann potential used in separable traffic assignment. The natural formulation is instead a variational inequality problem with an asymmetric cost mapping. As a result, method-of-successive-averages procedures are fixed-point heuristics as noted in \citet{zhang2017algorithm}.

A closer stream relaxes the fixed-frequency assumption by linking dwell times to the \emph{effective frequency} provided by a fixed fleet. Building on the concept of effective frequency in \citet{decea1993transit}, \citet{bellei2000transit} study transit assignment with variable frequencies and congestion effects; \citet{lam2002capacity} propose a capacity-restraint assignment model with elastic line frequency; and \citet{leurent2014traffic} provide a macroscopic simulation framework that captures interactions between flow-induced dwell times and line frequencies. This stream addresses an important limitation of the fixed-frequency models, but otherwise shares many of the same limitations: the models remain user-equilibrium or simulation-based, demand is represented through strategies, hyperpaths, or aggregate loading rules rather than freely routed multi-commodity flows, and vehicle capacity is often omitted or represented indirectly. Their solution methods also rely on MSA or related fixed-point iterations. A further limitation is computational: the numerical studies are typically restricted to small corridors or stylized networks, making it difficult to assess whether the proposed procedures scale to routing subproblems of the size needed inside line-planning algorithms.

\subparagraph{Rail scheduling, control, and frequency-setting models with dwell times.}
A second stream endogenizes dwell times in rail scheduling, passenger-flow control, and frequency-setting models without solving a routing problem over a line network. For example, \citet{szeto2014transit} consider dwell times that depend on onboard passenger loads. \citet{canca2018railway} consider load/unload dwell times for a railway rapid transit frequency problem. \citet{yuan2022joint} jointly optimize train timetables and passenger-flow control on a single metro direction, where dwell time is a linear function of boarding and alighting flows. However, several differences limit the applicability of these models to our setting. The resulting formulations often contain many classes of timetable, control, assignment, and passenger-flow variables resulting in large-scale MILP/MINLP formulations, and the numerical results are typically limited to small network case studies of a handful of hubs, even when heuristic methods are used. More importantly, demand is often not routed freely across a line network, and the models do not close the loop between flow-dependent dwell times to frequencies. As a result, transfer times and in-vehicle dwell costs are either simplified or omitted from the routing objective.

\subparagraph{Freight liner-shipping models.}
Centralized freight applications are closer to our setting because demand is routed by an operator rather than by self-optimizing users. In liner shipping, \citet{wang2014liner} model port time as a function of the number of containers handled and study network design with delivery deadlines. Instead of elastic frequency, their model fixes a weekly service frequency, and absorbs additional port time through buffer time or additional ships. However they do not allow cargo transshipment because connection times at transshipment ports would substantially complicate the model which is already a large MINLP that is solved by a column-generation heuristic for a 12-port Asia--Europe network. \citet{ozcan2020cargo} study cargo allocation and vessel scheduling on a fixed liner-shipping network with flow-dependent port-stay lengths and transshipment synchronization. However, services and desired frequencies are fixed, so flow-dependent port stays are absorbed through vessel schedules and shipment tardiness rather than feeding back into effective service frequency, transfer waiting times, or transport capacity.
\subparagraph{Parcel delivery models.}
In urban parcel delivery, \cite{ido_tal} model dynamic parcel routing over a line network by assigning each arriving parcel to a shortest path in a time-expanded network and updating edge capacities for future arrivals. Their dwell times are fixed and exogenous. \citet{eizen4903075dynamic} study dynamic parcel routing in a hyper-connected network with endogenous dwell times that depend on loaded and unloaded parcels, but use heuristic real-time routing and simulation over fixed vehicle routes, which take a significant amount of time to simulate.

In summary, handling-induced dwell times have been studied in both user-equilibrium passenger assignment models and centralized freight or parcel-routing settings. Some papers capture parts of the feedback between dwell times, cycle times, and effective frequency, but none provides the routing oracle needed here. Existing approaches typically rely on equilibrium, simulation, MSA, or heuristic methods; impose restrictive assumptions on routing behavior; omit or indirectly represent capacity; simplify transfer and in-vehicle dwell costs; or test only small stylized networks. In contrast, our focus is a centralized routing subproblem for a fixed line plan that captures the interaction between handling times, travel times, effective frequencies, transfer times, and capacities, while producing high-quality primal solutions and strong optimality certificates on larger instances.

%% file: sections/03-problem-description.tex
\section{Problem Description and Modelling}
\label{sec:methods}
We now describe the RLN-HT problem, with model parameters summarized in Table~\ref{table:problem_parameters}. We use boldface letters for vectors, generalized inequalities are interpreted element-wise, and $\mathbf{1}$ denotes the all-ones vector.

\begin{table}[htbp]
\centering
\renewcommand{\arraystretch}{0.7} %
\caption{Summary of model parameters.\label{table:problem_parameters}}
\begin{tabular}{ll}
\toprule
\textbf{Symbol} & \textbf{Description} \\
\midrule
$\mathcal{N}$ & set of hubs in an underlying PTN \\
$\mathcal{A}$ & set of arcs in an underlying PTN \\
$c_{ij}$ & travel time on arc $(i,j) \in \mathcal{A}$ \\
$\mathcal{L}$ & set of lines, each as a simply cycle of hubs in the PTN \\
$\mathcal{N}(l)$ & set of hubs visited by line $l$ \\
$\mathcal{A}(l)$ & set of consecutive hub pairs along line $l$ \\
$\mathcal{L}(i)$ & set of lines visiting hub $i$ \\
$a_{l}$ & number of vehicles assigned to line $l$ \\
$C_{l}$ & total driving time of line $l$, $\sum_{ij \in \mathcal{A}(l)} c_ij$ \\
$\mathcal{D}$ & set of origin–destination pairs \\
$w_{k}$ & demand per period for OD pair $k$ \\
$\text{orig}(k)$ & origin of OD-pair $k$ \\
$\text{dest}(k)$ & destination of OD-pair $k$ \\
$T$ & length of a period (time units) \\
$\alpha$ & time to load one unit of demand at a hub \\
$\beta$ & time to unload one unit of demand at a hub \\
$Q$ & vehicle capacity (units of demand) \\
$f_{l}$ & frequency of line $l$ \\
\bottomrule
\end{tabular}
\end{table}

The line network is defined over a physical transport network (PTN) represented by a directed graph \(G = (\mathcal{N}, \mathcal{A})\), where each arc \((i,j) \in \mathcal{A}\) represents a direct connection between two hubs with travel time \(c_{ij}\). The set of lines is denoted by \(\mathcal{L}\); each line \(l \in \mathcal{L}\) has \(a_l\) vehicles equally spaced along a fixed simple cycle of hubs, and services are repeated indefinitely. For notational convenience, let \(\mathcal{N}(l)\) denote the set of hubs visited by line \(l\), \(\mathcal{A}(l)\) the set of pairs of consecutive hubs along line \(l\), \(\mathcal{L}(i)\) the set of lines that visit hubs \(i\) and \(C_{l}\) the summation of the travel time across a line.

The line network routes parcel demand between origin-destination (OD) hub pairs. Let $\mathcal{D}$ denote the set of OD pairs, and let $w_k$ be the demand for pair $k$ per period of length $T$. For each $k \in \mathcal{D}$, we write $\mathrm{orig}(k)$ and $\mathrm{dest}(k)$ for its origin and destination, respectively. Demand is routed to its destination along lines, possibly via multi-leg journeys with transfers at intermediate hubs. Loading and unloading one parcel at a hub take $\alpha$ and $\beta$ time units, respectively, and each vehicle has a capacity of $Q$.

The objective is to minimize the average travel time (equivalently, total travel time) of the parcels, which comprises the time spent waiting at a hub to board a vehicle, the in-vehicle time while the vehicle is moving, and the in-vehicle time while the vehicle dwells to load and unload parcels. We devise a planning-level oracle for evaluating a given fixed line network. Note that this is not a real-time dispatching policy. A fully operational model would require a time-expanded, online, integer routing-and-timetabling formulation, which is already intractable for modest instances and ill-suited for embedding within line-planning frameworks. We therefore use a fluid approximation that captures the feedback loop between parcel flows, handling-induced dwell times, effective frequencies, and capacity.  As a result, our model makes the following assumptions: (1) the demand rate for each origin-destination pair is constant; (2) parcels are infinitely divisible (modeled as continuous flow); (3) the expected waiting time of parcels to be boarded is half the cycle time of the line divided by the number of vehicles assigned to it; (4) the system operates indefinitely under the same conditions and hence we optimize its steady-state performance.

\subsection{Change \& Go Network}
\label{ssec:CGN}
To model the line-specific travel, boarding, and unloading decisions, we use a modified version of a \emph{change-and-go network} (CGN) from the line-planning literature, denoted $G_{\text{CGN}} = (\mathcal{N}_{\text{CGN}}, \mathcal{A}_{\text{CGN}})$. Its node set $\mathcal{N}_{\text{CGN}}$ consists of an input/output node $IO_i$ for each physical hub $i \in \mathcal{N}$, representing the point where demand enters, leaves, or transfers at hub $i$, and a line node $V_i^l$ for each line $l \in \mathcal{L}$ and hub $i \in \mathcal{N}(l)$, representing being on board a vehicle of line $l$ at hub $i$. The arc set $\mathcal{A}_{\text{CGN}}$ contains three types of arcs:
\begin{enumerate}
    \item \emph{Travel arcs} $V_i^l \rightarrow V_j^l$ for $(i,j) \in \mathcal{A}(l)$,
    representing in-vehicle travel along arc $(i,j)$ on line $l$.
    \item \emph{Boarding arcs} $IO_i \rightarrow V_i^l$ for each $i \in \mathcal{N}(l)$ and $l$, representing boarding line $l$ at hub $i$.
    \item \emph{Unloading arcs} $V_i^l \rightarrow IO_i$ for each $i \in \mathcal{N}(l)$ and $l$, representing getting off line $l$ at hub $i$.
\end{enumerate}
In the CGN, boarding arcs represent the time a parcel waits before getting on a line at its origin, or at an intermediate hub while awaiting a connecting line. The cost of a boarding arc comprises the expected time until a vehicle of that line arrives and the dwell time until it departs. Loading occurs during this dwell period, so a separate loading arc would be redundant, as any parcel that uses a boarding arc necessarily boards the corresponding line before departure. We therefore capture boarding and loading through boarding arcs, while unloading is represented separately when the parcel leaves a vehicle. Note that because each line's cycle time is determined endogenously, the cost of a boarding arc is itself a function of the routing decisions rather than a fixed parameter. This is an inherent source of nonlinearity in our model.

An illustrative example is shown in Figure~\ref{fig:cgn-illustration}. On the left, a two-line network is depicted, consisting of the red thin line $3 \rightarrow 4 \rightarrow 2 \rightarrow 3$ and the green thicker line $1 \rightarrow 4 \rightarrow 1$. The corresponding CGN is shown in the middle, where each line’s copy is enclosed in a square and connected to the IO nodes via boarding and unloading arcs. On the right, a path for a parcel with origin $1$ to destination $2$ is highlighted, corresponding to boarding the green line at node $1$, traveling from $1$ to $4$, transferring at node $4$ to the red line, traveling from $4$ to $2$, and lastly alighting at node $2$.

\begin{figure}[htbp]
\centering
\begin{minipage}[c]{0.32\textwidth}
\centering
\begin{tikzpicture}[>=Latex, node distance=2cm, every node/.style={circle, draw, minimum size=0.5cm}, scale=1]
\node[fill=white] (io1) at (-1, -1) {$1$};
\node[fill=white] (io2) at (1, -1) {$2$};
\node[fill=white] (io3) at (-1, 1) {$3$};
\node[fill=white] (io4) at (1, 1) {$4$};
\tikzset{
  greenEdge/.style={->, draw=green, line width=1.6pt},
  redEdge/.style={->, draw=red,   line width=0.5pt}
}

\draw[greenEdge] (io1) to[bend left=20] (io4);
\draw[greenEdge] (io4) to[bend left=20] (io1);
\draw[redEdge] (io3) -- (io4);
\draw[redEdge] (io4) -- (io2);
\draw[redEdge] (io2) -- (io3);

\end{tikzpicture}
\end{minipage}
\label{fig:cgn_ptn}
\hfill
\begin{minipage}[c]{0.32\textwidth}
\centering
\begin{tikzpicture}[scale=0.7,>=Latex, node distance=2cm,
  arcred/.style={->, line width=0.5pt, red},
  arcgreen/.style={->, line width=1.6pt, green},
  arcdashed/.style={->, thick, dashed},
  arcdotted/.style={->, thick, dotted}
]

\node[fill=red, draw=black, shape= circle,inner sep=1pt, outer sep=0pt] (v31) at (-2,6.5) {$V_{3}^{1}$};
\node[fill=red, draw=black, shape= circle,inner sep=1pt, outer sep=0pt] (v41) at (0,6.5) {$V_{4}^{1}$};
\node[fill=red, draw=black, shape= circle,inner sep=1pt, outer sep=0pt] (v21) at (0,4.5) {$V_{2}^{1}$};
\node[fill=green, draw=black, shape= circle,inner sep=1pt, outer sep=0pt] (v12) at (-2,1) {$V_{1}^{2}$};
\node[fill=green, draw=black, shape= circle,inner sep=1pt, outer sep=0pt] (v42) at (0,3) {$V_{4}^{2}$};

\node[fill=white, draw=black, shape= circle,inner sep=1pt, outer sep=0pt] (io3) at (-4,5.25) {$IO_{3}$};
\node[fill=white, draw=black, shape= circle,inner sep=1pt, outer sep=0pt] (io4) at (2,5.25) {$IO_{4}$};
\node[fill=white, draw=black, shape= circle,inner sep=1pt, outer sep=0pt] (io1) at (-4,2) {$IO_{1}$};
\node[fill=white, draw=black, shape= circle,inner sep=1pt, outer sep=0pt] (io2) at (2,2) {$IO_{2}$};

\draw[arcgreen] (v12) to[bend left=20] (v42);
\draw[arcgreen] (v42) to[bend left=20] (v12);
\draw[arcred] (v31) -- (v41);
\draw[arcred] (v41) -- (v21);
\draw[arcred] (v21) -- (v31);

\draw[arcdashed] (io1) to[bend left=10] (v12);
\draw[arcdashed] (io4) to[bend left=10] (v42);
\draw[arcdashed] (io3) to[bend left=10] (v31);
\draw[arcdashed] (io4) to[bend left=10] (v41);
\draw[arcdashed] (io2) to[bend left=10] (v21);

\draw[arcdotted] (v12) to[bend left=10] (io1);
\draw[arcdotted] (v42) to[bend left=10] (io4);
\draw[arcdotted] (v31) to[bend left=10] (io3);
\draw[arcdotted] (v41) to[bend left=10] (io4);
\draw[arcdotted] (v21) to[bend left=10] (io2);

\draw[blue, thick, rounded corners] ($(v31) + (-0.6, 0.6)$) rectangle ($(v21) + (0.6, -0.6)$);
\draw[blue, thick, rounded corners] ($(v12) + (-0.6, -0.6)$) rectangle ($(v42) + (0.6, 0.6)$);

\end{tikzpicture}
\end{minipage}
\hfill
\begin{minipage}[c]{0.32\textwidth}
\centering
\begin{tikzpicture}[scale=0.7,>=Latex, node distance=2cm,
  arcred/.style={->, line width=0.5pt, red},
  arcgreen/.style={->, line width=1.6pt, green},
  arcdashed/.style={->, thick, dashed},
  arcdotted/.style={->, thick, dotted}
]

  \node[fill=red, draw=black, shape= circle,opacity=0.2,inner sep=1pt, outer sep=0pt] (v31) at (-2,6.5) {$V_{3}^{1}$};
  \node[fill=red, draw=black, shape= circle,inner sep=1pt, outer sep=0pt] (v41) at (0,6.5) {$V_{4}^{1}$};
  \node[fill=red, draw=black, shape= circle,inner sep=1pt, outer sep=0pt] (v21) at (0,4.5) {$V_{2}^{1}$};
  \node[fill=green, draw=black, shape= circle,inner sep=1pt, outer sep=0pt] (v12) at (-2,1) {$V_{1}^{2}$};
  \node[fill=green, draw=black, shape= circle,inner sep=1pt, outer sep=0pt] (v42) at (0,3) {$V_{4}^{2}$};

  \node[fill=white, draw=black, shape= circle,opacity=0.2,inner sep=1pt, outer sep=0pt] (io3) at (-4,5.25) {$IO_{3}$};
  \node[fill=white, draw=black, shape= circle,inner sep=1pt, outer sep=0pt] (io4) at (2,5.25) {$IO_{4}$};
  \node[fill=white, draw=black, shape= circle,inner sep=1pt, outer sep=0pt] (io1) at (-4,2) {$IO_{1}$};
  \node[fill=white, draw=black, shape= circle,inner sep=1pt, outer sep=0pt] (io2) at (2,2) {$IO_{2}$};

  \draw[arcgreen] (v12) to[bend left=20] (v42);
  \draw[arcgreen,opacity=0.2] (v42) to[bend left=20] (v12);
  \draw[arcred,opacity=0.2] (v31) -- (v41);
  \draw[arcred] (v41) -- (v21);
  \draw[arcred,opacity=0.2] (v21) -- (v31);

  \draw[arcdashed] (io1) to[bend left=10] (v12);
  \draw[arcdashed,opacity=0.2] (io4) to[bend left=10] (v42);
  \draw[arcdashed,opacity=0.2] (io3) to[bend left=10] (v31);
  \draw[arcdashed] (io4) to[bend left=10] (v41);
  \draw[arcdashed,opacity=0.2] (io2) to[bend left=10] (v21);

  \draw[arcdotted,opacity=0.2] (v12) to[bend left=10] (io1);
  \draw[arcdotted] (v42) to[bend left=10] (io4);
  \draw[arcdotted,opacity=0.2] (v31) to[bend left=10] (io3);
  \draw[arcdotted,opacity=0.2] (v41) to[bend left=10] (io4);
  \draw[arcdotted,opacity=0.2] (v21) to[bend left=10] (io2);

  \draw[blue, thick, rounded corners] ($(v31) + (-0.6, 0.6)$) rectangle ($(v21) + (0.6, -0.6)$);
  \draw[blue, thick, rounded corners] ($(v12) + (-0.6, -0.6)$) rectangle ($(v42) + (0.6, 0.6)$);

\end{tikzpicture}
\end{minipage}
\par\vspace{1.2ex}
\hspace*{0.34\textwidth}%
\begin{minipage}{0.64\textwidth}
\centering
\begin{tikzpicture}
\node[draw, rounded corners, fill=white, font=\tiny, inner sep=3pt] {
  \begin{tabular}{@{}c@{\hspace{2pt}}l@{\hspace{7pt}}c@{\hspace{2pt}}l@{\hspace{7pt}}c@{\hspace{2pt}}l@{\hspace{7pt}}c@{\hspace{2pt}}l@{}}
    \tikz[baseline=-0.5ex]{\draw[-{Latex[length=2.0mm,width=1.5mm]}, line width=0.45pt, red] (0,0) -- (0.42,0);} & Travel on line 1 &
    \tikz[baseline=-0.5ex]{\draw[-{Latex[length=2.0mm,width=1.5mm]}, line width=1.2pt, green] (0,0) -- (0.42,0);} & Travel on line 2 &
    \tikz[baseline=-0.5ex]{\draw[-{Latex[length=2.0mm,width=1.5mm]}, line width=0.6pt, dashed] (0,0) -- (0.42,0);} & Boarding &
    \tikz[baseline=-0.5ex]{\draw[-{Latex[length=2.0mm,width=1.5mm]}, line width=0.6pt, dotted] (0,0) -- (0.42,0);} & Unload
  \end{tabular}
};
\end{tikzpicture}
\end{minipage}
\caption{Two-line network (left), CGN (middle), and highlighted path from $IO_{1}$ to $IO_{2}$ (right).}
\label{fig:cgn-illustration}
\end{figure}
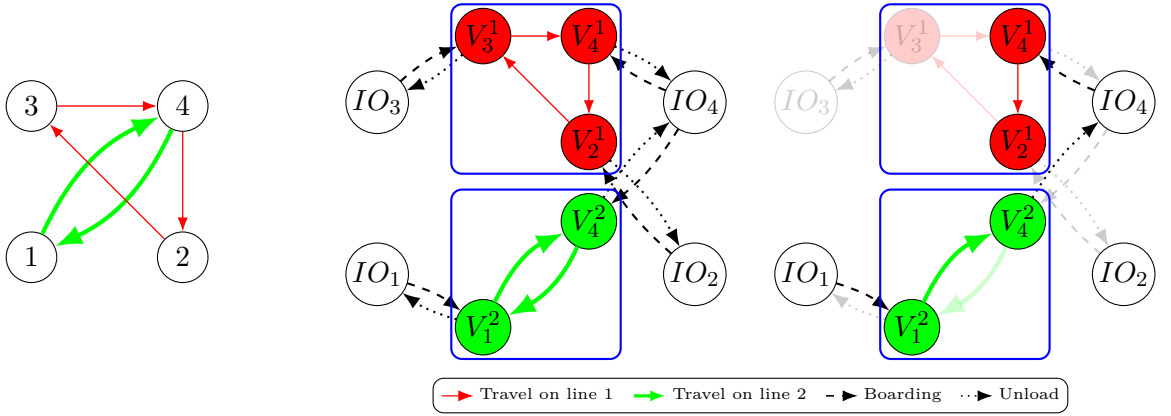

\subsection{Generalized Capacity-and-Flow Assignment Formulation}
\label{ssec:gen_CFA}
We formulate RLN-HT as a generalized capacity-and-flow assignment (CFA) model. Define the following flow variables:
\begin{description}
    \item[$d_{ij}^{kl}$] the quantity of demand on travel arc $V_{i}^{l} \rightarrow V_{j}^{l}$ for commodity $k \in \mathcal{D}$ during the period.
    \item [$t_{i}^{kl}$] the quantity of demand on boarding arc $IO_{i} \rightarrow V_{i}^{l}$ for commodity $k \in \mathcal{D}$ during the period.
    \item [$u_{i}^{kl}$] the quantity of demand on unloading arc $V_{i}^{l} \rightarrow IO_{i}$ for commodity $k \in \mathcal{D}$ during the period.
\end{description}
We refer to the entirety of the flow variables by $\boldsymbol{x} \coloneq (\boldsymbol{d},\boldsymbol{t},\boldsymbol{u})$. In addition, let $\Delta_i^l$ denote the dwell time of a vehicle of line $l$ at hub $i$, that is, the time each vehicle of line $l$ spends at hub $i$ per visit to load and unload parcels. Our goal is to determine the routes for all commodities that minimize total travel time. We define travel time on a travel arc $(V_i^l, V_j^l)$ as $c_{ij} + \Delta_j^l$, representing the time from leaving hub $i$ until leaving hub $j$. We assign unloading arc $(V_i^l,IO_i)$ a time cost of $\Delta_i^l$, where $\Delta_i^l$ is the dwell time of a vehicle of line $l$ at hub $i$. Since the model does not track the within-dwell loading and unloading order, this cost corresponds to the conservative case in which the parcel is unloaded last. Lastly, the $a_l$ vehicles assigned to line $l$ collectively complete
\[
f_l \coloneqq \frac{T}{C_l + \sum_{i \in \mathcal{N}(l)} \Delta_i^l}\, a_l
\]
cycles during a period $T$, where the denominator is the cycle time of line $l$, including all dwell times. The travel time of boarding arcs $(IO_{i}, V_{i}^{l})$ is set to the average expected waiting time $\frac{T}{2 f_l}$, which simplifies to $\frac{C_l + \sum_{i \in \mathcal{N}(l)} \Delta_i^l}{2 a_l}$.  For a hub $i$ on line $l$, let $i_{-}$ and $i_{+}$ denote the preceding and subsequent hubs to $i$ along line $l$, respectively. To minimize the total travel time of demand, we solve the program $(\mathbf{RD}_{x,\Delta})$, where the subscripts $x$ and $\Delta$ represent the optimization over flow and dwell time variables, respectively.

\begin{subequations}\label{model:RDa_delta}
\begin{align}
(\mathbf{RD}_{x,\Delta})
\, \,
\min_{d,t,u,\Delta}
& \sum_{l \in \mathcal{L}} \Bigg[
     \sum_{(i,j) \in \mathcal{A}(l)}
     \underbrace{\bigl(c_{ij} + \Delta_{j}^{l}\bigr)}_{\text{in-vehicle times}}
     \sum_{k \in \mathcal{D}} d_{ij}^{kl}
     +
     \sum_{i \in \mathcal{N}(l)} &
     \underbrace{\frac{C_{l} + \sum_{j \in \mathcal{N}(l)} \Delta_{j}^{l}}{2a_l}}_{\text{off-vehicle times}}
     \sum_{k \in \mathcal{D}} t_{i}^{kl}
  \Bigg]
\label{eq:RD_form_1_obj}
\\[0.8ex]
\text{s.t.}\quad
& \sum_{l \in \mathcal{L}(i)} t_{i}^{kl} - u_{i}^{kl} =
\begin{cases}
w_{k},   & i = \mathrm{orig}(k),\\
-w_{k},  & i = \mathrm{dest}(k),\\
0,       & \text{otherwise},
\end{cases}
& \forall i \in \mathcal{N},\, k \in \mathcal{D} \label{eq:RD_form_1_const_1}
\\[0.6ex]
& d_{i\,i_{+}}^{kl} + u_i^{kl} - d_{i_{-}\,i}^{kl} - t_i^{kl} = 0,
& \forall l \in \mathcal{L},\, i \in \mathcal{N}(l),\, k \in \mathcal{D} \label{eq:RD_form_1_const_2}
\\[0.6ex]
& \sum_{k \in \mathcal{D}} d_{ij}^{kl}
  \le Q a_l
  \tfrac{T}{C_{l} + \sum_{s \in \mathcal{N}(l)} \Delta_{s}^{l}},
& \forall l \in \mathcal{L},\, (i,j) \in \mathcal{A}(l), \label{eq:RD_form_1_const_3}
\\[0.6ex]
& \sum_{k \in \mathcal{D}} \bigl(\alpha t_{i}^{kl} + \beta u_{i}^{kl}\bigr)
  \le \Delta_{i}^{l} a_l
  \tfrac{T}{C_{l} + \sum_{j \in \mathcal{N}(l)} \Delta_{j}^{l}}
& \forall l \in \mathcal{L},\, i \in \mathcal{N}(l), \label{eq:RD_form_1_const_4}
\\[0.6ex]
& \boldsymbol{d},\, \boldsymbol{t},\, \boldsymbol{u},\, \boldsymbol{\Delta} \ge 0. \label{eq:RD_form_1_const_5}
\end{align}
\end{subequations}

The objective \eqref{eq:RD_form_1_obj} minimizes total travel time across the three arc types. The first term captures in-vehicle travel from \(i\) to \(j\) on line \(l\), with driving time \(c_{ij}\). Demand remaining on board at \(j\) and demand transferring at \(j\) both incur dwell time \(\Delta_j^l\), so these terms are aggregated. The second term captures expected waiting time for a vehicle, including any dwell time after boarding before departure. Constraints \eqref{eq:RD_form_1_const_1} and \eqref{eq:RD_form_1_const_2} enforce flow conservation at IO nodes and line nodes in $G_{\text{CGN}}$, respectively. Constraint \eqref{eq:RD_form_1_const_3} imposes vehicle capacity: the left-hand side is the total number of parcels that travel from $i$ to $j$ along line $l$ during a period, and the right-hand side is the effective capacity provided by line $l$, given by the vehicle capacity $Q$ times the line frequency $f_l$. Similarly, constraint \eqref{eq:RD_form_1_const_4} enforces a temporal capacity relation at each hub $i \in \mathcal{N}(l)$ for line $l$. The left-hand side is the total loading and unloading time required in the period at hub $i$, as implied by the flow variables, while the right-hand side is the total dwell time available there, given by the dwell time per visit $\Delta_i^l$ multiplied by the frequency, $f_l$. Lastly, constraints \eqref{eq:RD_form_1_const_5} enforce non-negativity on flow and dwell time variables.

Formulation $(\mathbf{RD}_{x,\Delta})$ is equivalent to a bilinear program and thus a special case of a nonconvex quadratically constrained quadratic program. The nonconvexity arises from products of flow and dwell-time variables in the objective and in the capacity-type constraints \eqref{eq:RD_form_1_const_3}--\eqref{eq:RD_form_1_const_4}. This formulation is closely related to classical \emph{capacity and flow assignment} (CFA) models and, more broadly, to multicommodity capacitated network design problems. Such formulations jointly optimize routing and capacity decisions by allowing additional flow on network arcs at the expense of capacity installation or activation costs that appear in the objective. See~\cite{gendron1999multicommodity} for an overview of problems and formulations. Many exact approaches exploit a master--subproblem structure by decomposing the problem into a master problem over capacity or design variables and a flow subproblem that can be solved efficiently as a multicommodity flow problem. See~\cite{costa2005survey} for a survey of such decomposition methods.

In the RLN-HT problem, dwell times play an analogous role to capacity in CFA because they determine the service frequency of a line $f_l$ through its cycle length. As the number of visits to each hub per period, $f_{l}$, increases, the effective capacity of a line increases and the expected transfer times shorten. A similar decomposition perspective is therefore attractive, since $(\mathbf{RD}_{x,\Delta})$ reduces to a linear multicommodity flow (MCF) problem for fixed dwell times $\Delta$. However, due to the presence of nonconvex bilinear terms, this decomposition does not come with the classical global optimality guarantees enjoyed by generalized Benders approaches in regular CFA settings (e.g.,~\cite{mahey2001capacity}).

\subsection{Recovering a Pure MCF Formulation}
\label{ssec:recover_MCF}
To recover a pure multicommodity flow formulation, we first observe that there exists an optimal solution to $(\mathbf{RD}_{x,\Delta})$ where constraint~\eqref{eq:RD_form_1_const_4} must hold at equality. Intuitively, this constraint links the handling workload at hub $i$ to the dwell time available there. If this constraint were slack, dwell time could be reduced
while preserving feasibility and without increasing the objective.
In particular, hubs with no handling can be assigned zero
dwell time without worsening the objective. Proposition~\ref{prop:1} formalizes this observation. All proofs appear in Appendix~\ref{app:proofs}.

\begin{proposition}
\label{prop:1}
There exists an optimal solution to
$(\mathbf{RD}_{x,\Delta})$ in which
constraint~\eqref{eq:RD_form_1_const_4} is tight for every
line $l\in\mathcal{L}$ and hub $i\in\mathcal{N}(l)$.
\end{proposition}

Given Proposition~\ref{prop:1}, constraint~\eqref{eq:RD_form_1_const_4} can be taken to hold at equality at optimality. For a fixed flow vector, define the handling workload at hub $i$ on line $l$ as $h_i^l(\mathbf{x}) := \sum_{k\in\mathcal{D}}(\alpha t_i^{kl}+\beta u_i^{kl})$. We get the following proposition:

\begin{proposition}
\label{prop:2}
Fix a nonnegative flow $\mathbf{x}=(\boldsymbol{d},\boldsymbol{t},\boldsymbol{u})$ satisfying the flow-conservation constraints \eqref{eq:RD_form_1_const_1}--\eqref{eq:RD_form_1_const_2}.  There exists a dwell-time vector $\boldsymbol{\Delta}$ such that $(\mathbf{x},\boldsymbol{\Delta})$ is feasible for $(\mathbf{RD}_{x,\Delta})$ if and only if
\begin{align}
T a_l
&>\sum_{i\in\mathcal{N}(l)} h_i^l(\mathbf{x}),
&& \forall\, l\in\mathcal{L}, \label{eq:new_const_prop}
\\
\sum_{k\in\mathcal{D}} d_{ij}^{kl}
&\le
\frac{Q}{C_l}
\Biggl(
Ta_l
-
\sum_{i\in\mathcal{N}(l)} h_i^l(\mathbf{x})
\Biggr),
&& \forall\, (i,j)\in\mathcal{A}(l),\ \forall\, l\in\mathcal{L}, \label{eq:new_cap_const_prop}
\end{align}
When these conditions hold, a dwell-time vector minimizing
total travel time for this fixed flow is given componentwise by
\begin{equation}
\Delta_i^{*l}
=
\frac{
C_l h_i^l(\mathbf{x})
}{
Ta_l
-
\sum_{j\in\mathcal{N}(l)} h_j^l(\mathbf{x})
},
\qquad \forall\, i\in\mathcal{N}(l),\ \forall\, l\in\mathcal{L}.
\label{eq:opt_delta}
\end{equation}
\end{proposition}

By Proposition~\ref{prop:1}, the handling constraints can be taken to hold at equality in an optimal solution. If the fixed flow satisfies flow conservation, solving this equality system for the dwell times gives the closed-form expression~\eqref{eq:opt_delta}, provided the denominator is positive as required by~\eqref{eq:new_const_prop}. Substituting these dwell times into the vehicle-capacity constraints yields the linear feasibility conditions~\eqref{eq:new_cap_const_prop}. Suppressing the dependence of $h_i^l(\mathbf{x})$ on $\mathbf{x}$, when constraint \eqref{eq:RD_form_1_const_4} is tight, the handling workload must satisfy $h_i^l = f_l \Delta_i^l$, i.e., the total handling time over the period equals the dwell time per visit multiplied by the number of visits. Equation~\eqref{eq:opt_delta} allows us to represent the frequency as
\[
f_l \;=\; \frac{1}{C_l}\Bigl(T a_l - \sum_{i \in \mathcal{N}(l)} h_i^l\Bigr).
\]
Expressing $f_l$ in terms of the flow-induced workloads $h_i^l$ rather than the dwell times $\Delta_i^l$ provides a useful interpretation: the more handling operations are performed along line $l$, the lower its frequency. Moreover, the linear constraint \eqref{eq:new_const_prop} reduces to the requirement that the bracketed term above be strictly positive, that is, $\sum_{i \in \mathcal{N}(l)} h_i^l \;<\; T a_l$. If this inequality fails, no nonnegative dwell-time schedule can accommodate the required handling within a period of length $T$, and the line cannot be operated feasibly. Condition~\eqref{eq:new_cap_const_prop} ensures that the routing respects the \emph{effective} physical capacity implied by the dwell-time--induced service frequency.

Expression~\eqref{eq:opt_delta} for the optimal dwell times depends only on the flow variables $(\boldsymbol{d}, \boldsymbol{t}, \boldsymbol{u})$ and can therefore be used to obtain a reformulation purely in terms of flows. Specifically, we can substitute the expression for $\Delta_i^l$ into $(\mathbf{RD}_{x,\Delta})$ and replace constraints~\eqref{eq:RD_form_1_const_3} and \eqref{eq:RD_form_1_const_4} with the linear constraints~\eqref{eq:new_const_prop} and~\eqref{eq:new_cap_const_prop}, which ensures that the flows admit a feasible dwell-time schedule. After some algebraic manipulation of the objective function, this produces a pure flow formulation, denoted $(\mathbf{RD}_x)$, produced below. In formulation $(\mathbf{RD}_{x})$, the strict inequality in~\eqref{eq:new_const_prop} is enforced with a small $\epsilon>0$ in~\eqref{eq:RD_form_3_const_3}.

\begin{subequations}\label{model:RDa}
\begin{align}
(\mathbf{RD}_{x}) \,
\min_{d,t,u}\; &
\sum_{l \in \mathcal{L}} C_l
\frac{\displaystyle\sum_{i \in \mathcal{N}(l)}
\Big(\sum_{k \in \mathcal{D}} \alpha t_{i}^{kl} + \beta u_{i}^{kl}\Big)
\Big( \tfrac{T}{2(\alpha + \beta)} + \sum_{k \in \mathcal{D}}d_{i_{-}i}^{kl} \Big) }
{T a_{l} -
\displaystyle\sum_{k \in \mathcal{D}}\sum_{i \in \mathcal{N}(l)} (\alpha t_i^{kl} + \beta u_{i}^{kl})}
&&+ \sum_{l \in \mathcal{L}}\sum_{ij \in \mathcal{A}(l)} \sum_{k \in \mathcal{D}} c_{ij} d_{ij}^{kl}
 \label{eq:RD_form_3_obj}
\\[0.8em]
\text{s.t.}\quad
&
\sum_{l \in \mathcal{L}(i)} t_{i}^{kl} - u_{i}^{kl} =
\begin{cases}
w_{k} & \text{if } i = \mathrm{orig}(k), \\
-w_{k} & \text{if } i = \mathrm{dest}(k), \\
0 & \text{otherwise,}
\end{cases}
&& \quad \forall i \in \mathcal{N},\; k \in \mathcal{D}
\label{eq:RD_form_3_const_1}
\\
& d_{i i_{+}}^{kl} + u_i^{kl} - d_{i_{-}i}^{kl} - t_i^{kl} = 0
&& \quad \forall i \in \mathcal{N}(l),\; k \in \mathcal{D},\; l \in \mathcal{L}
\label{eq:RD_form_3_const_2}
\\
& \sum_{i \in \mathcal{N}(l)} \sum_{k \in \mathcal{D}}
\left(\alpha t_{i}^{kl} + \beta u_{i}^{kl}\right)
\le a_{l} T - \epsilon
&& \quad \forall l \in \mathcal{L}
\label{eq:RD_form_3_const_3}
\\
& \sum_{k \in \mathcal{D}} d_{ij}^{kl}
\le \frac{Q}{C_{l}}\Big(T a_{l} -
\sum_{k \in \mathcal{D}}\sum_{r \in \mathcal{N}(l)} \alpha t_{r}^{kl} + \beta u_{r}^{kl}\Big)
&& \quad \forall (i,j) \in \mathcal{A}(l),\; l \in \mathcal{L}
\label{eq:RD_form_3_const_4}
\\
& \boldsymbol{d},\; \boldsymbol{t},\; \boldsymbol{u} \ge 0
&&
\label{eq:RD_form_3_const_5}
\end{align}
\end{subequations}

Formulation $(\mathbf{RD}_{x})$ belongs to the class of non-convex, non-separable multi-commodity flow (MCF) problems with generalized capacity constraints. Its nonlinearity arises solely from the objective function \eqref{eq:RD_form_3_obj}, while all constraints are linear. Specifically, \eqref{eq:RD_form_3_const_1}–\eqref{eq:RD_form_3_const_2} are flow-conservation constraints, and \eqref{eq:RD_form_3_const_3}–\eqref{eq:RD_form_3_const_4} are capacity-type constraints. The objective \eqref{eq:RD_form_3_obj} is a linear term plus nonconvex fractional quadratic terms, driven by indefinite bilinear interactions between $\boldsymbol{t}/\boldsymbol{u}$ and $\boldsymbol{d}$.

Nonlinear MCF problems have been studied extensively and are often classified along three dimensions: convexity, separability, and integrality. Convexity refers to whether the arc cost functions are convex; separability indicates whether the cost on an arc $a$ depends only on the flow on $a$ as opposed to flows on other arcs; and integrality distinguishes between discrete and continuous flow variables. The simplest case, convex and separable MCF, in either integer or continuous form, has been well studied (see Chapters~8--10 of \citet{bertsekas1998network}), and can be handled efficiently by network flow and convex optimization methods. In contrast, nonconvex MCF problems inherit the full difficulties of general nonconvex optimization, and these difficulties are further compounded when separability is lost \citep{hochbaum2007complexity}. Even the nonconvex but separable case is highly challenging, as many standard structural properties of optimal solutions break down \citep{mahey2017multicommodity}.

A standard approach for global optimization of nonconvex problems is spatial branch-and-bound \citep{tawarmalani2013convexification}, which solves convex relaxations, often based on McCormick envelopes, at nodes of a search tree. The strength of these relaxations depends heavily on variable bounds, which branching gradually tightens. As shown in Section~\ref{sec:num_res}, off-the-shelf spatial branch-and-bound struggles to produce strong dual bounds for our formulations, even with generous computation time.

%% file: sections/04-primal-dual-certificates.tex
\section{Obtaining Strong Primal and Dual Certificates}
\label{sec:primal_dual_certificates}
Although $(\mathbf{RD}_{x})$ is difficult for generic global optimization, its structure enables fast primal solutions and strong dual bounds. Subsection~\ref{ssec:structure} formalizes this structure, and Subsection~\ref{ssec:convex_relaxations} derives two convex relaxations that yield feasible solutions and non-trivial optimality certificates. Section~\ref{sec:convex_frontier_bound} develops a residual-frontier framework for similar problems, yielding even stronger lower bounds.

\subsection{Structure of \texorpdfstring{$(\mathbf{RD}_{x})$}{(RDx)}}
\label{ssec:structure}

We first isolate the feasible region. Since all constraints in $(\mathbf{RD}_{x})$ are linear, this region is polyhedral, defined as $\Xset \coloneqq \{\boldsymbol{x}=(\boldsymbol{d},\boldsymbol{t},\boldsymbol{u}) : \text{\eqref{eq:RD_form_3_const_1}--\eqref{eq:RD_form_3_const_5} hold}\}$. Therefore, the nonlinearity in $(\mathbf{RD}_{x})$ lies in the objective, which we write as
$F(\boldsymbol{x}) =\phi(\boldsymbol{x})+r(\boldsymbol{x}) + f_{\mathrm{lin}}(\boldsymbol{x})$, where,
\[ \phi(\boldsymbol{x}) = \sum_{l\in\mathcal{L}} C_l \frac{\frac{T}{2(\alpha+\beta)} \sum\limits_{ij\in\mathcal{A}(l)} \sum\limits_{k\in\mathcal{D}} (\alpha t_i^{kl} + \beta u_i^{kl})} {T a_l - \sum\limits_{i\in\mathcal{N}(l)} \sum\limits_{k\in\mathcal{D}} (\alpha t_i^{kl} + \beta u_i^{kl})},
\quad
r(\boldsymbol{x}) = \sum_{l\in\mathcal{L}} C_l \frac{\sum\limits_{ij\in\mathcal{A}(l)} \Big[\big(\sum\limits_{k \in \mathcal{D}} (\alpha t_j^{kl} + \beta u_j^{kl})\big) \big(\sum\limits_{k \in \mathcal{D}} d_{ij}^{kl}\big)\Big]} {T a_l - \sum\limits_{i\in\mathcal{A}(l)} \sum\limits_{k \in \mathcal{D}} (\alpha t_i^{kl} + \beta u_i^{kl})} \quad \]

and $f_{\mathrm{lin}}(\boldsymbol{x}) = \sum_{l\in\mathcal{L}} \sum_{ij\in\mathcal{A}(l)} \sum_{k\in\mathcal{D}} c_{ij} d_{ij}^{kl}$. This decomposition corresponds to the three time-cost components in the original objective of $(\mathbf{RD}_{x,\Delta})$: $f_{\mathrm{lin}}(\boldsymbol{x})$ is the fixed driving-time cost, $\phi(\boldsymbol{x})$ is the transfer-time cost, and $r(\boldsymbol{x})$ is the dwell-time cost experienced by parcels after boarding a vehicle. This last term includes both the time penalty incurred by parcels that unload at a hub and the in-vehicle delay incurred by parcels that remain on board while the vehicle dwells to load and unload other parcels.

For arc $(i,j) \in \mathcal{A}(l)$ for line $l \in\mathcal{L}$, define the aggregate driving flow $X_{ij}^{l} = \sum_{k\in\mathcal{D}} d_{ij}^{kl}$, and let $\boldsymbol{X}_l$ collect these values over the arcs of line $l$. Then, both non-linear functions can be written more elegantly as follows, with bold indicating vectors and $\mathbf{h}^{T}$ indicating the transpose of the vector $\mathbf{h}$

\[
\phi(\boldsymbol{x}) = \sum_{l\in\mathcal{L}} C_l \frac{ \frac{T}{2(\alpha+\beta)}\boldsymbol{h}_{l}^{\top} \boldsymbol{1}}{Ta_l - \boldsymbol{h}_{l}^{\top} \boldsymbol{1}},
\quad r(\boldsymbol{x}) = \sum_{l\in\mathcal{L}} C_l \frac{ \boldsymbol{h}_{l}^{\top}\boldsymbol{X}_{l}}{Ta_l - \boldsymbol{h}_{l}^{\top} \boldsymbol{1}}. \quad \]

In this form, several structural features become apparent. Both $\phi(\boldsymbol{x})$ and $r(\boldsymbol{x})$ are expressed as sums of line-wise separable fractional terms, and for each line, the corresponding terms in $\phi(\boldsymbol{x})$ and $r(\boldsymbol{x})$ share the same affine denominator $Ta_{l} - \boldsymbol{h}_{l}^{\top} \boldsymbol{1}$. Moreover, on each line, the numerator is a dot product between the vector of handling times $\boldsymbol{h}$ at hubs along the line and another vector: in $\phi(\boldsymbol{x})$, this second vector is simply the all-ones vector $\boldsymbol{1}$ and a constant $\frac{T}{2(\alpha+\beta)}$, whereas in $r(\boldsymbol{x})$ it is the vector of aggregate driving flows along the line $\boldsymbol{X}_l$. Finally, since constraint~\eqref{eq:RD_form_3_const_3} enforces that the denominator $Ta_{l} - \boldsymbol{h}_{l}^{\top} \boldsymbol{1}$ is strictly positive for every line $l$, each fraction in $\phi(\boldsymbol{x})$ is convex on the feasible domain, and thus $\phi(\boldsymbol{x})$ is convex over the feasible region.

The cost function $\phi(\boldsymbol{x})$ has a structure similar to Kleinrock's classical queueing-based congestion function $\frac{x}{c-x}$ in \citet{kleinrock1964communicationnets}, which arises from an $M/M/1$ model with Poisson arrivals and exponential service. In that setting, $\frac{x}{c-x}$ is proportional to the total delay accumulated per unit time by traffic of rate $x$ served by an arc of capacity $c$. In our case, an analogous congestion term is defined not for an individual arc but over an entire line as $\frac{H_l}{T a_l - H_l}$, where $H_l = \boldsymbol{h}_l^{\top}\boldsymbol{1}$. Here, the capacity of line $l$ is the total available minutes per period to drive or handle, namely $T a_l$, and the total time dedicated to handling demand along the line is $H_l$.
Defining the convex portion of the objective as $g(\boldsymbol{x}) \coloneqq f_{\mathrm{lin}}(\boldsymbol{x}) + \phi(\boldsymbol{x})$, we can express $(\mathbf{RD}_{x})$ compactly as minimizing the sum of a convex function $g(\boldsymbol{x})$ and a nonconvex function $r(\boldsymbol{x})$:
\begin{equation}
(\mathbf{RD}_{x}) \qquad
F^* = \min_{\boldsymbol{x} \in \Xset}
\left\{ g(\boldsymbol{x}) + r(\boldsymbol{x}) \right\}.
\label{eq:RD_x_nice}
\end{equation}

\subsection{Convex Relaxations}
\label{ssec:convex_relaxations}
Writing $(\mathbf{RD}_{x})$ as in~\eqref{eq:RD_x_nice} identifies the natural convex-relaxation strategy, namely replacing the nonconvex term $r(\boldsymbol{x})$ with a convex lower approximation while leaving the feasible region unchanged. Specifically, suppose that $\underline{r}$ is a convex function satisfying $\underline{r}(\boldsymbol{x}) \le r(\boldsymbol{x}) \forall \boldsymbol{x}\in\Xset$. This gives the convex relaxation
\[
LB^* =
\min_{\boldsymbol{x} \in \Xset}
\left\{ g(\boldsymbol{x}) + \underline{r}(\boldsymbol{x}) \right\}.
\]
Since $\underline{r}$ underestimates $r$ on the same feasible region, $LB^*$ is a valid lower bound on the optimal value of $(\mathbf{RD}_{x})$, i.e., $LB^*\le F^*$. At the same time, any optimizer $\boldsymbol{x}_{LB}$ of the relaxation is feasible for the original problem, so evaluating the original objective at this point gives a valid upper bound, $F^*\le F(\boldsymbol{x}_{LB})$. Thus, obtaining strong upper and lower bounds amounts to finding a strong convex underestimator $\underline{r}$.

The simplest choice is $\underline{r}(\boldsymbol{x})=0$, since $r(\boldsymbol{x})\ge 0$ on $\Xset$. This drops the dwell-time interaction term while retaining the convex transfer term $\phi(\boldsymbol{x})$ and the linear driving cost. We denote the resulting relaxation by $(\mathbf{RD}_{x}^{\mathrm{transfer}})$, where $LB_{\mathrm{tr}}^*=\min_{\boldsymbol{x}\in\Xset}\{g(\boldsymbol{x})\}$. The name reflects the fact that the only nonlinear term retained in the objective is the transfer-time term $\phi(\boldsymbol{x})$. Conveniently, this convex program can be written as a second-order conic program (SOCP). Define $y_l=Ta_l-\boldsymbol{h}_l^\top\boldsymbol{1}$. The line-$l$ fraction in $\phi(\boldsymbol{x})$ can then be written as

\[
C_l\frac{\frac{T}{2(\alpha+\beta)}\boldsymbol{h}_l^\top\boldsymbol{1}}{Ta_l-\boldsymbol{h}_l^\top\boldsymbol{1}}
=
\frac{T}{2(\alpha+\beta)}\left(Ta_l\frac{C_l}{y_l}-C_l\right).
\]
Introducing an epigraph variable $z_l\ge C_l/y_l$ gives the rotated-SOC representable constraint $z_l y_l\ge C_l$ with $z_l\ge0$; the positivity of $y_l$ follows from $\boldsymbol{x}\in\Xset$. This yields the compact SOCP
\begin{equation}
\label{model:RDa_socp}
\begin{aligned}
(\mathbf{RD}_{x}^{\mathrm{transfer}})\qquad
LB_{\mathrm{tr}}^*
=\min_{\boldsymbol{x},\boldsymbol{y},\boldsymbol{z}}\quad
& f_{\mathrm{lin}}(\boldsymbol{x})
+ \frac{T}{2(\alpha+\beta)}
\sum_{l\in\mathcal{L}}\left(Ta_l z_l-C_l\right) \\
\mathrm{s.t.}\quad
& \boldsymbol{x}\in\Xset,\quad
y_l=Ta_l-\boldsymbol{h}_l^\top\boldsymbol{1},\quad
z_l y_l\ge C_l,\quad z_l\ge0,
\qquad \forall l\in\mathcal{L}.
\end{aligned}
\end{equation}
The only constraints added to the feasible flow region are the definitions of $y_l$ and the rotated-SOC epigraph constraints for $C_l/y_l$. This relaxation captures how dwell times $\boldsymbol{\Delta}$ affect each line's \emph{cycle length} and, consequently, expected transfer times. However, it does not account for unloading times incurred when parcels alight from a line, nor for the time that on-board parcels spend waiting at hubs while the vehicle loads and unloads other parcels. In the original bilinear formulation $(\mathbf{RD}_{x,\Delta})$, these effects are captured by the term $d_{ij}^{kl}\Delta_j^l$, which corresponds to $r(\boldsymbol{x})$ in the pure flow formulation.

Instead of dropping the entire non-convex term $r(\boldsymbol{x})$, we can get a strictly stronger lower bound by replacing it with a strictly positive convex under-estimator over the feasible region. Constructing such an under-estimator is, however, not immediate as we need an effective lower bound on the bilinear form $\boldsymbol{h}_l^\top\boldsymbol{X}_l$ for each line $l$. Proposition~\ref{prop:x_lb} provides a structural lower bound on the driving flow that we leverage to construct such an under-estimator.

\begin{proposition}
\label{prop:x_lb}
Let $\mathbf{x}^\star = (\boldsymbol{d}^\star,\boldsymbol{t}^\star,\boldsymbol{u}^\star)$ be any optimal solution to $(\mathbf{RD}_{x})$. For every line $l \in \mathcal{L}$, consecutive hubs $(i,j) \in \mathcal{A}(l)$, and commodity $k \in \mathcal{D}$, we have $d_{ij}^{kl\,\star} \;\ge\; u_{j}^{kl\,\star}$.
\end{proposition}
Proposition~\ref{prop:x_lb} has a simple interpretation. The flow unloading from line $l$ at hub $j$ cannot exceed the flow entering $j$ on that line. Otherwise, some parcels would board and immediately unload at $j$, and removing this redundant movement would preserve feasibility and weakly improve the objective. Therefore, a valid convex under-estimator for the numerator is derived as follows, where we define $U_{j}^{l} \coloneq \sum_{k \in \mathcal{D}}u_{j}^{kl}$ and $T_{j}^{l} \coloneq \sum_{k \in \mathcal{D}}t_{j}^{kl}$:
\[
\sum_{ij \in \mathcal{A}(l)} h_{j}^{l} X_{ij}^{l} = \sum_{ij \in \mathcal{A}(l)} (\alpha T_{j}^{l} + \beta U_{j}^{l}) X_{ij}^{l} \geq  \sum_{ij \in \mathcal{A}(l)} (\alpha T_{j}^{l} + \beta U_{j}^{l}) U_{j}^{l} \geq \sum_{ij \in \mathcal{A}(l)} \beta (U_{j}^{l})^{2}
\]
where the second step arises from Proposition~\ref{prop:x_lb} where $X_{ij}^{l} \geq U_{j}^{l}$. Hence, a valid under-estimator for fraction $l$ of $r(\boldsymbol{x})$, denoted $r_{l}(\boldsymbol{x})$, is $\underline{r}_l(\boldsymbol{x})$
\[
r_l(\boldsymbol{x})
\;\ge\;
\underline{r}_l(\boldsymbol{x})
= C_l \beta \sum_{i\in\mathcal{N}(l)} \frac{\bigl(U_i^l\bigr)^2}{y_l}.
\]

Each term $(U_i^l)^2/y_l$ is quadratic-over-linear and is therefore convex for $y_l>0$. Consequently, $\underline{r}_l(\boldsymbol{x})$, and hence $\underline{r}(\boldsymbol{x}) := \sum_l \underline{r}_l(\boldsymbol{x})$, can be represented using rotated SOC constraints. Introducing auxiliary variables $s_i^l$, we obtain the SOCP that we denote $(\mathbf{RD}_{x}^{\mathrm{under}})$

\begin{equation}
  \label{model:RDa_under}
  \begin{aligned}
  (\mathbf{RD}_{x}^{\mathrm{under}})\qquad
  LB_{\mathrm{under}}^*
  =\min_{\boldsymbol{x},\boldsymbol{y},\boldsymbol{z}, \boldsymbol{s}}\quad
  & f_{\mathrm{lin}}(\boldsymbol{x})
  + \frac{T}{2(\alpha+\beta)}
  \sum_{l\in\mathcal{L}}\left(Ta_l z_l-C_l\right) + \sum_{l\in\mathcal{L}}\sum_{i\in\mathcal{N}(l)} 2 C_l \beta\, s_i^l \\
  \mathrm{s.t.}\quad
  & \boldsymbol{x}\in\Xset,\quad
  y_l=Ta_l-\boldsymbol{h}_l^\top\boldsymbol{1},\quad
  z_l y_l\ge C_l,\quad z_l\ge0,
  \qquad \forall l\in\mathcal{L},\\
  & \bigl(U_i^l\bigr)^2 \le 2y_l s_i^l,\quad s_i^l\ge0,
  \qquad \forall l\in\mathcal{L},\ i\in\mathcal{N}(l).
  \end{aligned}
  \end{equation}

Moreover, $LB_{\mathrm{under}}^* > LB_{\mathrm{tr}}^{*}$ because this relaxation uses a strictly positive convex underestimator of $r(\boldsymbol{x})$ instead of dropping the term. As before, any optimizer of $(\mathbf{RD}_{x}^{\mathrm{under}})$ is feasible for $(\mathbf{RD}_{x})$, so evaluating the original objective gives a valid upper bound.

The numerical results in Section~\ref{sec:num_res} show that both SOCP relaxations solve quickly, even for large instances, and produce meaningful optimality certificates. Dropping $r(\boldsymbol{x})$ yields certified gaps of around 10--15\% in our test instances, while the positive underestimator reduces these certificates to roughly 7--10\%. At the same time, a consistent trend emerges across all instances: \emph{the relaxation obtained by dropping $r(\boldsymbol{x})$ often returns primal solutions that are much closer to globally optimal than their lower bounds certify}. In instances where off-the-shelf spatial branch-and-bound can eventually prove optimality, the realized primal gap is often much smaller than the certified gap. For example, a relaxation solution with a 10\% certified gap can be within 1\% of the proven optimum.

%% file: sections/05-convex-frontier-bound.tex
\section{Convex Frontier Bound}
\label{sec:convex_frontier_bound}

The empirical behavior from the computation results of the SOCPs from Section~\ref{sec:primal_dual_certificates} suggests that $r(\boldsymbol{x})$ shifts objective values more than it reshapes the objective landscape. Unlike objectives with naturally opposing terms, such as revenue and cost in profit maximization, improving $g(\boldsymbol{x})$ does not appear to require a large penalty in $r(\boldsymbol{x})$. We formalize this alignment and use it to construct stronger lower bounds, beginning with an abstract setting that contains $(\mathbf{RD}_{x})$ as a special case.
\subsection{The residual frontier}
\label{subsec:frontier_general}
Consider the optimization problem $F^* := \min_{x \in \Xset} F(\boldsymbol{x}) \coloneqq g(\boldsymbol{x})+r(\boldsymbol{x})$ where  $\Xset$ is a linear feasible region, $g(\boldsymbol{x})\geq 0$ is continuous and convex, and $r(\boldsymbol{x})\geq 0$ is a continuous nonconvex residual term. A valid relaxation is obtained by dropping $r(\boldsymbol{x})$
\[
    g_c := \min\{g(\boldsymbol{x}):x\in\Xset\},
    \qquad
    x_c \in \operatorname*{arg\,min}\{g(\boldsymbol{x}):x\in\Xset\}.
    \label{eq:gc_frontier}
\]
The gap left to certify from the convex relaxation is $F(x_c)-g_c=r(x_c)$, i.e., the residual term evaluated at $x_c$. Let $r_c:=r(x_c)$ and $F_c:=F(x_c)=g_c+r_c$. Define the value function
\begin{equation}
    m(\gamma)
    :=
    \min\{r(\boldsymbol{x}):x\in\Xset,\; g(\boldsymbol{x})\le g_c+\gamma\}.
    \label{eq:m_gamma_def}
\end{equation}
We call $m$ the \emph{residual frontier}. For each budget $\gamma$, it gives the smallest residual value attainable while allowing the convex component $g(\boldsymbol{x})$ to increase by at most $\gamma$ above its minimum. This frontier captures the tradeoff hidden by the relaxation that drops $r(\boldsymbol{x})$. That relaxation identifies a minimizer of $g(\boldsymbol{x})$, but does not indicate whether a modest increase in $g(\boldsymbol{x})$ could produce a much smaller residual value. The function $m$ makes this tradeoff explicit. Since $r(\boldsymbol{x})\ge 0$, it is sufficient to consider $\gamma\in[0,r_c]$. Indeed, if $g(\boldsymbol{x})-g_c>r_c$, then $F(\boldsymbol{x})=g(\boldsymbol{x})+r(\boldsymbol{x})>g_c+r_c=F(x_c)$, so such a point cannot improve on the incumbent. The central observation is that this frontier gives an exact representation of the optimality gap of the convex-relaxation incumbent.

\begin{theorem}[Frontier representation of the optimal value]
\label{thm:frontier_gap}
The optimal value can be written as
\(\displaystyle F^*=g_c+\min_{0\le \gamma\le r_c}\{\gamma+m(\gamma)\}\). Equivalently,
\(\displaystyle F(x_c)-F^*=\max_{0\le \gamma\le r_c}\{r_c-\gamma-m(\gamma)\}\).
\end{theorem}

The intuition of Theorem~\ref{thm:frontier_gap} is as follows. Every feasible solution can be viewed as choosing a deterioration level
\(\gamma=g(\boldsymbol{x})-g_c\) relative to the minimum possible value of the convex component. Once this deterioration level is fixed, the best residual one could hope for is $m(\gamma)$. Hence the best objective achievable at that level is $g_c+\gamma+m(\gamma)$, and optimizing over all possible deterioration levels gives the representation above. In this view, $r_c-m(\gamma)$ is the improvement in the residual, while $\gamma$ is the price paid in the convex part. The net improvement over $x_c$ is therefore $\bigl(r_c-m(\gamma)\bigr)-\gamma=r_c-\gamma-m(\gamma)$, which gives the equivalent gap expression.

Theorem~\ref{thm:frontier_gap} provides a direct way to certify the quality of the convex-relaxation incumbent. Define the line $f(\gamma):=r_c-\gamma$. If the frontier lies above this line, then the reduction in residual is never large enough to compensate for the deterioration in the convex part, outlined in the following remark.

\begin{remark}[Optimality certificates]
\label{rem:frontier_certificate}
The incumbent $x_c$ is globally optimal if and only if $m(\gamma)\ge r_c-\gamma$ for all $\gamma\in[0,r_c]$. More generally, $x_c$ is $\epsilon$-optimal if and only if $m(\gamma)\ge r_c-\gamma-\epsilon$ for all $\gamma\in[0,r_c]$.
\end{remark}

From Remark~\ref{rem:frontier_certificate}, certifying a small optimality gap is equivalent to certifying that the residual frontier does not dip far below the line $r_c-\gamma$. If $m(\gamma)$ is above the line at some value of $\gamma$, then even the best residual attainable with that much convex deterioration cannot improve upon $x_c$. If $m(\gamma)$ falls below the line, then the vertical distance $r_c-\gamma-m(\gamma)$ is exactly the improvement that could be achieved at that tradeoff level. This is the sense in which the frontier measures the alignment between the two objective components: the relaxation solution is good precisely when any attempt to reduce $r(\boldsymbol{x})$ requires a comparable increase in $g(\boldsymbol{x})$.

Computing the exact frontier is generally as difficult as solving the original nonconvex problem, because each value $m(\gamma)$ minimizes the nonconvex residual $r(\boldsymbol{x})$ over a convex level set. However, its monotonicity is immediate and sufficient for our certificate. Since increasing $\gamma$ enlarges the feasible set in~\eqref{eq:m_gamma_def}, $m(\gamma)$ is nonincreasing. We use this property to build a certified piecewise-constant lower approximation. Let
\[
    0=\gamma_0<\gamma_1<\cdots<\gamma_B=r_c
\]
be a grid, and suppose that for each endpoint $\gamma_b$ we compute a valid lower bound $\ell_b\le m(\gamma_b)$. Since $m$ is nonincreasing, for every $\gamma\in[\gamma_{b-1},\gamma_b]$ we have $m(\gamma)\ge m(\gamma_b)\ge \ell_b$. Therefore, any solution whose convex deterioration lies in the interval $[\gamma_{b-1},\gamma_b]$ has objective value at least
$g_c+\gamma_{b-1}+\ell_b$. This gives the certified staircase bound
\begin{equation}
    LB_{\rm front}
    =
    g_c+
    \min_{b=1,\ldots,B}
    \{\gamma_{b-1}+\ell_b\}.
    \label{eq:frontier_staircase_bound}
\end{equation}

\input{graphics/m_gamma.tex}
The weakest stair, i.e., the interval attaining the minimum in~\eqref{eq:frontier_staircase_bound}, defines the certified bound. Figure~\ref{fig:frontier_schematic} illustrates this construction. Rather than approximating the full frontier, the certificate only needs endpoint lower bounds and monotonicity: each endpoint bound extends backward over its grid interval, and the largest remaining distance to the target line gives a valid optimality certificate.

\subsection{Computing endpoint bounds for RLN-HT}
\label{subsec:aggregate_sdp}
We now compute lower bounds $\ell_b \le m(\gamma_b)$ for RLN-HT using two ideas. First, each endpoint problem is restricted by $g(x)\le g_c+\gamma_b$, which can be much smaller than the full feasible region and therefore supports tighter variable bounds. Second, rather than using the term-wise underestimators from Section~\ref{ssec:convex_relaxations}, we build a semi-definite programming (SDP) relaxation that better captures the interaction between the line-level handling-time and onboard-flow vectors. For an endpoint $\gamma_b$, we work on the incumbent-relevant localized set
\[
    \Xset_b^{F_{c}}:=\{x\in\Xset: g(x)\le g_c+\gamma_b,\; g(x)+\underline r_b(x)\le F_{c}\},
    \label{eq:Xb_frontier_relevant}
\]
where $\underline r_b$ is any convex underestimator of $r$ valid on the localized domain. The second inequality is a safe cutoff as any point improving the incumbent satisfies $g(x)+r(x)\le F_c$, and hence $g(x)+\underline r_b(x)\le F_c$. Thus any bounds and cuts are constructed on $\Xset_b^{F_c}$ rather than on the full feasible region, which substantially tightens variable bounds and strengthen the resulting convex relaxation.

We now specialize this idea to the RLN-HT residual. As shown in Section~\ref{ssec:structure}, $r(x)$ is a sum over lines, with each line-$l$ term proportional to $\boldsymbol{h}_l^\top\boldsymbol{X}_l/y_l$. Let $z_l=(\boldsymbol{h}_l,\boldsymbol{X}_l)$, and let $e_l$ select the handling coordinates, so that $H_l=e_l^\top z_l$ and $y_l=Ta_{l}-e_l^\top z_l$. Then
\[
    r_l(x)=C_l\frac{z_l^\top M_l z_l}{y_l},
    \qquad
    M_l=\frac12\begin{pmatrix}0&I\\ I&0\end{pmatrix},
\]
where $z_l^\top M_l z_l=\boldsymbol{h}_l^\top\boldsymbol{X}_l$. Since $M_l$ is indefinite, this term is nonconvex. For each line, introduce variables $s_l$, $w_l$, and $Z_l$ representing $1/y_l$, $z_l/y_l$, and $z_lz_l^\top/y_l$, respectively. This can be written neatly as
\begin{equation}
    \begin{pmatrix}Z_l&w_l\\ w_l^\top&s_l\end{pmatrix}
    =\frac1{y_l}\begin{pmatrix}z_l\\1\end{pmatrix}\begin{pmatrix}z_l\\1\end{pmatrix}^{\!\top}.
    \label{eq:exact_perspective_lift}
\end{equation}
The matrix in~\eqref{eq:exact_perspective_lift} is PSD and rank one. Moreover, the denominator identity $y_l=Ta_{l}-e_l^\top z_l$ implies the linear relations $Ta_{l}s_l-e_l^\top w_l=1$ and $Ta_{l}w_l-Z_le_l=z_l$. We can now obtain a SDP relaxation by keeping the PSD constraint on~\eqref{eq:exact_perspective_lift}, but dropping rank one as it is the source of non-convexity and difficulty. The resulting SDP is defined as:
\begin{subequations}
\label{model:frontier_sdp_endpoint}

\begin{align}
    \ell_b=\min \quad
    &\sum_{l\in\mathcal L} C_l\langle M_l,Z_l\rangle \label{eq:frontier_sdp_obj}\\
    \text{s.t.}\quad
    &x\in \Xset_b^{F_{c}},\quad z_l=E_lx,
      \quad y_l=Ta_{l}-e_l^\top z_l, && l\in\mathcal L, \label{eq:frontier_sdp_link}\\
    &Ta_{l}s_l-e_l^\top w_l=1,
      \quad Ta_{l}w_l-Z_le_l=z_l, && l\in\mathcal L, \label{eq:frontier_sdp_identities}\\
    &\begin{pmatrix}Z_l&w_l\\ w_l^\top&s_l\end{pmatrix}\succeq0, && l\in\mathcal L. \label{eq:frontier_sdp_psd}
\end{align}
\end{subequations}
Every feasible routing in $\Xset_b^{F_{c}}$ has an exact lift satisfying these constraints with objective value equal to its residual cost. Consequently, dropping rank one gives a valid lower bound on the residual over the localized set. The SDP remains modest in size because the matrix blocks are line-level as there is one block of dimension $2|\mathcal N(l)|+1$ per line $l$, while the full multicommodity flow vector $x$ appears only in the sparse flow and localization constraints.

The base SDP can still have a significant relaxation gap, so each endpoint relaxation is strengthened using bounds computed on the localized set $\Xset_b^{F_c}$. We first perform optimization-based bound tightening (OBBT) for the denominators $y_l$, the coordinates of $z_l$, and selected projections $q^\top z_l$. These bounds support standard fractional envelopes for $w_{lj}=z_{lj}/y_l$ and McCormick constraints for the lifted identities $Z_{ljk}=z_{lj}w_{lk}$ and $Z_{ljk}=z_{lk}w_{lj}$. The resulting SDP with these additional constraints is described in more detail in Appendix~\ref{app:frontier}.

The remaining looseness comes from dropping rank one, especially in directions where handling and onboard flow move in opposite ways. To see this, write $h_l^\top X_l=\frac14\|h_l+X_l\|_2^2-\frac14\|h_l-X_l\|_2^2$. Both squared-norm terms are convex, but the second enters with a negative sign, creating the indefinite part of the residual. We therefore add cuts that limit artificial lifted variance in these anti-alignment directions as follows. Let $q$ be a direction in the coordinates of $z_l$, and suppose OBBT on the localized set gives $\ell_{lq}^{(b)}\le q^\top z_l\le u_{lq}^{(b)}$. At an exact lift, $q^\top Z_lq=(q^\top z_l)^2/y_l$, $q^\top w_l=(q^\top z_l)/y_l$, and $s_l=1/y_l$. The secant upper bound for a square over the interval $[\ell_{lq}^{(b)},u_{lq}^{(b)}]$ gives the valid cut
\begin{equation}
    q^\top Z_lq
    \le
    (\ell_{lq}^{(b)}+u_{lq}^{(b)})q^\top w_l
    -\ell_{lq}^{(b)}u_{lq}^{(b)}s_l.
    \label{eq:spectral_cut_manuscript}
\end{equation}
We add this cut for coordinate directions, hubwise anti-alignment directions such as $(e_i^h-e_i^X)/\sqrt2$, and contiguous hub segments. If substantial lifted slack remains, we generate additional dense directions from the leading eigenvectors of that slack, recompute localized bounds for the corresponding $q^\top z_l$, and add cuts of the form~\eqref{eq:spectral_cut_manuscript}. The endpoint SDP can be strengthened in many ways, so we focus additional effort where it is most likely to improve the staircase bound. To identify loose line-level lifts, we compare the residual represented by the SDP lift with the residual implied by its own aggregate variables. For line $l$, define
\[
    \Delta_l^{\mathrm{lift}}
    \coloneqq C_l\frac{z_l^\top M_lz_l}{y_l}-C_l\langle M_l,Z_l\rangle .
    \label{eq:eval_minus_sdp_manuscript}
\]
At an exact lift, this quantity is zero. Large values indicate that the SDP is using artificial lifted variance to reduce the residual on line $l$, so subsequent OBBT and dense anti-alignment cuts are focused on these lines. These line-level strengthening steps are embedded in an adaptive frontier procedure, summarized in Algorithm~\ref{alg:frontier_general}. The algorithm begins with a coarse grid, computes endpoint lower bounds, and identifies the intervals that determine the current staircase bound. It then refines only these active intervals, either by adding nearby grid points or by strengthening their endpoint relaxations. The RLN-HT implementation used in the experiments follows this same structure with the endpoint relaxations described above.

\begin{algorithm}[t]
    \caption{Adaptive frontier certificate}
    \label{alg:frontier_general}
    \begin{enumerate}
        \item Solve $\min_{x\in\Xset} g(x)$ to obtain $g_c$ and a feasible incumbent value $F_c$ and set $\gamma_B=F_c-g_c$.
        \item Choose an initial grid $0=\gamma_0<\gamma_1<\cdots<\gamma_B$ and compute valid endpoint bounds $\ell_b\le m(\gamma_b)$.
        \item Set $LB_{\rm front}=g_c+\min_b\{\gamma_{b-1}+\ell_b\}$ and identify the active interval(s).
        \item While the active interval is wide or its endpoint relaxation is loose, either add nearby grid points and compute their endpoint bounds, or strengthen the endpoint relaxations for active and near-active intervals. Update $LB_{\rm front}$ and the active interval(s).
        \item Return $LB_{\rm front}$.
    \end{enumerate}
    \end{algorithm}

%% file: graphics/m_gamma.tex
\begin{figure}[t]
    \centering
    \captionsetup[subfigure]{font=footnotesize}

    \begin{subfigure}{0.285\textwidth}
    \centering
    \begin{adjustbox}{max width=\linewidth}
\begin{tikzpicture}
    \begin{axis}[
        width=0.84\textwidth,
        height=0.84\textwidth,
        scale only axis,
        xmin=0, xmax=1,
        ymin=0, ymax=1,
        xlabel={$\gamma$},
        xlabel style={yshift=1.2em},
        ylabel={residual value},
        ylabel style={at={(axis description cs:-0.13,0.48)}, anchor=center},
        xtick={0,1},
        xticklabels={$0$,$r_c$},
        ytick={0,1},
        yticklabels={,$r_c$},
        axis lines=left,
        clip=false,
    ]
    \addplot[orange, dashed, very thick] coordinates {(0,1) (1,0)};

    \addplot[black, very thick] coordinates {
        (0.00,1.00)
        (0.08,0.86)
        (0.18,0.72)
        (0.30,0.46)
        (0.42,0.36)
        (0.55,0.22)
        (0.68,0.18)
        (0.82,0.07)
        (1.00,0.04)
    };

    \node[font=\normalsize, align=center] at (axis cs:0.55,0.95)
    {true residual frontier};
    \end{axis}
    \end{tikzpicture}
\end{adjustbox}
    \caption{Residual frontier}
    \end{subfigure}
    \hspace{0.009\textwidth}
    \begin{subfigure}{0.285\textwidth}
    \centering
    \begin{adjustbox}{max width=\linewidth}
\begin{tikzpicture}
    \begin{axis}[
        width=0.84\textwidth,
        height=0.84\textwidth,
        scale only axis,
        xmin=0, xmax=1,
        ymin=0, ymax=1,
        xlabel={\phantom{$\gamma$}},
        xlabel style={yshift=1.2em},
        xtick={0,0.15,0.30,0.45,0.60,0.75,1.0},
        xticklabels={$0$,$\gamma_1$,$\gamma_2$,$\gamma_3$,$\gamma_4$,$\gamma_5$,$r_c$},
        ytick={0,1},
        yticklabels={,$r_c$},
        axis lines=left,
        clip=false,
        x tick label style={font=\small, rotate=0, anchor=north},
    ]
    \addplot[orange, dashed, very thick] coordinates {(0,1) (1,0)};

    \addplot[black, very thick] coordinates {
        (0.00,1.00)
        (0.08,0.86)
        (0.18,0.72)
        (0.30,0.46)
        (0.42,0.36)
        (0.55,0.22)
        (0.68,0.18)
        (0.82,0.07)
        (1.00,0.04)
    };

    \addplot[
        only marks,
        mark=*,
        mark size=1.8pt,
        black
    ] coordinates {
        (0.15,0.76)
        (0.30,0.46)
        (0.45,0.33)
        (0.60,0.21)
        (0.75,0.13)
        (1.00,0.04)
    };

    \addplot[
        only marks,
        mark=*,
        mark size=2pt,
        blue
    ] coordinates {
        (0.15,0.71)
        (0.30,0.42)
        (0.45,0.29)
        (0.60,0.17)
        (0.75,0.09)
        (1.00,0.02)
    };

    \pgfplotsinvokeforeach{0.15,0.30,0.45,0.60,0.75,1.00}{
        \draw[gray!25] (axis cs:#1,0) -- (axis cs:#1,1);
    }
    \end{axis}
    \end{tikzpicture}
\end{adjustbox}
    \caption{Endpoint lower bounds}
    \end{subfigure}
    \hspace{0.018\textwidth}
    \begin{subfigure}{0.285\textwidth}
    \centering
    \begin{adjustbox}{max width=\linewidth}
\begin{tikzpicture}
    \begin{axis}[
        width=0.84\textwidth,
        height=0.84\textwidth,
        scale only axis,
        xmin=0, xmax=1,
        ymin=0, ymax=1,
        xlabel={\phantom{$\gamma$}},
        xlabel style={yshift=1.2em},
        xtick={0,0.15,0.30,0.45,0.60,0.75,1.0},
        xticklabels={$0$,$\gamma_1$,$\gamma_2$,$\gamma_3$,$\gamma_4$,$\gamma_5$,$r_c$},
        ytick={0,1},
        yticklabels={,$r_c$},
        axis lines=left,
        clip=false,
        x tick label style={font=\small, rotate=0, anchor=north},
    ]
    \path[fill=red!10, draw=none] (axis cs:0.15,0) rectangle (axis cs:0.30,1);

    \addplot[orange, dashed, very thick] coordinates {(0,1) (1,0)};

    \addplot[black, very thick] coordinates {
        (0.00,1.00)
        (0.08,0.86)
        (0.18,0.72)
        (0.30,0.46)
        (0.42,0.36)
        (0.55,0.22)
        (0.68,0.18)
        (0.82,0.07)
        (1.00,0.04)
    };

    \addplot[blue, very thick] coordinates {
        (0.00,0.71)
        (0.15,0.71)
        (0.15,0.42)
        (0.30,0.42)
        (0.30,0.29)
        (0.45,0.29)
        (0.45,0.17)
        (0.60,0.17)
        (0.60,0.09)
        (0.75,0.09)
        (0.75,0.02)
        (1.00,0.02)
    };

    \addplot[
        only marks,
        mark=*,
        mark size=2pt,
        blue
    ] coordinates {
        (0.15,0.71)
        (0.30,0.42)
        (0.45,0.29)
        (0.60,0.17)
        (0.75,0.09)
        (1.00,0.02)
    };

    \draw[red, thick, <->] (axis cs:0.15,0.42) -- (axis cs:0.15,0.85);
    \node[red, font=\normalsize, anchor=west] at (axis cs:0.20,0.64) {certified gap};

    \node[red, font=\normalsize, anchor=south west] at (axis cs:0.2,0.9)
    {weakest stair};
    \end{axis}
    \end{tikzpicture}
\end{adjustbox}
    \caption{Staircase certificate}
    \end{subfigure}

    \vspace{0.4em}

    {\footnotesize
    \setlength{\fboxsep}{1.5pt}
    \setlength{\fboxrule}{0.3pt}
    \fbox{%
    \renewcommand{\arraystretch}{0.85}%
    \begin{tabular}{@{}c@{\ }l@{\quad}c@{\ }l@{\quad}c@{\ }l@{\quad}c@{\ }l@{\quad}c@{\ }l@{}}
    \tikz[baseline=-0.6ex]{\draw[orange, dashed, very thick] (0,0) -- (0.55,0);} &
    $r_c-\gamma$ &
    \tikz[baseline=-0.6ex]{\draw[black, very thick] (0,0) -- (0.55,0);} &
    $m(\gamma)$ &
    \tikz[baseline=-0.6ex]{\draw[black, fill=black] (0.28,0) circle (1.6pt);} &
    $m(\gamma_b)$ points &
    \tikz[baseline=-0.6ex]{\draw[blue, fill=blue] (0.28,0) circle (1.8pt);} &
    $\ell_b$ points &
    \tikz[baseline=-0.6ex]{\draw[blue, very thick] (0,0) -- (0.25,0) -- (0.25,-0.10) -- (0.55,-0.10);} &
    staircase lower bound
    \end{tabular}%
    }}

    \caption{Residual-frontier certificate from endpoint lower bounds and monotonicity.}
    \label{fig:frontier_schematic}
    \end{figure}
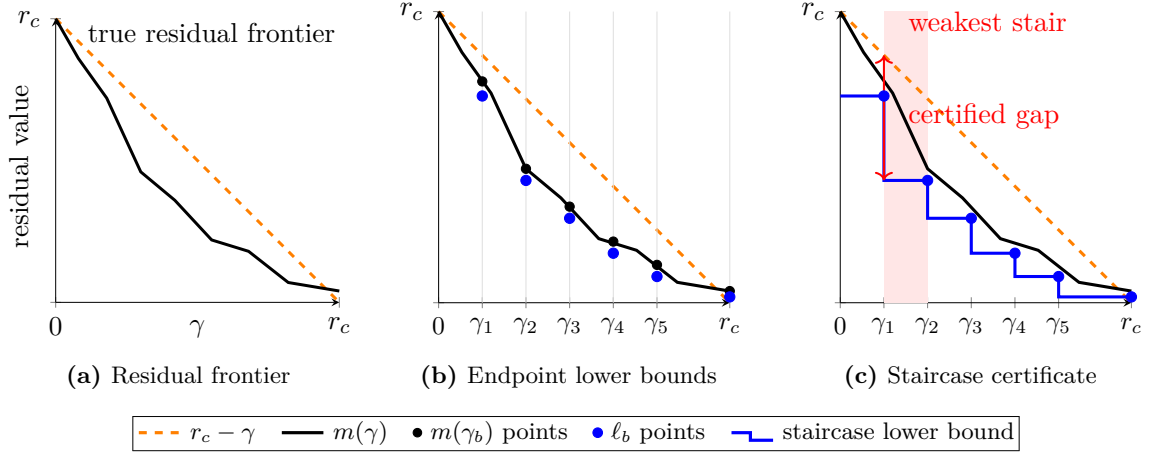

%% file: sections/06-computational-results.tex
\section{Computational Results}
\label{sec:num_res}
In this section, we evaluate the formulations as routing oracles on synthetic and real-world instances, comparing primal solution quality, certified lower bounds, and computational speed. Section~\ref{ssec:instances} describes the instances; because comparable parcel-line datasets are not publicly available, we use public-transport line-planning datasets. Section~\ref{ssec:model_strength} compares formulation strength and the value of endogenous handling times, Section~\ref{ssec:frontier_bound_comp} evaluates the frontier bound, and Section~\ref{ssec:congestion} studies sensitivity to handling times. All experiments were implemented in Python using Gurobi Optimizer v12 on a Linux workstation with an AMD Ryzen 9 5950X processor and 128~GB of RAM.

\subsection{Instances}
\label{ssec:instances}
The synthetic and real-world PTNs used in the experiments are depicted in Appendix~\ref{app:instances}.

\paragraph{Synthetic PTNs.}
We generate synthetic PTNs in three configurations, each embedded in a $60 \times 60$ minute square. The $G_{n,m}$ instances are regular $n\times m$ grids with horizontal and vertical adjacencies; the $K_n$ instances are complete graphs with $n$ nodes placed on a grid layout; and the $P_n$ instances are random planar graphs with $n$ nodes placed uniformly in the square. The five sizes of each configuration are shown in Appendix~\ref{app:instances}.

\paragraph{Real-world PTNs.}
We use two real-world benchmarks from the line-planning literature. The \emph{Sioux Falls (SF)} instance, introduced by \citet{leblanc1975efficient}, is based on the road network of Sioux Falls, South Dakota, and has 24 nodes and 38 edges. The \emph{G\"oVB} instance is based on the Göttingen bus network and has 257 nodes and 548 edges. Both datasets were obtained from LinTim \citep{SchieweSchoebelHerralaetal2025} and are depicted in Appendix~\ref{app:instances}.

For each instance, we generate line networks by sampling $B$ lines from a line pool, where $B$ is the vehicle budget. The line pool consists of all simple cycles for the synthetic\footnote{For $K_{16}$, enumerating all simple cycles was too expensive. Instead, we kept 60\% of all arcs, and then exhaustively generated the full line pool.} and Sioux Falls instances, and the expert LinTim line pool for G\"oVB.\footnote{Section~\ref{sec:methods} models each line as a simple cycle in the CGN. Real-world line networks, including G\"oVB, are typically paths that reverse at a terminus. We model this by duplicating each intermediate hub by direction, allowing different dwell times on outbound and inbound passes. This only enlarges the CGN; the formulations and results are otherwise unchanged.} Synthetic demand is distributed uniformly across all $|\mathcal{N}|(|\mathcal{N}|-1)$ OD pairs, while real-world instances use the provided OD matrices. Unless stated otherwise, $\alpha=\beta=0.5$ minutes per parcel, and the period is one day. Table~\ref{table:PTN_table} summarizes the instance parameters.

\begin{table}[h]
\centering
\caption{PTN instance parameters. \textbf{G/K/P} denote grid, complete, and planar synthetic configurations.}
\label{table:PTN_table}
\footnotesize
\setlength{\tabcolsep}{5pt}
\renewcommand{\arraystretch}{0.86}
\begin{tabular}{@{}llccccc@{}}
\toprule
\textbf{Type} & \textbf{Instance} & \textbf{\# nodes} & \textbf{\# arcs} & \textbf{B} & \textbf{Demand} & \textbf{Q}\\
\midrule
Synthetic & $G_{2,2}/K_{4}/P_{4}$ & 4  & 4/6/4 & 4 & 1500 & 200\\
 & $G_{2,3}/K_{6}/P_{6}$ & 6  & 7/15/7 & 6  & 2000 & 200\\
 & $G_{3,3}/K_{9}/P_{9}$ &9  & 12/36/11 & 8  & 3000  & 200\\
 & $G_{3,4}/K_{12}/P_{12}$ & 12 & 17/66/18 & 10  & 4000 & 200 \\
 & $G_{4,4}/K_{16}/P_{16}$ & 16 & 24/120/23 & 12 & 4000 & 200\\
\addlinespace[1pt]
Real-world & SF    & 24  & 38  & 12 & 4114 & 100 \\
 & G\"oVB & 257 & 548 & 25 & 8112 & 100\\
\bottomrule
\end{tabular}
\end{table}

\subsection{Formulation strength}
\label{ssec:model_strength}
Across all instances, we evaluate the two endogenous dwell-time formulations $(\mathbf{RD}_{x,\Delta})$ and $(\mathbf{RD}_{x})$, as well as the convex relaxations $(\mathbf{RD}^{\text{transfer}}_{x})$ and $(\mathbf{RD}^{\text{under}}_{x})$. We solve the non-convex formulation using Gurobi's off-the-shelf spatial branch-and-bound solver and the SOCP relaxations using its interior-point/barrier SOCP solver.

We aim to understand whether endogenously accounting for handling operations is important. To that end, we compare the solutions to the abovementioned formulations with a baseline that ignores handling times. Specifically, we solve $(\mathbf{RD}_{x,\Delta})$ with $\alpha=\beta=0$ and without the dwell-time variables $\boldsymbol{\Delta}$, resulting in a linear program denoted $(\mathbf{R}_{x})$. Solving $(\mathbf{R}_{x})$ yields an \emph{optimistic} routing plan that does not account for handling-induced delays. The objective function value of this linear program provides a valid lower bound for the RLN-HT problem. In addition, if the routing solution satisfies the constraints in Proposition~\ref{prop:2}, evaluating the objective of $(\mathbf{RD}_{x})$ at that routing solution provides a valid upper bound for the RLN-HT problem. This upper bound represents the \emph{realized} average travel time with the obtained routing solution once dwell times are accounted for.

\subsubsection{Synthetic instances}
To demonstrate a single instance, let us consider a line network for $K_{4}$ shown in Figure~\ref{fig:k_4_lc}. Figure~\ref{fig:k_4_gaps} presents the evolution over time of the corresponding upper (solid) and lower (dashed) bounds produced by the spatial branch-and-bound solver for the different formulations. The accompanying table reports, for each formulation, the final upper and lower bounds (UB and LB). In addition, for the two formulations that are capable of finding and proving the optimal solution, we report the time to find the optimal solution (TtOS) and the time to prove optimality (TtPO). The distinction between TtOS and TtPO highlights the gap between quickly identifying a high-quality solution and certifying its optimality.

\begin{figure}[ht]
\centering
\captionsetup[subfigure]{font=footnotesize,skip=1pt}
\begin{subfigure}[b]{0.35\linewidth}
    \centering
    \raisebox{1em}{\includegraphics[width=0.8\linewidth]{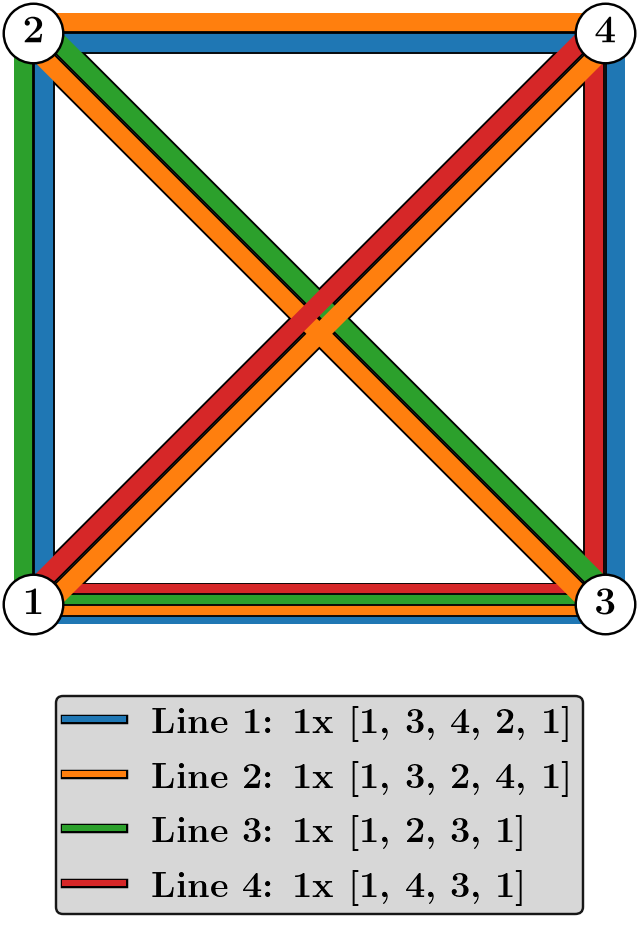}}
    \caption{line network}
    \label{fig:k_4_lc}
\end{subfigure}
\begin{subfigure}[b]{0.64\linewidth}
    \begin{adjustbox}{max width=\linewidth}
\begin{tikzpicture}
\begin{axis}[
    name=leftaxis,
    width=\linewidth,
    height=0.64\linewidth,
    xmin=-15, xmax=2600,
    ymin=170, ymax=282,
    ylabel={Average Travel Time (aTT)},
    grid=both,
    axis lines=left,
    ytick pos=left,%
    grid style={line width=.1pt, draw=gray!30},
    major grid style={line width=.2pt, draw=gray!45},
    tick align=outside,
    xtick distance=600,
    xtick pos=bottom,                 %
  xmajorgrids=true,
  tick style={major tick length=3pt, semithick, draw=black, inner sep=0pt},
  xlabel={Time (s)},
	  xlabel style={
	    at={(axis description cs:0.5,0.08)}, %
	    anchor=south,                     %
	    yshift=-5pt,                       %
	    font=\footnotesize
	  },
    legend style={draw=none}
    axis line style={line width=0.9pt},  %
    tick style={line width=0.9pt},
    every axis plot/.append style={line width=1.4pt},  %
]

\addplot[black] coordinates {(0,278.633) (2550,278.633)};
\addplot[black, lbBase] coordinates {(0,175) (2550,175)};
\addplot[socp] coordinates {(0,249.489) (2550,249.489)};
\addplot[socp, lbSOCP] coordinates {(0,220.906) (2550,220.906)};
\addplot[under] coordinates {(0,249.533) (2550,249.533)};
\addplot[under, lbUnder] coordinates {(0,238.696) (2550,238.696)};
\addplot[color=ArcOrig]
  table [x=time, y=ub, col sep=comma] {graphics/Figure_for_Complete_2x2/Arc_Orig_bounds.csv};
\addplot[color=ArcOrig, lbOrig]
  table [x=time, y=lb, col sep=comma] {graphics/Figure_for_Complete_2x2/Arc_Orig_bounds.csv};
\addplot[color=ArcNoDelta]
  table [x=time, y=ub, col sep=comma] {graphics/Figure_for_Complete_2x2/Arc_No_Delta_bounds.csv};
\addplot[color=ArcNoDelta, lbNoDelta]
  table [x=time, y=lb, col sep=comma] {graphics/Figure_for_Complete_2x2/Arc_No_Delta_bounds.csv};
\draw[gray, dashed] (axis cs:0,230) rectangle (axis cs:13,252);
\end{axis}

\begin{axis}[
  at={(leftaxis.north east)},   %
  anchor=north east,
  xshift=-0.2cm,                 %
  yshift=-3.0cm,
  width=3.8cm,                   %
  height=2.75cm,
  xmin=0, xmax=13,
  axis lines=left,
  ytick pos=left,%
  ymin=230, ymax=252,
  grid=both,
  grid style={line width=.1pt, draw=gray!30},
  tick label style={font=\scriptsize},
  label style={font=\scriptsize},
  line width=0.7pt,
  unbounded coords=discard,
  restrict x to domain*=0:13,
  restrict y to domain*=0:500,
]
\addplot[very thick, color=ArcNoDelta]
  table [x=time, y=ub, col sep=comma] {graphics/Figure_for_Complete_2x2/Arc_No_Delta_bounds.csv};
\addplot[thick, color=ArcNoDelta, dashed]
  table [x=time, y=lb, col sep=comma] {graphics/Figure_for_Complete_2x2/Arc_No_Delta_bounds.csv};
\end{axis}
\end{tikzpicture}
\end{adjustbox}
\par\vspace{0.35em}
{\centering
\scriptsize
\setlength{\tabcolsep}{2.8pt}
\renewcommand{\arraystretch}{0.88}
\begin{tabular}{lccccc}
 & $(\mathbf{R}_{x})$ & $(\mathbf{RD}_{x,\Delta})$  & $(\mathbf{RD}_{x})$ & $(\mathbf{RD}_{x}^{\text{transfer}})$ & $(\mathbf{RD}_{x}^{\text{under}})$ \\
\midrule
UB & \raisebox{1.5pt}{\tikz{\draw[very thick, black] (0,0)--(0.4,0);}} 278.6
   & \raisebox{1.5pt}{\tikz{\draw[very thick, ArcOrig] (0,0)--(0.4,0);}} 249.3
   & \raisebox{1.5pt}{\tikz{\draw[very thick, ArcNoDelta] (0,0)--(0.4,0);}} 249.3
   & \raisebox{1.5pt}{\tikz{\draw[very thick, socp] (0,0)--(0.4,0);}} 249.5
   & \raisebox{1.5pt}{\tikz{\draw[very thick, under] (0,0)--(0.4,0);}} 249.5 \\
LB & \raisebox{1.5pt}{\tikz{\draw[very thick, black,lbBase] (0,0)--(0.4,0);}}175
   & \raisebox{1.5pt}{\tikz{\draw[very thick, ArcOrig,lbOrig] (0,0)--(0.4,0);}} 249.3
   & \raisebox{1.5pt}{\tikz{\draw[very thick, ArcNoDelta,lbNoDelta] (0,0)--(0.4,0);}} 249.3
   & \raisebox{1.5pt}{\tikz{\draw[very thick, socp,lbSOCP] (0,0)--(0.4,0);}} 220.9
   & \raisebox{1.5pt}{\tikz{\draw[very thick, under,lbUnder] (0,0)--(0.4,0);}} 238.7 \\
TtOS & -- & 2.37 & 1.01 & -- & --\\
TtPO & -- & 2560 & 12.2 & -- & --\\
\bottomrule
\end{tabular}
\par}
\caption{MIP Gaps}
\label{fig:k_4_gaps}
\end{subfigure}
\caption{MIP progress of different formulations for a candidate line network of $K_{4}$.}
\label{fig:K_4_example_run}
\end{figure}
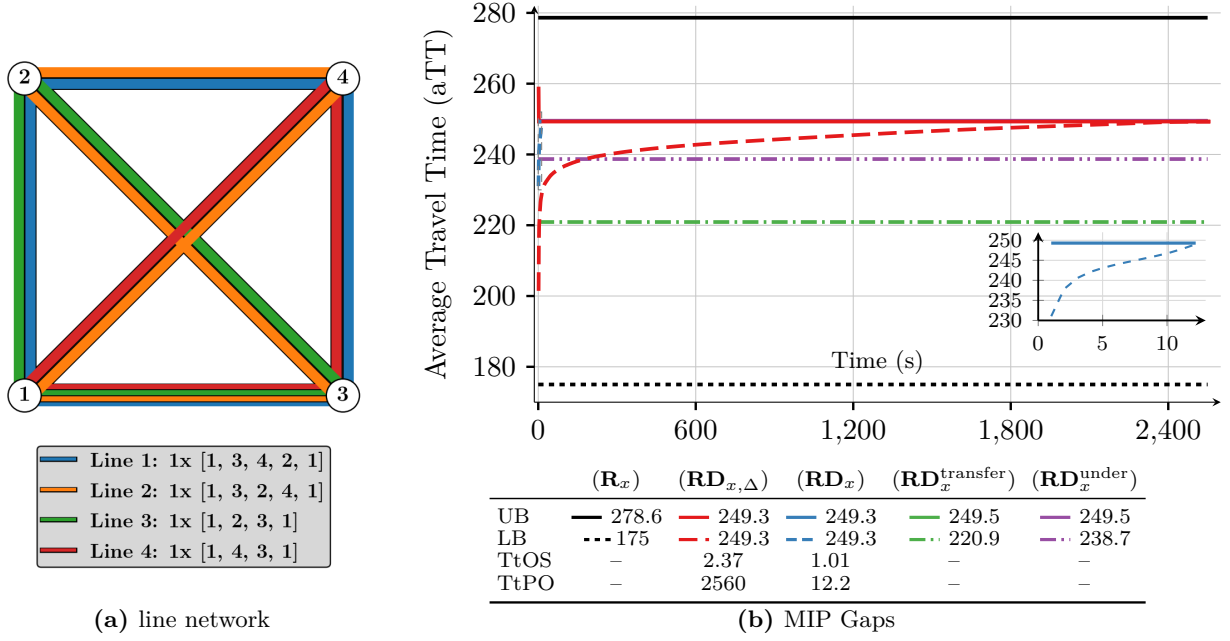

In Figure~\ref{fig:k_4_gaps}, the baseline yields an \emph{optimistic} expected average Travel Time (aTT) of 175 minutes; however, when the same routing is evaluated with handling operations, the realized aTT increases to 278.6 minutes. In contrast, $(\mathbf{RD}_{x,\Delta})$ and $(\mathbf{RD}_{x})$ that endogenize dwell times find and certify the optimal solution with aTT of 249.3 minutes, corresponding to a 10.4\% reduction relative to the realized baseline. The convex relaxations $(\mathbf{RD}_{x}^{\text{transfer}})$ and $(\mathbf{RD}_{x}^{\text{under}})$ return slightly suboptimal solutions with aTT of 249.5 minutes, with certified optimality gaps of 11.4\% and 4.3\%, respectively. For this instance, these gaps are driven by the lower bounds, since the upper bounds are already near-optimal. Comparing the two formulations solved by spatial branch-and-bound, the main difference lies in the time required to certify optimality. Formulation $(\mathbf{RD}_{x,\Delta})$ reaches the optimal solution in 2.37 seconds but requires roughly 40 additional minutes to prove optimality. By contrast, $(\mathbf{RD}_{x})$ reaches the optimum in 1 second and certifies optimality in 12.2 seconds (see the zoomed-in axis).

\begin{figure}[ht]
\centering
\captionsetup[subfigure]{font=footnotesize,skip=1pt}
\begin{subfigure}[b]{0.39\linewidth}
    \centering
    \raisebox{1em}{\includegraphics[width=0.88\linewidth]{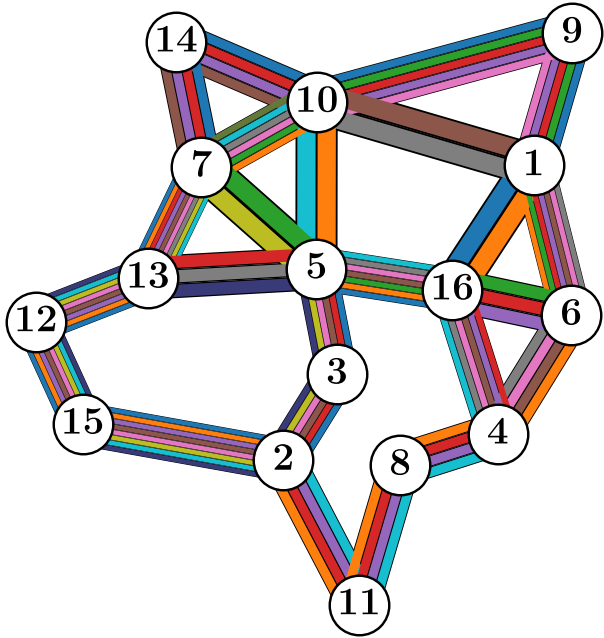}}
    \caption{line network}
    \label{fig:p_16_lc}
\end{subfigure}
\hfill
\begin{subfigure}[b]{0.6\linewidth}
    \begin{adjustbox}{max width=\linewidth}
\begin{tikzpicture}
\begin{axis}[
    width=\linewidth,
    height=0.55\linewidth,
    xmin=-50, xmax=3800,
    ymin=100, ymax=210,
    ylabel={Average Travel Time (aTT)},
    grid=both,
    grid style={line width=.1pt, draw=gray!30},
    major grid style={line width=.2pt, draw=gray!45},
    tick align=outside,
    xtick distance=600,
    xtick pos=bottom,
    axis lines=left,
    ytick pos=left,%
  xmajorgrids=true,
  tick style={major tick length=3pt, semithick, draw=black, inner sep=0pt},
  xtick align=inside,               %
  x tick label style={anchor=south},%
  xlabel={Time (s)},
	  xlabel style={
	    at={(axis description cs:0.5,0.15)}, %
	    anchor=north,                     %
	    yshift=9pt,                       %
	    font=\footnotesize
	  },
    legend style={draw=none}
    axis line style={line width=0.9pt},  %
    tick style={line width=0.9pt},
    every axis plot/.append style={line width=1.7pt},  %
]

\addplot[black] coordinates {(0,207) (3600,207)};
\addplot[black,lbBase] coordinates {(0,130) (3600,130)};
\addplot[socp] coordinates {(0,179.86) (3600,179.86)};
\addplot[socp,lbSOCP] coordinates {(0,159.59) (3600,159.59)};
\addplot[under] coordinates {(0,179.22) (3600,179.22)};
\addplot[under,lbUnder] coordinates {(0,164.85) (3600,164.85)};

\addplot[color=ArcOrig]
  table [x=time, y=ub, col sep=comma] {graphics/Figure_for_Planar_16/Arc_Orig_bounds.csv};
\addplot[color=ArcOrig, lbOrig]
  table [x=time, y=lb, col sep=comma] {graphics/Figure_for_Planar_16/Arc_Orig_bounds.csv};

\addplot[color=ArcNoDelta]
  table [x=time, y=ub, col sep=comma] {graphics/Figure_for_Planar_16/Arc_No_Delta_bounds.csv};
\addplot[color=ArcNoDelta, lbNoDelta]
  table [x=time, y=lb, col sep=comma] {graphics/Figure_for_Planar_16/Arc_No_Delta_bounds.csv};

\end{axis}

\end{tikzpicture}
\end{adjustbox}
\par\vspace{-0.55em}
{\centering
\scriptsize
\setlength{\tabcolsep}{1.35pt} %
\renewcommand{\arraystretch}{0.66} %
\begin{tabular}{lccccc}
\toprule
 & $(\mathbf{R}_{x})$ & $(\mathbf{RD}_{x,\Delta})$ & $(\mathbf{RD}_{x})$ & $(\mathbf{RD}_{x}^{\text{transfer}})$ & $(\mathbf{RD}_{x}^{\text{under}})$ \\
\midrule
UB & \raisebox{1.5pt}{\tikz{\draw[very thick, black] (0,0)--(0.4,0);}} 207
   & \raisebox{1.5pt}{\tikz{\draw[very thick, ArcOrig] (0,0)--(0.4,0);}} 207
   & \raisebox{1.5pt}{\tikz{\draw[very thick, ArcNoDelta] (0,0)--(0.4,0);}} 178
   & \raisebox{1.5pt}{\tikz{\draw[very thick, socp] (0,0)--(0.4,0);}} 179
   & \raisebox{1.5pt}{\tikz{\draw[very thick, under] (0,0)--(0.4,0);}} 179\\
LB & \raisebox{1.5pt}{\tikz{\draw[very thick, black, lbBase] (0,0)--(0.4,0);}} 130
   & \raisebox{1.5pt}{\tikz{\draw[very thick, ArcOrig,lbOrig] (0,0)--(0.4,0);}} 136
   & \raisebox{1.5pt}{\tikz{\draw[very thick, ArcNoDelta,lbNoDelta] (0,0)--(0.4,0);}} 144
   & \raisebox{1.5pt}{\tikz{\draw[very thick, socp,lbSOCP] (0,0)--(0.4,0);}} 160
   & \raisebox{1.5pt}{\tikz{\draw[very thick, under,lbUnder] (0,0)--(0.4,0);}} 165\\
Gap (\%) & 37.1 & 34.4 & 19.4 & 11.3 & 8.0\\
TtBS     & -- & 3572 & 31.3 &-- &--\\
\bottomrule
\end{tabular}
\par}
  \caption{MIP Gaps}
  \label{fig:p_16_gaps}
\end{subfigure}
\caption{MIP progress of different formulations for a candidate line network of $P_{16}$.}
\label{fig:planar_16_example}
\end{figure}
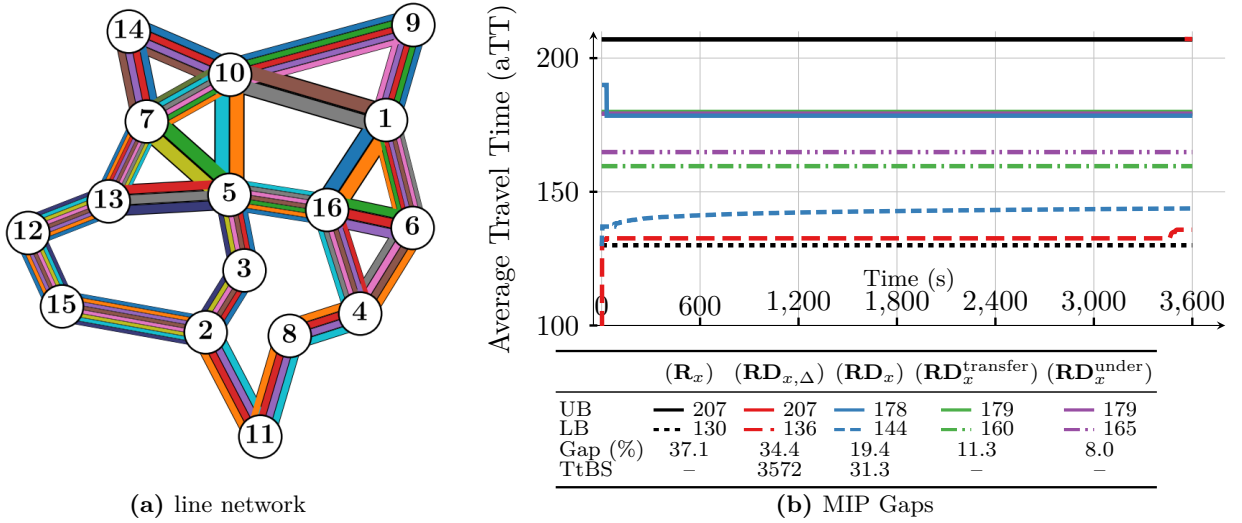
Figure~\ref{fig:planar_16_example} reports the same experiment for a line network on $P_{16}$. Since neither $(\mathbf{RD}_{x})$ nor $(\mathbf{RD}_{x,\Delta})$ certifies optimality within the one hour time limit, the accompanying table reports the best upper bound (UB), lower bound (LB), the resulting optimality gap (Gap), and the time in seconds to the best-known solution (TtBS).

For this instance and line network, the baseline produces a poor routing policy with an expected aTT of 130 minutes, but once handling operations are applied, the induced (realized) aTT increases to 207 minutes. In contrast, $(\mathbf{RD}_{x})$ identifies a stronger routing with aTT of 178 minutes within 31.3 seconds, and spends the remainder of the one-hour run tightening the dual bound, reaching a final gap of 19.4\%. Although the convex relaxations produce a slightly sub-par routing solution (UBs of 179 vs.\ 178), they produce much stronger lower bounds, reducing the final gaps to 11.3\% and 8.0\%. These lower bounds are obtained nearly instantly as the relaxations are solved as SOCPs in Gurobi.

To assess the consistency of these results, Table~\ref{tab:syn} reports statistics over 20 randomly generated line networks for each of the 15 synthetic PTNs under a one-hour time limit. For the baseline, we report the number of line networks that yield a feasible routing (Feas) and, when feasible, the percentage increase from the expected aTT to the induced aTT once handling times are applied (aG). For $(\mathbf{RD}_{x,\Delta})$ and $(\mathbf{RD}_{x})$, we report averages (over the 20 line networks) of: (i) the percentage improvement in aTT relative to the baseline (aI), (ii) the sub-optimality gap at the one-hour time limit (aG), and (iii) the time to best solution (TtBS, in seconds). For the convex relaxations, we report only aI and aG, and not the TtBS, as they are solved using Gurobi's barrier method in less than a second.

\begin{table}[htbp]
\centering
\caption{Performance of formulations over 20 randomly generated line networks for synthetic instances. -- indicates no feasible solution was found within the time limit.\label{tab:syn}}
\begingroup
\setlength{\tabcolsep}{3.0pt}%
\renewcommand{\arraystretch}{0.9}%
\small
\begin{tabular}{ccccccccccccc}
\toprule
PTN
& \multicolumn{2}{c}{$(\mathbf{R}_{x})$}
& \multicolumn{3}{c}{$(\mathbf{RD}_{x,\Delta})$}
& \multicolumn{3}{c}{$(\mathbf{RD}_{x})$}
& \multicolumn{2}{c}{$(\mathbf{RD}_{x}^{\text{transfer}})$}
& \multicolumn{2}{c}{$(\mathbf{RD}_{x}^{\text{under}})$} \\
\cmidrule(lr){2-3}\cmidrule(lr){4-6}\cmidrule(lr){7-9}\cmidrule(lr){10-11}\cmidrule(lr){12-13}
  & Feas & aG(\%) & aI(\%) & aG(\%) & TtBS & aI(\%) & aG(\%) & TtBS & aI(\%) & aG(\%) & aI(\%) & aG(\%) \\
\midrule
$K_{4}$     & 20 & 39.8 & 8.1   & 0.4  & 1.1   & 8.1  & 0.0 & 0.5 & 8.1 & 16.0 & 8.1 & 7.3 \\
$G_{2,2}$   & 20 & 45.8 & 8.7   & 0.0  & 0.6   & 8.7  & 0.0 & 0.1 & 8.5 & 20.8 & 8.6 & 9.7 \\
$P_{4}$     & 20 & 42.3 & 0.9   & 0.0  & 0.1   & 0.9  & 0.0 & 0.0 & 0.8 & 21.0 & 0.8 & 10.2 \\
$K_{6}$     & 15 & 35.2 & 11.1  & 10.8 & 10.3  & 11.1 & 4.3 & 1.0 & 9.9 & 11.8 & 10.3 & 5.3 \\
$G_{2,3}$   & 20 & 41.4 & 6.4   & 7.7  & 3.6   & 6.4  & 0.7 & 1.0 & 6.0 & 18.4 & 6.2 & 11.3 \\
$P_{6}$     & 20 & 44.9 & 6.9   & 1.8  & 2.1   & 6.9  & 0.0 & 0.8 & 6.1 & 21.1 & 6.2 & 11.7 \\
$K_{9}$     & 19 & 39.3 & 14.6  & 17.2 & 219.9 & 14.6 & 9.9 & 1.0 & 14.2 & 9.4 & 14.4 & 5.2 \\
$G_{3,3}$   & 15 & 43.5 & 13.2  & 16.5 & 55.8  & 13.2 & 7.5 & 1.1 & 12.7 & 18.0 & 13.0 & 12.2 \\
$P_{9}$     & 18 & 81.3 & 12.1  & 4.2  & 3.5   & 12.1 & 0.0 & 0.8 & 11.4 & 39.0 & 11.6 & 19.5 \\
$K_{12}$    & 11 & 42.6 & - & - & - & 17.1 & 18.3 & 1.2   & 16.2 & 8.9 & 16.6 & 4.7 \\
$G_{3,4}$   & 17 & 47.6 & - & - & - & 18.0 & 17.9 & 4.7   & 17.4 & 17.7 & 17.6 & 12.4 \\
$P_{12}$    & 11 & 53.5 & - & - & - & 27.0 & 18.5 & 9.4   & 26.3 & 17.9 & 26.8 & 12.8 \\
$K_{16}$    & 12 & 40.2 & - & - & - & 18.1 & 17.7 & 3.3   & 17.6 & 7.5 & 17.9 & 4.7 \\
$G_{4,4}$   & 17 & 48.6 & - & - & - & 21.3 & 18.2 & 16.1  & 20.9 & 13.9 & 21.1 & 10.5 \\
$P_{16}$    & 16 & 44.1 & - & - & - & 20.7 & 18.4 & 175.0 & 20.3 & 12.1 & 20.5 & 8.6 \\
\midrule
\textbf{Avg} & - & 46.0 & - & - & - & 13.6 & 8.8 & - & 13.1 & 16.9 & 13.3 & 9.7 \\
\bottomrule
\end{tabular}
\endgroup
\end{table}
Table~\ref{tab:syn} reveals several consistent trends. First, for a non-negligible number of line networks, the baseline routing is infeasible once handling times are accounted for, while the endogenous formulations still find feasible routing solutions. Figure~\ref{fig:baseline_infeasibility_examples} in Appendix~\ref{app:frontier} illustrates line concepts that induce feasible and infeasible baselines. Second, the pure-flow formulation $(\mathbf{RD}_{x})$ outperforms the bilinear formulation $(\mathbf{RD}_{x,\Delta})$ in terms of sub-optimality gaps, time to best solution, and scalability where for the 12- and 16-node instances, $(\mathbf{RD}_{x,\Delta})$ even fails to find a feasible solution within the one-hour time limit. We note that while the table reports the average improvement, the pure flow formulation consistently outperformed the bilinear formulation across all 300 line networks.

Third, the convex relaxations provide optimality gaps averaging at 16.9\% and 9.7\% across instances and appear largely independent of instance size. Already for the 12- and 16-node instances, the dual bounds from the convex relaxations are substantially tighter than the dual bound obtained by Gurobi’s spatial branch-and-bound after one hour. Fourth, Gurobi consistently finds slightly better primal solutions than the convex relaxations and typically does so quickly; however, it spends the remaining time limit just marginally improving the dual bound. This behavior is typical of spatial branch-and-bound for nonconvex problems, where heuristics often identify near-optimal incumbents early, and the remaining time is spent on certification by improving the dual bound. This motivated the frontier analysis of Section~\ref{sec:convex_frontier_bound}. Finally, across instances, endogenously accounting for handling times improves routing quality by an average of 13.6\%, with larger instances exhibiting larger improvements.

\subsubsection{Real world instances}

Figure~\ref{fig:sioux_falls_example} shows a candidate line network for the Sioux Falls PTN (Figure~\ref{fig:SF_LC_1}) and the bound progressions of all formulations (Figure~\ref{fig:sf_gaps}). The baseline aTT is 375 minutes, obtained by solving $(\mathbf{R}_{x})$. Within a one-hour time limit, $(\mathbf{RD}_{x,\Delta})$ does not find a feasible solution. In contrast, $(\mathbf{RD}_{x})$ finds a feasible solution in under 10 seconds with aTT of 303 minutes, but the remainder of the run is spent tightening the dual bound, ending with a 19.8\% gap after one hour. Extending the run to 24 hours does not improve the primal solution, and the dual bound improves only marginally (to an 18.2\% gap). The convex relaxations $(\mathbf{RD}_{x}^{\text{transfer}})$ and $(\mathbf{RD}_{x}^{\text{under}})$ produce slightly worse primal solutions, with aTT 305 and 304 minutes, but substantially stronger dual bounds, with final gaps of 9.0\% and 6.8\%, respectively.

\begin{figure}[ht]
\centering
\captionsetup[subfigure]{font=footnotesize,skip=1pt}
\begin{subfigure}[b]{0.35\linewidth}
    \centering
    \raisebox{1em}{\includegraphics[width=0.8\linewidth]{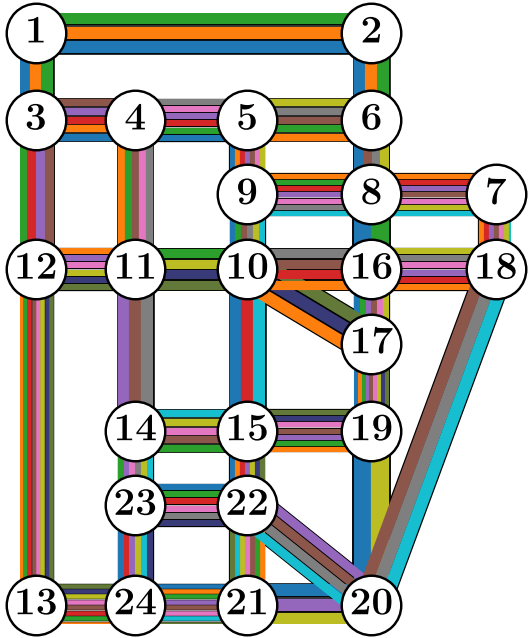}}
    \caption{line network}
    \label{fig:SF_LC_1}
\end{subfigure}
\hfill
\begin{subfigure}[b]{0.64\linewidth}
    \begin{adjustbox}{max width=\linewidth}
\begin{tikzpicture}
\begin{axis}[
    width=\linewidth,
    height=0.55\linewidth,
    xmin=-50, xmax=3620,
    ymin=180, ymax=380,
    ylabel={Average Travel Time (aTT)},
    grid=both,
    axis lines=left,
    ytick pos=left,%
    grid style={line width=.1pt, draw=gray!30},
    major grid style={line width=.2pt, draw=gray!45},
    tick align=outside,
    xtick distance=600,
    xtick pos=bottom,                 %
  xmajorgrids=true,
  tick style={major tick length=3pt, semithick, draw=black, inner sep=0pt},
  xtick align=inside,               %
  x tick label style={anchor=south},%
  xlabel={Time (s)},
	  xlabel style={
	    at={(axis description cs:0.5,0.13)}, %
	    anchor=north,                     %
	    yshift=8pt,                       %
	    font=\footnotesize
	  },
    legend style={draw=none}
    axis line style={line width=0.9pt},  %
    tick style={line width=0.9pt},
    every axis plot/.append style={line width=1.7pt},  %
]

\addplot[black] coordinates {(0,375) (3600,375)};
\addplot[black,lbBase] coordinates {(0,221) (3600,221)};
\addplot[socp] coordinates {(0,305.03) (3600,305.03)};
\addplot[socp, lbSOCP] coordinates {(0,277.56) (3600,277.56)};
\addplot[under] coordinates {(0,304.782) (3600,304.782)};
\addplot[under, lbUnder] coordinates {(0,284.175) (3600,284.175)};

\addplot[color=ArcNoDelta]
  table [x=time, y=ub, col sep=comma] {graphics/Figure_for_sioux_falls/Arc_No_Delta_bounds.csv};
\addplot[color=ArcNoDelta, lbNoDelta]
  table [x=time, y=lb, col sep=comma] {graphics/Figure_for_sioux_falls/Arc_No_Delta_bounds.csv};
\end{axis}

\end{tikzpicture}
\end{adjustbox}
\par\vspace{-0.55em}
{\centering
\scriptsize
\setlength{\tabcolsep}{1.35pt} %
\renewcommand{\arraystretch}{0.66} %
\begin{tabular}{lccccc}
  \toprule
   & $(\mathbf{R}_{x})$ & $(\mathbf{RD}_{x,\Delta})$ & $(\mathbf{RD}_{x})$ & $(\mathbf{RD}_{x}^{\text{transfer}})$ & $(\mathbf{RD}_{x}^{\text{under}})$  \\
  \midrule
  UB & \raisebox{2pt}{\tikz{\draw[very thick, black] (0,0)--(0.4,0);}} 375
     & -
     & \raisebox{2pt}{\tikz{\draw[very thick, ArcNoDelta] (0,0)--(0.4,0);}} 303
     & \raisebox{2pt}{\tikz{\draw[very thick, socp] (0,0)--(0.4,0);}} 305
     & \raisebox{2pt}{\tikz{\draw[very thick, under] (0,0)--(0.4,0);}} 304 \\
  LB & \raisebox{2pt}{\tikz{\draw[very thick, black, lbBase] (0,0)--(0.4,0);}}221
     & -
     & \raisebox{2pt}{\tikz{\draw[very thick, ArcNoDelta,lbNoDelta] (0,0)--(0.4,0);}} 243
     & \raisebox{2pt}{\tikz{\draw[very thick, socp,lbSOCP] (0,0)--(0.4,0);}} 277.6
     & \raisebox{2pt}{\tikz{\draw[very thick, under,lbUnder] (0,0)--(0.4,0);}} 284.2 \\
  Gap (\%) & 41.0 & -  & 19.8 & 9.0 & 6.8 \\
  TtBS & - & -  & 9.43 & 0.05 & 0.05 \\
	  \bottomrule
	    \end{tabular}
\par}
  \caption{MIP Gaps}
  \label{fig:sf_gaps}
\end{subfigure}
\caption{MIP progress of different formulations for a candidate line network of the Sioux Falls PTN.}
\label{fig:sioux_falls_example}
\end{figure}
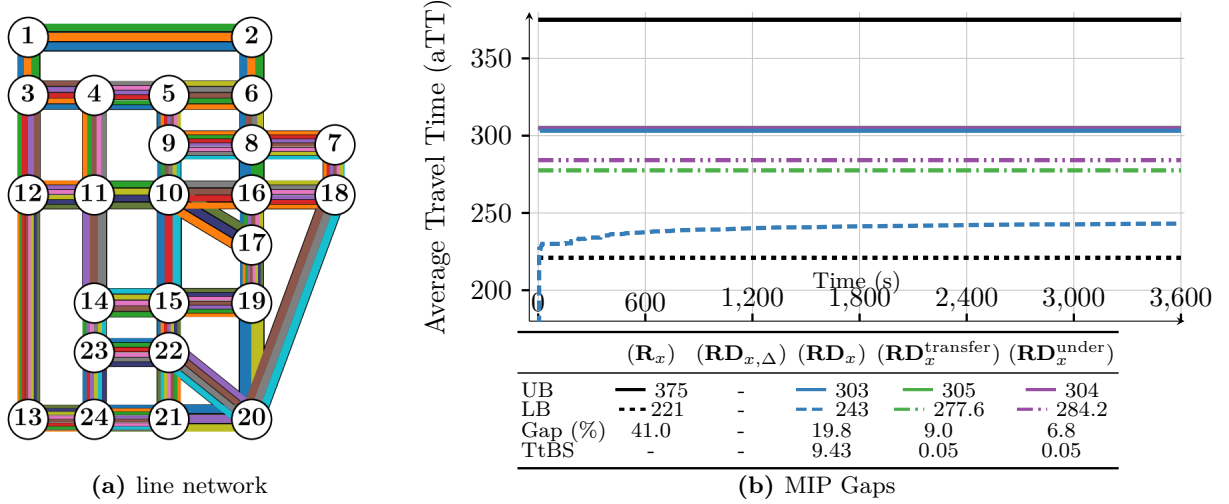

Figure~\ref{gob} shows a candidate line network with 25 unique lines for the G\"oVB PTN. Although the lines are drawn as simple paths, each represents a bi-directional line that turns back after reaching its terminus. For this instance, the baseline yields an optimistic expected aTT of 429.5 minutes while the induced (realized) aTT is 1105.9 minutes, representing a 61.2\% gap. Neither $(\mathbf{RD}_{x,\Delta})$ nor $(\mathbf{RD}_{x})$ finds a feasible solution within a 24-hour time limit. In contrast, the convex relaxation produces a valid lower bound of 712.3 and an upper bound of 812.9, corresponding to a 12.4\% optimality gap. Relative to the induced baseline aTT, the relaxation upper bound improves routing quality by 26.5\%. Finally, formulation $(\mathbf{RD}_{x}^{\text{under}})$ strengthens the lower bound to 725.6, reducing the optimality gap to 10.7\%.

Table~\ref{tab:real_world} reports results over 10 line networks per real-world instance. For Sioux Falls, handling-aware routing improves aTT by 18.1\% on average, and $(\mathbf{RD}_{x}^{\text{under}})$ certifies a 6.9\% average gap.

For the larger G\"oVB instance (257 nodes, 25 lines), $(\mathbf{RD}_{x})$ does not find a feasible primal solution within the one-hour time limit for any of our 10 instances. Moreover, the baseline $(\mathbf{R}_{x})$ often yields an infeasible routing solution after accounting for handling times. Specifically, out of the 10 random line networks that we randomly generated, only the one depicted in Figure~\ref{gob} is such that the baseline admitted a feasible routing solution. For this particular instance, the routing solution obtained from the convex relaxations represents an improvement of 26.5\% over the solution induced by $(\mathbf{R}_{x})$.

\begin{figure}[htbp]
  \centering
  \begin{minipage}[t]{0.36\linewidth}
    \centering
    \vspace{0pt}
    \includegraphics[width=\linewidth]{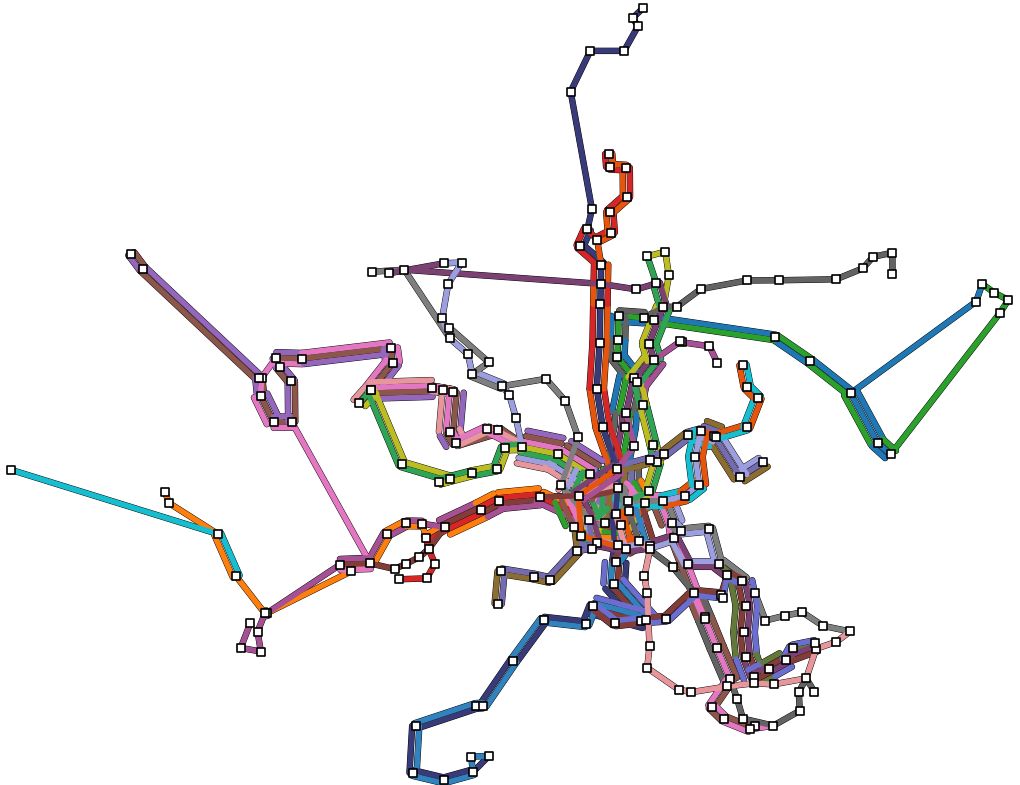}
    \vspace{-0.55em}
    \captionof{figure}{G\"oVB line network.}
    \label{gob}
  \end{minipage}
  \hspace{0.012\linewidth}
  \begin{minipage}[t]{0.615\linewidth}
    \centering
    \vspace{0pt}
    \captionof{table}{Real-world formulation performance.}
    \label{tab:real_world}
    \vspace{0.25em}
    \begingroup
    \setlength{\tabcolsep}{2.5pt}%
    \renewcommand{\arraystretch}{0.86}%
    \scriptsize
    \begin{tabular}{lccccccccc}
    \toprule
    PTN
    & \multicolumn{2}{c}{$(\mathbf{R}_{x})$}
    & \multicolumn{3}{c}{$(\mathbf{RD}_{x})$}
    & \multicolumn{2}{c}{$(\mathbf{RD}_{x}^{\text{transfer}})$}
    & \multicolumn{2}{c}{$(\mathbf{RD}_{x}^{\text{under}})$} \\
    \cmidrule(lr){2-3}\cmidrule(lr){4-6}\cmidrule(lr){7-8}\cmidrule(lr){9-10}
      & Feas & aG(\%) & aI(\%) & aG(\%) & TtBS & aI(\%) & aG(\%) & aI(\%) & aG(\%) \\
    \midrule
    $SF$   & 10 & 41.4 & 18.1 & 20.7 & 55.7  & 17.9 & 9.6 & 18.0 & 6.9 \\
    G\"oVB & 1 & 61.2 & - & - & - & 26.5 & 13.0 & 26.5 & 11.3 \\
    \bottomrule
    \end{tabular}
    \endgroup
  \end{minipage}
\end{figure}

\subsection{Frontier bound}
\label{ssec:frontier_bound_comp}
Across the synthetic and real-world instances above, the SOCP relaxations consistently produced primal solutions only slightly worse than those found by Gurobi, while providing substantially stronger lower bounds. On the small instances where Gurobi certifies optimality, these results show that the relaxation incumbents are much closer to optimal than their SOCP certificates suggest. This observation motivates the frontier bound developed in Section~\ref{sec:convex_frontier_bound}. Figure~\ref{fig:frontier_residual_staircase} illustrates the adaptive staircase construction on a $3\times 3$ grid instance. The initial staircase is shown in blue, and the refined staircase is shown in green. Table~\ref{tab:frontier_gaps} reports the resulting optimality gaps across the benchmark instances.

\begin{figure}[htbp]
  \centering
  \makebox[\textwidth][c]{%
  \begin{minipage}[t]{0.43\textwidth}
    \centering
    \vspace{0pt}
    \includegraphics[width=0.8\linewidth]{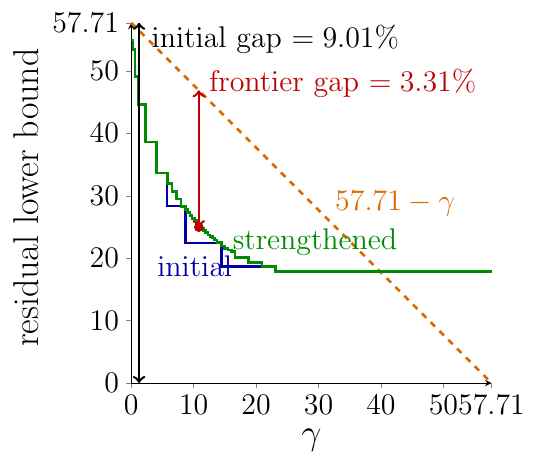}
    \vspace{-0.35em}
    \captionof{figure}{Frontier staircase bound.}
    \label{fig:frontier_residual_staircase}
  \end{minipage}%
  \hspace{0.045\textwidth}%
  \begin{minipage}[t]{0.43\textwidth}
    \centering
    \vspace{0pt}
    \captionof{table}{Optimality gaps (\%).}
    \label{tab:frontier_gaps}
    \vspace{0.25em}
    \begingroup
    \setlength{\tabcolsep}{3.0pt}%
    \renewcommand{\arraystretch}{0.9}%
    \scriptsize
    \begin{tabular}{lccc}
    \toprule
    Instance & $(\mathbf{RD}_{x}^{\text{transfer}})$ & $(\mathbf{RD}_{x}^{\text{under}})$ & Frontier \\
    \midrule
    $K_{4}$  & 11.41 & 4.38 & \textbf{0.06} \\
    $K_{6}$  & 11.80 & 5.78 & \textbf{0.14} \\
    $K_{9}$  & 9.01  & 5.84 & \textbf{3.31} \\
    $K_{12}$ & 8.88  & 4.92 & \textbf{3.56} \\
    $K_{16}$ & 6.72  & 4.55 & \textbf{3.91} \\
    SF       & 8.99  & 6.80 & \textbf{3.61} \\
    G\"oVB    & 13.0  & 11.3 & \textbf{5.11} \\
    \bottomrule
    \end{tabular}
    \endgroup
  \end{minipage}%
  }
\end{figure}

On the smallest instances, the frontier bound nearly closes the gap, certifying final gaps of 0.06\% and 0.14\%. As instances grow, the final gaps increase to 3--5\%, but remain consistently smaller than the gaps from both SOCP relaxations. Finally, passing the slightly suboptimal SOCP incumbent to a local NLP solver recovers the same primal solution found by Gurobi, including the solutions that Gurobi certifies as optimal on small instances. These results support the alignment intuition behind the frontier analysis that the residual term is strongly aligned with the convex objective and only modestly shifts the minimizer.

\subsection{Sensitivity analysis of handling times}
\label{ssec:congestion}
We next vary the per-parcel loading and unloading times, $\alpha$ and $\beta$, to assess how handling intensity affects routing quality, relaxation strength, and path choice. Table~\ref{tab:alpha_bounds_congestion} reports results for the G\"oVB instance in Figure~\ref{gob}, comparing the baseline $(\mathbf{R}_{x})$ with $(\mathbf{RD}_{x}^{\text{under}})$. For each setting, we report lower and upper bounds, the optimality gap, and the fraction of each period spent handling rather than driving.

\begin{table}[t]
\centering
\begingroup
\setlength{\tabcolsep}{4.5pt}%
\renewcommand{\arraystretch}{1.05}%
\small
\begin{tabular}{c cccc cccc}
\toprule
$\alpha \& \beta$ &
\multicolumn{4}{c}{$(\mathbf{R}_{x})$} &
\multicolumn{4}{c}{$(\mathbf{RD}^{\mathrm{under}}_{x})$} \\
\cmidrule(lr){2-5}\cmidrule(lr){6-9}
& LB & UB & Gap & Handling (\%) & LB & UB & Gap & Handling (\%) \\
\midrule
0.1    & 429.5     &481.4     &10.8     & 8.1          & 467.4    & 478.2    & 2.26    & 8.0           \\
0.2    & 429.5     &550.2     &21.9     & 16.0          & 512.5    & 535.8    & 4.35     & 16.0           \\
0.3    & 429.5     &647.5     &33.7     & 24.2          & 566.8    & 605.5    & 6.40    & 24.0          \\
0.4    & 429.5     &800.9     &46.4     & 32.2          & 635.3     & 694.7    & 8.55     & 32.0           \\
0.5    & 429.5     &1105.9     &61.2     & 40.2           & 725.6    & 813.2    & 10.8    & 40.0          \\
0.6    & 429.5     & --     & --     & --           & 852.7    & 980.7    & 13.1    & 48.1          \\
0.7    & 429.5     & --     & --    & --           & 1047.3    & 1238.9    & 15.5     & 56.4          \\
0.8  & -- & -- & --& -- & --& --& --& -- \\
\bottomrule
\end{tabular}
\caption{Performance for increasing handling times for the G\"oVB PTN shown in Figure~\ref{gob}.}
\label{tab:alpha_bounds_congestion}
\endgroup
\end{table}

Two patterns emerge. First, the value of modeling handling increases with $\alpha$ and $\beta$: the baseline gap widens and eventually becomes infeasible, while $(\mathbf{RD}_{x}^{\text{under}})$ continues to find feasible routings with moderate gaps. Second, the relaxation gap also increases with handling intensity because the nonconvex residual becomes more important. For example, the gap is 2.26\% when vehicles spend 8\% of the period handling, and 15.5\% when handling reaches 56.4\%. At $\alpha=\beta=0.8$, RLN-HT is infeasible regardless of the routing model.

The improvement is not driven by reducing total handling time. The handling share is nearly unchanged across the baseline and endogenous solutions, suggesting that the benefit comes from better spatial dispersion of demand. To isolate this effect, Figure~\ref{fig:last_fig} considers the Sioux Falls line network with infinite vehicle capacity. As $\alpha$ and $\beta$ increase, more OD pairs split demand across multiple paths, and the resulting aTT grows much more slowly than under the baseline routing. This reflects the role of dwell time as an endogenous, costly capacity expansion: rerouting some flow can be better than adding more dwell time on already-used lines.

\begin{figure}[htbp]
    \centering
    \begin{subfigure}[t]{0.49\linewidth}
        \centering
        \includegraphics[width=\linewidth]{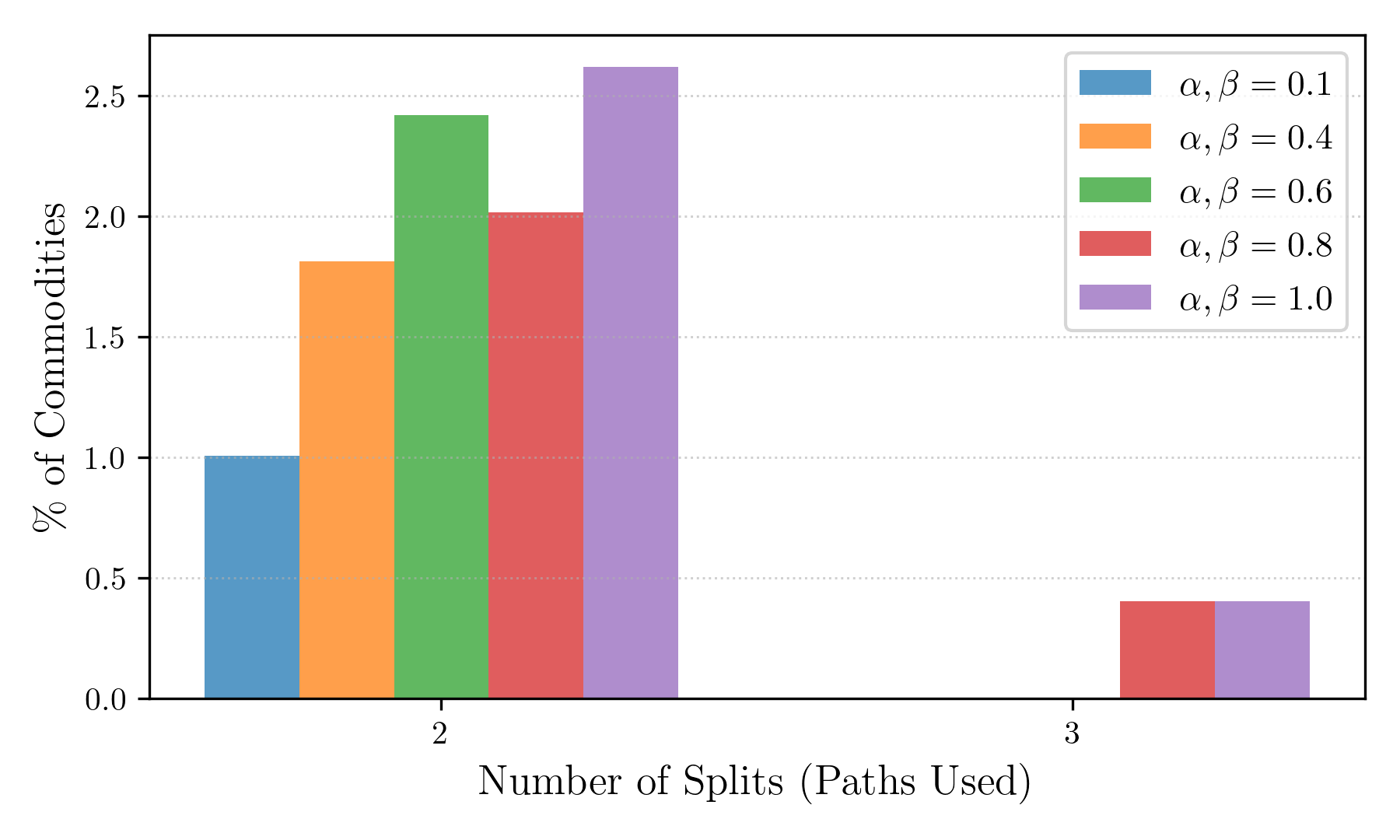}
        \caption{Path distribution}
        \label{fig:path_distribution}
    \end{subfigure}
    \begin{subfigure}[t]{0.49\linewidth}
        \centering
        \includegraphics[width=\linewidth]{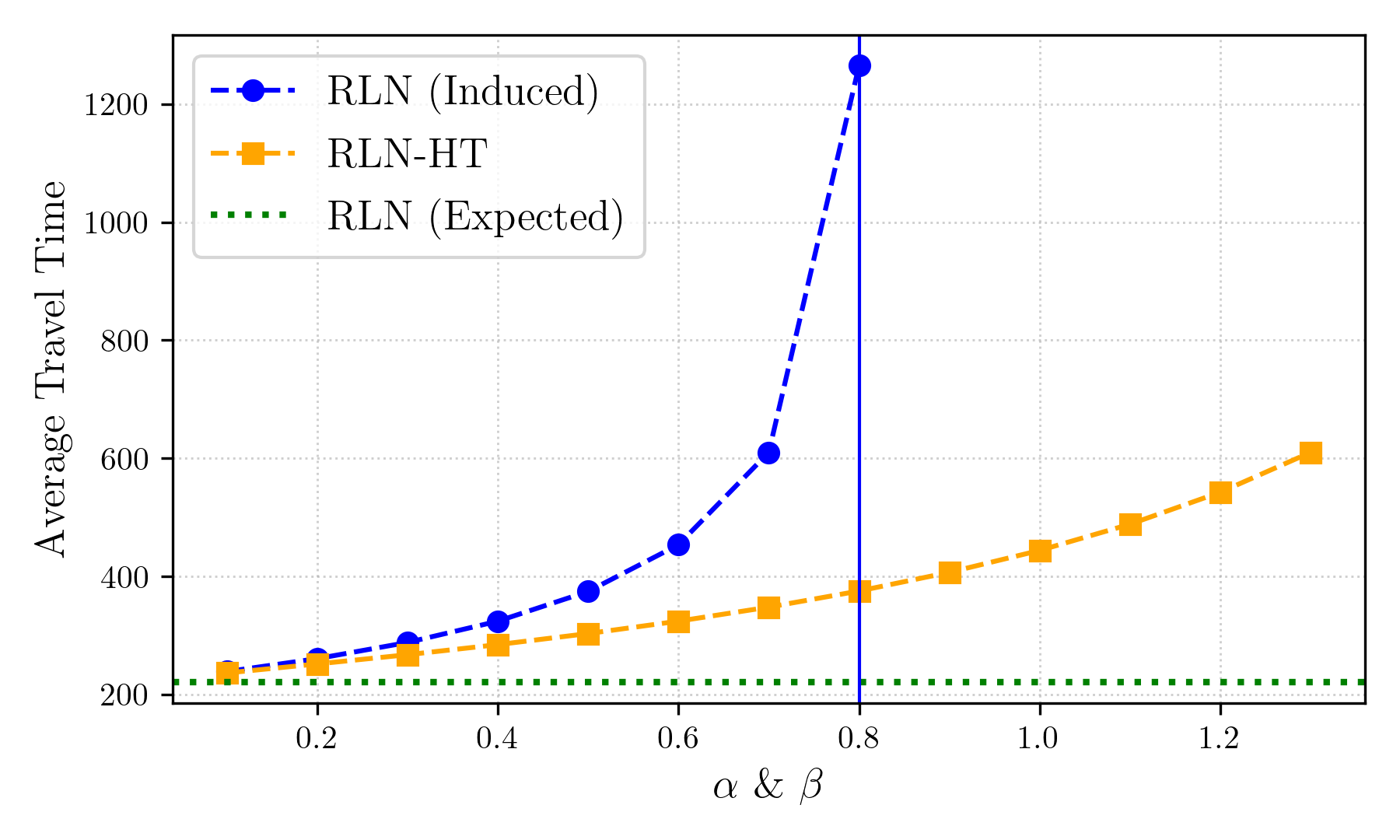}
        \caption{Average travel times}
        \label{fig:RLNHT_benefit}
    \end{subfigure}
    \caption{Distribution of paths and aTT for increasing per-parcel handling times $\alpha$ \& $\beta$ values for the line network in Figure~\ref{fig:sioux_falls_example} (Sioux Falls).}
    \label{fig:last_fig}
\end{figure}

%% file: sections/07-conclusion.tex
\section{Concluding Remarks}
\label{sec:conclusion}
Motivated by hyperconnected urban parcel networks, we studied the routing oracle that arises after a set of recurring inter-hub services has been fixed. We introduced RLN-HT as a fixed-line routing problem with endogenous handling times, formulated it over a change-and-go network, and used optimality conditions to reduce it to a pure multicommodity flow formulation. The resulting model is nonconvex and nonseparable, but its objective decomposes into a convex component and a structured nonconvex residual. Exploiting this structure, we developed two SOCP relaxations that preserve the original routing feasible region and therefore return feasible routing solutions together with valid lower bounds.

Computational experiments on 15 synthetic and 2 real-world PTNs show that RLN-HT is a difficult nonconvex routing problem, but also that it can serve effectively as a scalable oracle. The SOCP relaxations solve quickly, scale to instances with up to 257 nodes and 25 lines, and typically return near-optimal primal solutions with immediate certificates in the 5--10\% range. The residual-frontier analysis further exploits the empirical alignment between the convex objective and the residual term, reducing certificates to 5.11\% or below on the tested instances when additional endpoint-bound computations are performed. The experiments also show that ignoring handling times can produce poor or infeasible routings, while endogenously accounting for handling improves routing quality by 26.5\% on real-world instances, primarily by redistributing demand across multiple paths.

Three directions for future work are especially promising. First, the oracle developed here can be embedded within the full line-network design problem for urban parcel systems, where both line networks and parcel routes must be chosen. Second, it would be valuable to develop scalable oracles for online routing and timetabling, and to incorporate frequency-dependent operating costs such as fuel consumption. Such costs introduce a tradeoff between service quality and operating expenditure that may favor dwell times beyond handling requirements, requiring further analysis to obtain a pure-flow reformulation. Third, the residual-frontier framework suggests a broader approach for certifying problems with aligned convex and residual objective components. Understanding when this structure arises and how to strengthen the resulting frontier bounds may yield useful lower-bound methods beyond RLN-HT.

%% file: appendices.tex
\section{Proofs}\label{app:proofs}
\begin{proof}[Proof of Proposition \ref{prop:1}]
Let $(\mathbf{x}^*,\Delta^*)$ be an optimal solution to $(\mathbf{RD}_{x,\Delta})$. Keeping $\mathbf{x}^*$ fixed, choose $\Delta^*$ to minimize total dwell time among dwell-time vectors preserving feasibility and optimality. Suppose that for some line $l\in\mathcal{L}$ and node $i\in\mathcal{N}(l)$, constraint~\eqref{eq:RD_form_1_const_4} has slack, that is,
\[
\sum_{k\in\mathcal{D}}\bigl(\alpha t_i^{*kl}+\beta u_i^{*kl}\bigr)
<
\Delta_i^{*l}\, a_l \frac{T}{C_l+\sum_{j\in\mathcal{N}(l)}\Delta_j^{*l}} .
\]
In particular, the right-hand side is positive, so $\Delta_i^{*l}>0$. Consider a perturbation that decreases $\Delta_i^{*l}$ slightly by setting
\[
\hat{\Delta}_i^{l}=\Delta_i^{*l}-\epsilon,\qquad \epsilon\in(0,\Delta_i^{*l}),
\]
and keep all other variables unchanged, i.e., $(\hat{\mathbf{x}},\hat{\Delta})=(\mathbf{x}^*,\Delta^*)$ except at $(i,l)$.

This perturbation does not affect flow conservation. Moreover, any constraints whose right-hand side depends on the line cycle time through
$C_l+\sum_{j\in\mathcal{N}(l)}\Delta_j^{l}$
become no tighter, since this denominator decreases when $\Delta_i^{l}$ is reduced. For constraint~\eqref{eq:RD_form_1_const_4} at $(i,l)$, the left-hand side is unchanged, while the right-hand side varies continuously in $\Delta_i^{l}$. Since the constraint is slack at $(\mathbf{x}^*,\Delta^*)$, it remains satisfied for all sufficiently small $\epsilon>0$. Hence, $(\hat{\mathbf{x}},\hat{\Delta})$ is feasible.

Finally, the objective is nondecreasing in each dwell time, so this perturbation cannot increase its value. The perturbed solution therefore remains optimal but has strictly smaller total dwell time, contradicting our choice of $\Delta^*$. Thus, all handling constraints are tight in the chosen optimal solution.

\end{proof}

\begin{proof}[Proof of Proposition \ref{prop:2}]
For necessity, suppose $\boldsymbol{\Delta}$ is feasible. Write $H_l=\sum_{i\in\mathcal{N}(l)} h_i^l$ and $f_l=Ta_l/(C_l+\sum_{i\in\mathcal{N}(l)}\Delta_i^l)$. Summing the handling constraints gives
\[
H_l\le f_l\sum_i\Delta_i^l =Ta_l-C_lf_l<Ta_l.
\]
Hence $f_l\le (Ta_l-H_l)/C_l$, and the vehicle-capacity constraints imply \eqref{eq:new_cap_const_prop}. Conversely, suppose
\eqref{eq:new_const_prop}--\eqref{eq:new_cap_const_prop} hold. Consider constraint~\eqref{eq:RD_form_1_const_4} at equality and sum it over a line $l \in \mathcal{L}$:
\begin{align}
\sum_{i \in \mathcal{N}(l)} \sum_{k \in \mathcal{D}} \left(\alpha t_{i}^{kl} + \beta u_{i}^{kl}\right) = \sum_{i \in \mathcal{N}(l)}\Delta_{i}^{l} a_{l} \frac{T}{C_{l} + \sum_{j \in \mathcal{N}(l)} \Delta_{j}^{l}}
\label{eq:appendix_dwell_sum}
\end{align}
Factoring out $\sum_{q \in \mathcal{N}(l)} \Delta_{q}^{l}$ gives the total dwell time along the line:
\begin{align*}
\sum_{q \in \mathcal{N}(l)} \Delta_{q}^{l} = \frac{C_{l} H_l}{T a_{l} - H_l}.
\end{align*}
This expression is valid only when the denominator is strictly positive, which gives condition~\eqref{eq:new_const_prop}. Substituting the total dwell time back into the equality for hub $i \in \mathcal{N}(l)$ gives
\begin{align*}
    h_{i}^{l} = \Delta_{i}^{l} \frac{Ta_{l}}{C_{l} + \frac{C_{l} H_l}{T a_{l} - H_l}} = \Delta_{i}^{l} \frac{Ta_{l} - H_l}{C_{l}}.
\end{align*}
Rearranging yields~\eqref{eq:opt_delta}. Substituting the same expression for the line frequency into the vehicle-capacity constraints gives~\eqref{eq:new_cap_const_prop}. For any feasible dwell-time vector, $\Delta_i^l\ge h_i^l/f_l \ge C_lh_i^l/(Ta_l-H_l)=\Delta_i^{*l}$. Since the objective is nondecreasing in each dwell time, the displayed vector minimizes total travel time for the fixed flow.
\end{proof}

\begin{proof}[Proof of Proposition \ref{prop:x_lb}]
Let $x^\star=(d^\star,t^\star,u^\star)$ be optimal. We first show that one never boards and alights simultaneously for the same $(i,k,l)$. Suppose $t_{i}^{kl\,\star}>0$ and $u_{i}^{kl\,\star}>0$ for some $(i,k,l)$, and set $\varepsilon=\min\{t_{i}^{kl\,\star},u_{i}^{kl\,\star}\}>0$. Define $\tilde x$ by decreasing both variables at $(i,k,l)$, i.e., $\tilde t_{i}^{kl}=t_{i}^{kl\,\star}-\varepsilon$, $\tilde u_{i}^{kl}=u_{i}^{kl\,\star}-\varepsilon$, and leaving all other variables unchanged.

This preserves feasibility. The node-balance and line-flow equalities depend on $t_{i}^{kl}-u_{i}^{kl}$, which is unchanged. The dwell-time inequality depends on $\sum_{i,k}(\alpha t_i^{kl}+\beta u_i^{kl})$ and therefore becomes looser. The capacity inequality depends on the same sum of $t$-variables and its right-hand side weakly increases when $t_{i}^{kl}$ is reduced. Nonnegativity holds by the choice of $\varepsilon$.

Moreover, the objective is strictly increasing in each $t_{i}^{kl}$ and $u_{i}^{kl}$ on the feasible set, since these variables enter the handled-volume terms $H_l=\sum_{i,k}(\alpha t_i^{kl}+\beta u_i^{kl})$ with positive coefficients, and the corresponding cost terms are increasing in $H_l$ (the linear travel cost is independent of $t$ and $u$). Hence $F(\tilde x)<F(x^\star)$, contradicting optimality. Therefore, for all $(i,k,l)$, $t_{i}^{kl\,\star}>0$ implies $u_{i}^{kl\,\star}=0$, and $u_{i}^{kl\,\star}>0$ implies $t_{i}^{kl\,\star}=0$.

Fix any $(i,k,l)$. Flow conservation~\eqref{eq:RD_form_3_const_2} gives
$d_{i_- i}^{kl\,\star}+t_{i}^{kl\,\star}=d_{i i_+}^{kl\,\star}+u_{i}^{kl\,\star}$.
If $u_{i}^{kl\,\star}>0$, then $t_{i}^{kl\,\star}=0$ and thus
$d_{i_- i}^{kl\,\star}=d_{i i_+}^{kl\,\star}+u_{i}^{kl\,\star}\ge u_{i}^{kl\,\star}$.
If $u_{i}^{kl\,\star}=0$, the inequality $d_{i_- i}^{kl\,\star}\ge u_{i}^{kl\,\star}$ is trivial by nonnegativity. Hence $d_{i_- i}^{kl\,\star}\ge u_{i}^{kl\,\star}$ for all $(i,k,l)$, and summing over $k$ yields $X_{i}^{l\,\star}\ge U_{i}^{l\,\star}$ for all $i\in\mathcal{N}(l)$ and $l\in\mathcal{L}$.
\end{proof}

\begin{proof}[Proof of Theorem \ref{thm:frontier_gap}]
For any $x \in \Xset$, define $\gamma(x) \coloneqq g(x) - g_{c} \geq 0$. Then $x$ is feasible for the definition of $m(\gamma(x))$, and therefore
\[
r(x) \geq m(\gamma(x)).
\]
Thus
\[
F^{*} \geq g_{c} + \min_{\gamma \geq 0} \{\gamma + m(\gamma)\}.
\]
Conversely, for any $\gamma \geq 0$ and any $\eta>0$, choose an $\eta$-optimal solution $x_{\gamma}^{\eta}$ for $m(\gamma)$, so that
$r(x_{\gamma}^{\eta}) \le m(\gamma)+\eta$ and $g(x_{\gamma}^{\eta})\le g_c+\gamma$. Then
\[
F(x_{\gamma}^{\eta}) = g(x_{\gamma}^{\eta}) + r(x_{\gamma}^{\eta}) \leq g_{c} + \gamma + m(\gamma)+\eta.
\]
Letting $\eta\downarrow0$ and minimizing over $\gamma$ gives the reverse inequality. Since $x_{c}$ is feasible, $F^{*} \leq F(x_c)=g_{c} + r_{c}$. If $\gamma > r_{c}$, then $\gamma + m(\gamma) \geq \gamma > r_{c}$, so such values of $\gamma$ cannot improve on the incumbent. Hence the minimization can be restricted to $[0,r_{c}]$. Subtracting from $F_{c} = g_{c} + r_{c}$ gives the second result.
\end{proof}

\clearpage

\section{Infeasibility and Frontier Bound Details}\label{app:frontier}

\paragraph{Baseline infeasibility}
The handling-ignorant baseline can become infeasible once handling times are imposed because handling affects the effective capacity of each line. Figure~\ref{fig:baseline_feasible} shows a $P_{16}$ line network for which the baseline routing remains feasible after handling times are accounted for. In this case, the handling workload induced by the baseline routing remains sufficiently dispersed, so all temporal and effective-capacity constraints are satisfied. By contrast, Figure~\ref{fig:physical_cap_failure} shows a $P_{16}$ line network for which the baseline routing becomes infeasible. The baseline routing satisfies the nominal physical capacity constraints, but handling causes line~8 to spend $1333$ of its $1440$ available time units handling parcels, reducing the effective capacity of arc $(1,6)$ to $146.3$, below the $536$ units routed on that arc. Lastly, for comparison, a line with one vehicle, $T=60$, and a total handling workload of $80$ time units is temporally infeasible, regardless of vehicle capacity.

\begin{figure}[H]
    \centering
    \begin{subfigure}[t]{0.48\linewidth}
        \centering
        \includegraphics[width=0.7\linewidth]{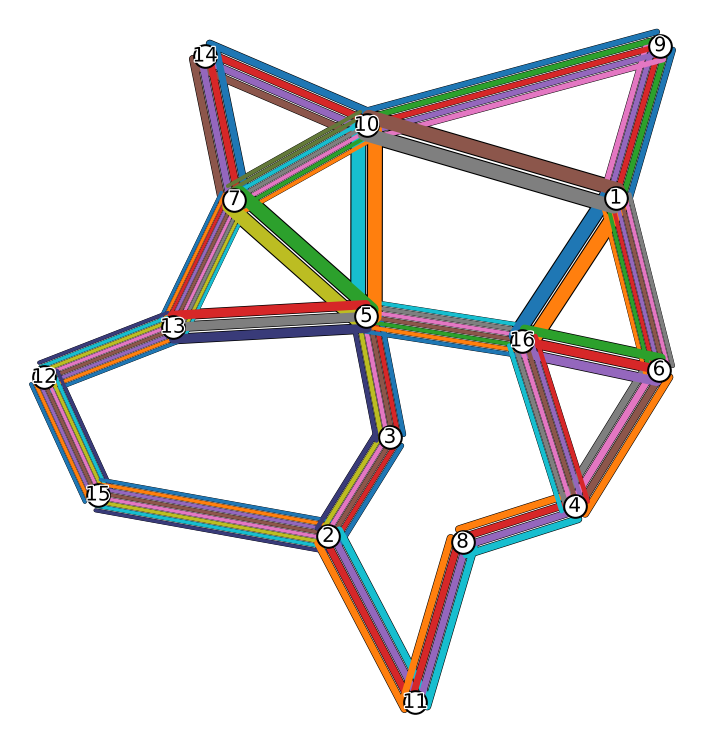}
        \caption{Feasible baseline}
        \label{fig:baseline_feasible}
    \end{subfigure}
    \hfill
    \begin{subfigure}[t]{0.48\linewidth}
        \centering
        \includegraphics[width=0.7\linewidth]{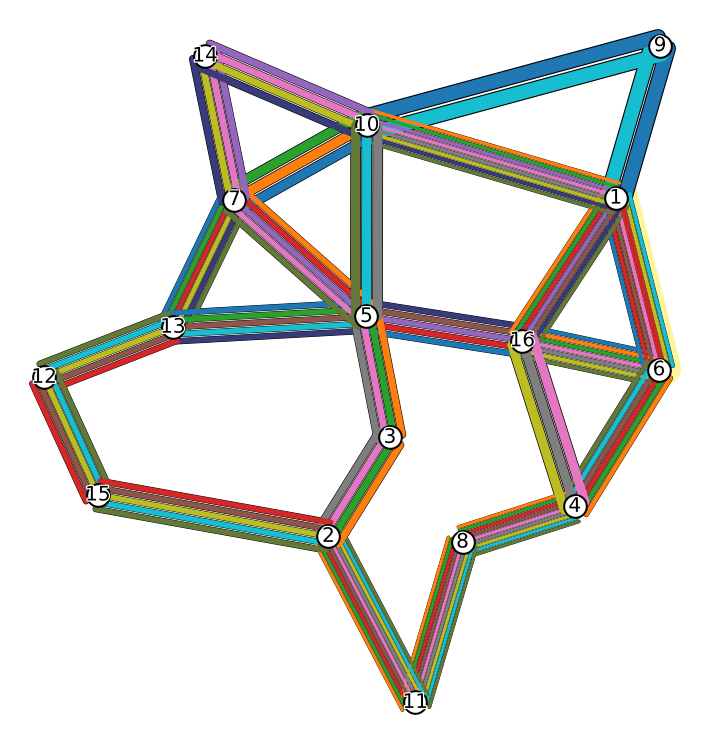}
        \caption{Infeasible baseline due to arc capacity overload.}
        \label{fig:physical_cap_failure}
    \end{subfigure}
    \caption{Two different line concepts admitting feasible and infeasible baseline routing solutions once handling is accounted for.}
    \label{fig:baseline_infeasibility_examples}
\end{figure}

\subsection*{The strengthened SDP relaxation}
\label{app:complete_endpoint}
As outlined in Section~5, for an endpoint $\gamma_b$, we begin with the following SDP relaxation of the localized endpoint problem:
\begin{subequations}
\label{model:appendix_base_sdp}
\begin{align}
    \min\quad
    &\sum_{l\in\mathcal L} C_l\langle M_l,Z_l\rangle
      \label{eq:appendix_base_obj}\\
    \text{s.t.}\quad
    &x\in\Xset_b^{F_c},\quad z_l=E_lx,
      \quad y_l=Ta_l-e_l^\top z_l,
      &&l\in\mathcal L,
      \label{eq:appendix_base_link}\\
    &Ta_ls_l-e_l^\top w_l=1,
      \quad Ta_lw_l-Z_le_l=z_l,
      &&l\in\mathcal L,
      \label{eq:appendix_base_identities}\\
    &\begin{pmatrix}Z_l&w_l\\w_l^\top&s_l\end{pmatrix}\succeq0,
      &&l\in\mathcal L.
      \label{eq:appendix_base_psd}
\end{align}
\end{subequations}
The flow vector $x$ describes the routing, while $z_l,y_l,s_l,w_l$ and the symmetric matrix $Z_l$ are auxiliary variables associated with each line. Here, $z_l=(\boldsymbol h_l,\boldsymbol X_l)=E_lx$ collects the handling workloads and incoming onboard flows at the hubs of line $l$, and $e_l$ selects the handling coordinates so that $e_l^\top z_l$ is their sum. As in Section~5, we can express the bilinear numerator $\boldsymbol h_l^\top\boldsymbol X_l$ in terms of $z_l$ using the following matrix $M_l$:
\[
M_l=\frac12\begin{pmatrix}0&I\\I&0\end{pmatrix},
\qquad
z_l^\top M_lz_l=\boldsymbol h_l^\top\boldsymbol X_l,
\]
where $I$ is the identity matrix of dimension $|\mathcal N(l)|$. At an exact lift, $s_l=1/y_l$, $w_l=z_l/y_l$, and $Z_l=z_lz_l^\top/y_l$, so the lifted objective equals $r(x)$. The base SDP retains the positive semidefinite constraint and the linear identities implied by these relations, but drops the rank-one requirement. Consequently, the lifted variables need not satisfy these product identities, allowing the objective to underestimate the residual cost.

To limit deviations from these identities, we use McCormick envelopes. For instance, consider $w_{lj}=z_{lj}/y_l$, or equivalently $z_{lj}=y_lw_{lj}$. While this identity need not hold in the relaxed SDP, we can add the following constraints, where an underline and overline denote a variable's lower and upper bounds, respectively:

\begin{equation}
  \label{eq:appendix_fractional_mc}
  \begin{aligned}
  z_{lj}&\ge\underline y_l w_{lj}+\underline w_{lj}y_l
                   -\underline y_l\underline w_{lj},\\
  z_{lj}& \ge\overline y_l w_{lj}+\overline w_{lj}y_l
                   -\overline y_l\overline w_{lj},\\
  z_{lj}& \le\overline y_l w_{lj}+\underline w_{lj}y_l
                   -\overline y_l\underline w_{lj},\\
  z_{lj}& \le\underline y_l w_{lj}+\overline w_{lj}y_l
                   -\underline y_l\overline w_{lj}.
  \end{aligned}
  \end{equation}

We impose these inequalities for all $l\in\mathcal L$ and every coordinate $j$ of $z_l$. We similarly impose McCormick envelopes for $Z_{ljk}=z_{lj}w_{lk}$ for every ordered pair $(j,k)$. Since $Z_l$ is symmetric, this also imposes the envelopes for $Z_{ljk}=z_{lk}w_{lj}$. Naturally, these McCormick inequalities become stronger as the variable bounds tighten. We therefore perform optimization-based bound tightening (OBBT) on $y_l$ and the coordinates of $z_l$ by minimizing and maximizing them over convex relaxations that include the endpoint budget and incumbent cutoff. A strong incumbent helps restrict the region over which these bounds are computed. We impose the resulting bounds and retain the bounds on boarding and unloading aggregates from preliminary tightening. Since $z_l$ is nonnegative and the denominator bounds are strictly positive, we also obtain bounds on the fractional variables:
\begin{equation}
\label{eq:appendix_fractional_bounds}
\begin{aligned}
&0<\underline y_l\le y_l\le\overline y_l,\qquad
  \underline z_{lj}\le z_{lj}\le\overline z_{lj},\\
&\frac1{\overline y_l}\le s_l\le\frac1{\underline y_l},\qquad
  \underline w_{lj}:=\frac{\underline z_{lj}}{\overline y_l}
  \le w_{lj}\le\frac{\overline z_{lj}}{\underline y_l}=:\overline w_{lj}.
\end{aligned}
\end{equation}

In addition, Section~4 gives a lower bound for $r_l(x)$ based on the unloading flow $U_i^l=\sum_k u_i^{kl}$. We retain the restriction $U_j^l\le X_{ij}^l$ on each incoming line arc $(i,j)$, which preserves an optimum as discussed in Section~4, and add
\begin{equation}
\label{eq:appendix_residual_bounds}
\langle M_l,Z_l\rangle\ge0,\qquad
\langle M_l,Z_l\rangle\ge
\frac{\beta}{y_l}\sum_{i\in\mathcal N(l)}(U_i^l)^2,
\qquad l\in\mathcal L.
\end{equation}
This ensures that the lifted residual is bounded below by the convex underestimator used in the SOCP relaxation. After adding the cuts described in Section~5, we get the following SDP that we solve in practice.
\begin{equation}
\label{eq:complete_endpoint}
\begin{aligned}
\ell_b=\min_{x,z,y,s,w,Z}\quad
 &\sum_{l\in\mathcal L}C_l\langle M_l,Z_l\rangle\\
\mathrm{s.t.}\quad
 &\text{base constraints \eqref{eq:appendix_base_link}--\eqref{eq:appendix_base_psd}},\\
 &\text{variable bounds described above, including \eqref{eq:appendix_fractional_bounds}},\\
 &\text{envelopes \eqref{eq:appendix_fractional_mc} and the analogous envelopes for }Z_{ljk},\\
 &U_j^l\le X_{ij}^l\quad((i,j)\in\mathcal A(l),\ l\in\mathcal L),\\
 &\text{residual inequalities \eqref{eq:appendix_residual_bounds}},\\
 &\ell_{lq}^{(b)}\le q^\top z_l\le u_{lq}^{(b)},\\
 &\min\left\{\frac{\ell_{lq}^{(b)}}{\underline y_l},
             \frac{\ell_{lq}^{(b)}}{\overline y_l}\right\}
   \le q^\top w_l\le
   \max\left\{\frac{u_{lq}^{(b)}}{\underline y_l},
             \frac{u_{lq}^{(b)}}{\overline y_l}\right\},\\
 &\text{McCormick envelopes for }q^\top z_l=y_l(q^\top w_l),\\
 &q^\top Z_lq\le
   (\ell_{lq}^{(b)}+u_{lq}^{(b)})q^\top w_l
   -\ell_{lq}^{(b)}u_{lq}^{(b)}s_l,\\
 &f_{\mathrm{lin}}(x)+
   \frac{T}{2(\alpha+\beta)}\sum_l C_l(Ta_ls_l-1)\le g_c+\gamma_b,\\
 &f_{\mathrm{lin}}(x)+
   \frac{T}{2(\alpha+\beta)}\sum_l C_l(Ta_ls_l-1)
   +\sum_l C_l\langle M_l,Z_l\rangle\le F_c.
\end{aligned}
\end{equation}
The constraints apply to every relevant line, coordinate, and selected direction. At an exact lift, the last two inequalities are the endpoint budget $g(x)\le g_c+\gamma_b$ and incumbent cutoff $F(x)\le F_c$.

For all experiments, we begin with \(B=12\) intervals and focus refinement on the weakest stair, its neighbors, and other intervals close to the minimum. In each adaptive round, we either split the selected intervals into four or strengthen their endpoint SDPs by generating up to three additional directions for each of up to three selected lines. We apply OBBT to the new projections before adding their corresponding envelopes and cuts. The procedure stops after at most six adaptive rounds, or earlier if a strengthening pass produces negligible improvement in the bound. We solve the SDPs using MOSEK and the OBBT subproblems using Gurobi, with parallel OBBT where applicable.

\clearpage

\section{Additional Instances Information}\label{app:instances}
\begin{figure}[H]
\centering
  \begin{tikzpicture}[x=1cm,y=1cm,baseline=(current bounding box.center)]
\begin{scope}[shift={(1,0)}]
    \CSVGraphPanel{graphics/PTNs_latex/Grid_2x2/nodes.csv}{graphics/PTNs_latex/Grid_2x2/edges.csv}{$G_{2,2}$}
  \end{scope}
  \begin{scope}[shift={(3.2,0)}]
    \CSVGraphPanel{graphics/PTNs_latex/Grid_2x3/nodes.csv}{graphics/PTNs_latex/Grid_2x3/edges.csv}{$G_{2,3}$}
  \end{scope}
  \begin{scope}[shift={(5.4,0)}]
    \CSVGraphPanel{graphics/PTNs_latex/Grid_3x3/nodes.csv}{graphics/PTNs_latex/Grid_3x3/edges.csv}{$G_{3,3}$}
  \end{scope}
  \begin{scope}[shift={(7.6,0)}]
    \CSVGraphPanel{graphics/PTNs_latex/Grid_3x4/nodes.csv}{graphics/PTNs_latex/Grid_3x4/edges.csv}{$G_{3,4}$}
  \end{scope}
  \begin{scope}[shift={(9.8,0)}]
    \CSVGraphPanel{graphics/PTNs_latex/Grid_4x4/nodes.csv}{graphics/PTNs_latex/Grid_4x4/edges.csv}{$G_{4,4}$}
  \end{scope}
  \begin{scope}[shift={(1,-2.6)}]
    \CSVGraphPanel{graphics/PTNs_latex/Complete_2x2/nodes.csv}{graphics/PTNs_latex/Complete_2x2/edges.csv}{$K_{4}$}
  \end{scope}
  \begin{scope}[shift={(3.2,-2.6)}]
    \CSVGraphPanel{graphics/PTNs_latex/Complete_2x3/nodes.csv}{graphics/PTNs_latex/Complete_2x3/edges.csv}{$K_{6}$}
  \end{scope}
  \begin{scope}[shift={(5.4,-2.6)}]
    \CSVGraphPanel{graphics/PTNs_latex/Complete_3x3/nodes.csv}{graphics/PTNs_latex/Complete_3x3/edges.csv}{$K_{9}$}
  \end{scope}
  \begin{scope}[shift={(7.6,-2.6)}]
    \CSVGraphPanel{graphics/PTNs_latex/Complete_3x4/nodes.csv}{graphics/PTNs_latex/Complete_3x4/edges.csv}{$K_{12}$}
  \end{scope}
  \begin{scope}[shift={(9.8,-2.6)}]
    \CSVGraphPanel{graphics/PTNs_latex/Complete_4x4/nodes.csv}{graphics/PTNs_latex/Complete_4x4/edges_full.csv}{$K_{16}$}
  \end{scope}
  \begin{scope}[shift={(1,-5.2)}]
    \CSVGraphPanel{graphics/PTNs_latex/Planar_4/nodes.csv}{graphics/PTNs_latex/Planar_4/edges.csv}{$P_{4}$}
  \end{scope}
  \begin{scope}[shift={(3.2,-5.2)}]
    \CSVGraphPanel{graphics/PTNs_latex/Planar_6/nodes.csv}{graphics/PTNs_latex/Planar_6/edges.csv}{$P_{6}$}
  \end{scope}
  \begin{scope}[shift={(5.4,-5.2)}]
    \CSVGraphPanel{graphics/PTNs_latex/Planar_9/nodes.csv}{graphics/PTNs_latex/Planar_9/edges.csv}{$P_{9}$}
  \end{scope}
  \begin{scope}[shift={(7.6,-5.2)}]
    \CSVGraphPanel{graphics/PTNs_latex/Planar_12/nodes.csv}{graphics/PTNs_latex/Planar_12/edges.csv}{$P_{12}$}
  \end{scope}
  \begin{scope}[shift={(9.8,-5.2)}]
    \CSVGraphPanel{graphics/PTNs_latex/Planar_16/nodes.csv}{graphics/PTNs_latex/Planar_16/edges.csv}{$P_{16}$}
   \end{scope}
  \end{tikzpicture}%

\caption{PTNs of the 15 synthetic instances. \label{fig:grids}}
\end{figure}
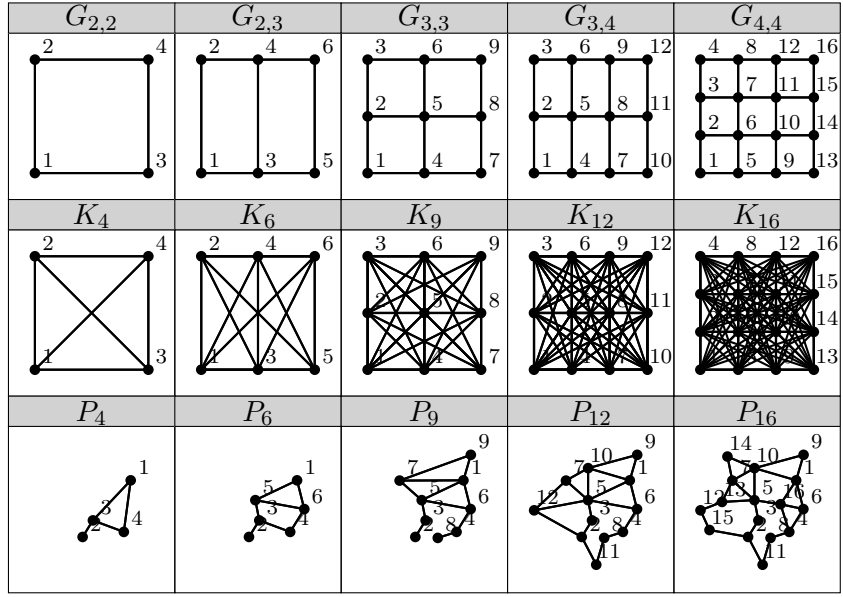

\begin{figure}[H]
\centering
\begin{subfigure}[t]{0.48\textwidth}
  \centering
  \ptnimg{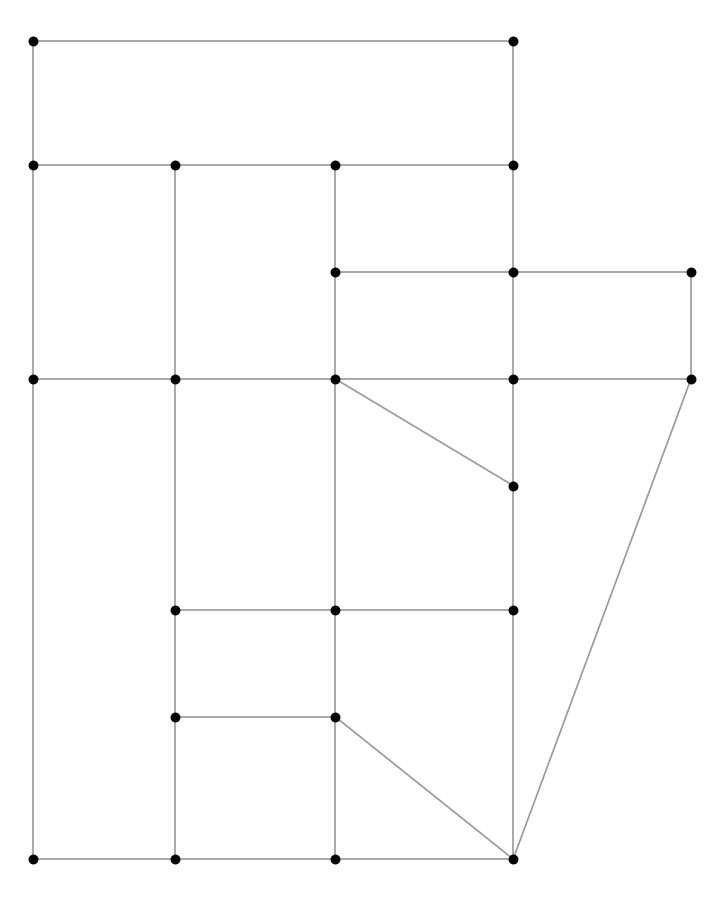}
  \caption{Sioux Falls (SF)}
\end{subfigure}\hfill
\begin{subfigure}[t]{0.48\textwidth}
  \centering
  \ptnimg{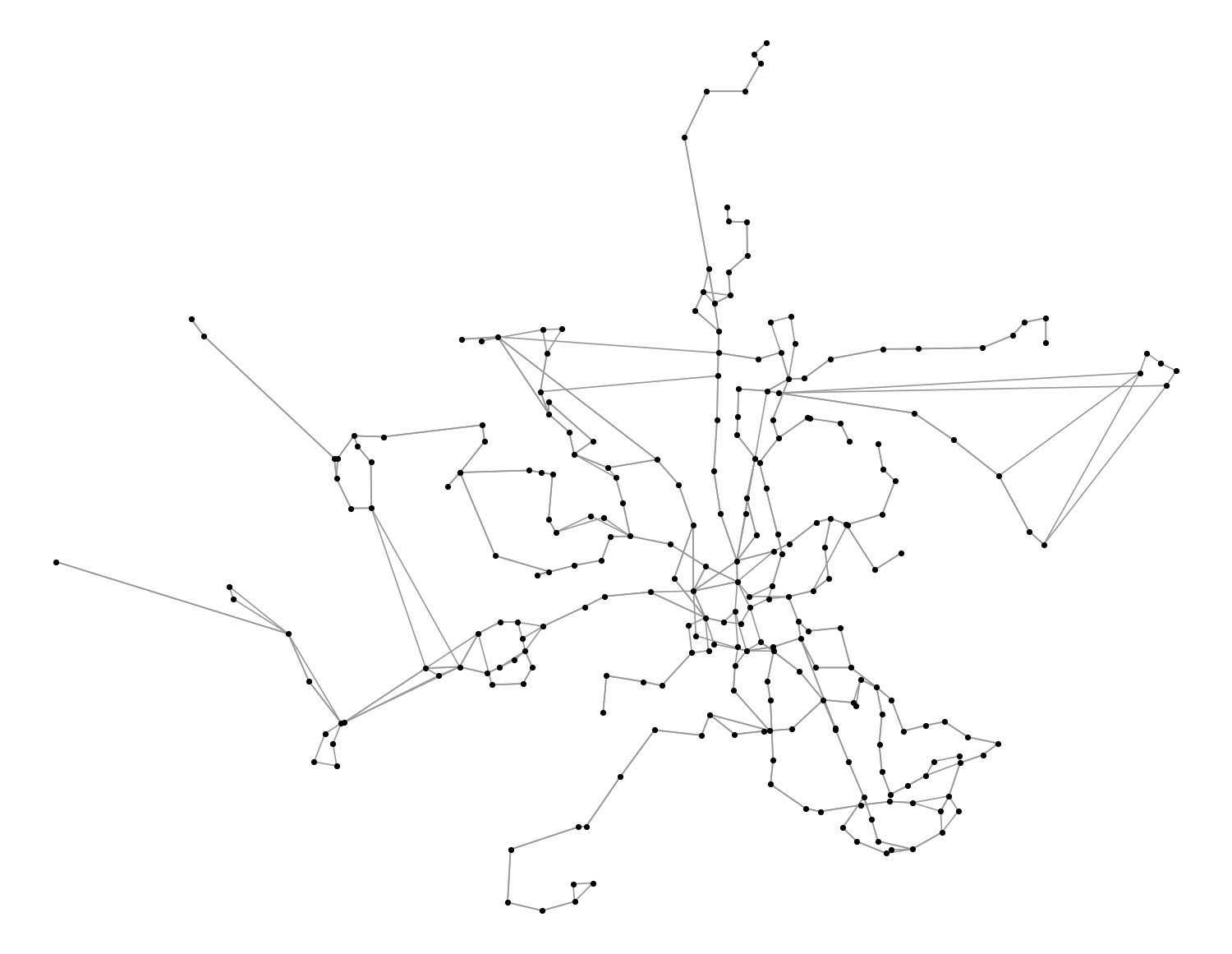}
  \caption{Göttingen (G\"oVB)}
\end{subfigure}
\caption{Real-world PTN instances used in the experiments.}
\label{fig:real_world_PTNs}
\end{figure}